\documentclass[11pt]{article}
\usepackage{amsfonts,latexsym,rawfonts,amsmath,amssymb,amsthm, a4wide,  times, xcolor}
\usepackage{graphicx}
\numberwithin{equation}{section}
\usepackage{hyperref}
\usepackage[titletoc,title]{appendix}
\usepackage[title]{appendix}

\numberwithin{equation}{section}

\newcommand{\pd}[2]{\frac {\partial #1}{\partial #2}}

\newcommand{\al}{\alpha}
\newcommand{\bb}{\beta}
\newcommand{\la}{\lambda}
\newcommand{\La}{\Lambda}
\newcommand{\oo}{\omega}
\newcommand{\Om}{\Omega}

\newcommand{\dd}{\delta}
\newcommand{\Na}{\nabla}

\def\ga{\gamma}

\newcommand{\ee}{\epsilon}

\newcommand{\si}{\sigma}

\newcommand{\Te}{\Theta}
\newcommand{\te}{\theta}

\newcommand{\beq}{\begin{equation}}
\newcommand{\eeq}{\end{equation}}
\newcommand{\beqs}{\begin{eqnarray*}}
\newcommand{\eeqs}{\end{eqnarray*}}
\newcommand{\beqn}{\begin{eqnarray}}
\newcommand{\eeqn}{\end{eqnarray}}
\newcommand{\beqa}{\begin{array}}
\newcommand{\eeqa}{\end{array}}

\def\td{\tilde}

\def\RR{{\mathbb R}}
\def\NN{{\mathbb N}}

\def\ri{\rightarrow}

\def\un{\underline}

\def\no{{\nonumber}}
\def\si{\sigma}
\def\pbp{\sqrt{-1}\partial\bar\partial}
\def\tr{{\rm tr}}

\def\vol{{\rm vol}}

\def\cA{{\mathcal A}}

\def\cH{{\mathcal H}}

\def\cP{{\mathcal P}}

\def\ka{{\kappa}}

\newtheorem{prop}{Proposition}[section]
\newtheorem{theo}[prop]{Theorem}
\newtheorem{lem}[prop]{Lemma}
\newtheorem{claim}[prop]{Claim}
\newtheorem{cor}[prop]{Corollary}

\newtheorem{defi}[prop]{Definition}

 \makeatletter
 \def\ExtendSymbol#1#2#3#4#5{\ext@arrow 0099{\arrowfill@#1#2#3}{#4}{#5}}
 
 \makeatother

 \makeatletter
 \def\ExtendSymbol#1#2#3#4#5{\ext@arrow 0099{\arrowfill@#1#2#3}{#4}{#5}}
 
 \makeatother

\title{Existence of twisted Calabi flow  and  deformation from the $J$-flow to Calabi flow }

\author{Jie, He \footnote{Supported by Mathematics Tianyuan Fund of NSFC grant No. 12426674.} \quad and \quad Haozhao Li \footnote{Supported by NSFC grant No. 12426669, No. 12471058,  the CAS Project for Young Scientists
in Basic Research (YSBR-001), and the Fundamental Research Funds
for the Central Universities.}}

\begin{document}
\bibliographystyle{plain}


\maketitle

\begin{abstract}
In this paper, we study a family of twisted Calabi flows connecting the $J$-flow and Calabi flow on a compact K\"ahler manifold with a constant scalar curvature (cscK) metric. We show that for any initial data the twisted Calabi flow near the $J$-flow has long time existence and converges smoothly to the  cscK metric. Moreover, we show that if a twisted Calabi flow has long time existence and converges, then the nearby twisted Calabi flow with the same initial data also has long time existence and converges. These results imply the openness of the continuity method to study Chen's long time existence conjecture on (twisted) Calabi flow on cscK manifolds.



\end{abstract}

\tableofcontents

\section{Introduction}

Let $(M^n, g)$ be a compact K\"ahler manifold of complex dimension $n$. A
family of K\"ahler metrics $\oo_{\varphi(t)}(t\in [0, T])$ in the same K\"ahler class $[\oo_g]$ is called a solution of twisted
Calabi flow, if the K\"ahler potential $\varphi_s(t)$ satisfies the equation
\beq
\pd {\varphi_s(t)}t=s(R(\oo_{\varphi_s(t)})-\un R)+(1-s)(n-\tr_{\varphi_s}\oo_g), \label{eq:000}
\eeq where $R(\oo_{\varphi_s(t)})$ denotes the scalar curvature of the metric $\oo_{\varphi_s(t)}$, $\un R$ denotes the average of the scalar curvature and $\tr_{\varphi_s}\oo_g=g_{\varphi_s}^{i\bar j}g_{i\bar j}.$  It was conjectured by Chen in \cite{[Chen2]} that the twisted Calabi flow with any initial K\"ahler potential exists for all time. It is natural to study Chen's conjecture by the continuity method. In this paper, we  show the long time existence and convergence of twisted Calabi flow near the $J$-flow,
 which can be viewed as the first step toward Chen's conjecture via the continuity method.

The twisted Calabi flow is an interpolation between Calabi flow and the $J$-flow.
The Calabi flow was introduced by Calabi in \cite{[Cal1]} as a decreasing flow of  Calabi energy, and it is expected to be
an
effective tool to find constant scalar curvature metrics in a K\"ahler class. In
Riemann surfaces, the long time behavior and convergence is
completely solved by Chrusciel \cite{[Chru]},    Chen
\cite{[Chen]} and   Struwe \cite{[Stru]} independently by different methods. Pook \cite{[Pook]} studied the twisted Calabi flow on Riemann surfaces.
In a series of papers \cite{[ChenHe1]}\cite{[ChenHe2]}\cite{[ChenHe3]}\cite{[He4]}\cite{[He5]}, Chen and He studied the  existence
and convergence of Calabi flow under various curvature conditions on K\"ahler manifolds. There are more interesting results on the existence and convergence of Calabi flow, see  Tosatti-Weinkove  \cite{[TW]}, Szekelyhidi \cite{[Sz]}, Streets
\cite{[St1]}\cite{[St2]} , Berman-Darvas-Lu \cite{[BDL]}\cite{[BDL2]}, Huang \cite{[Huang]} etc.
A breakthrough was made by Chen-Cheng in \cite{[CC1]} and they showed that the Calabi flow always exists as long as the scalar curvature is bounded.
 In \cite{[LZZ]} Li-Zhang-Zheng showed the long time existence under the bounded $L^p(p>n)$  scalar curvature condition.

 The $J$-flow introduced by Donaldson  \cite{[Don]} in the setting of moment maps and by Chen  \cite{[Chen-Mabuchi]} as the gradient flow of the $J$ functionals. Chen \cite{[Chen0]} showed that the $J$-flow has long time existence for any initial data. Weinkove \cite{[Wenk1]}\cite{[Wenk2]} and Song-Weinkove \cite{[SW]} studied the convergence of the $J$-flow and give a necessary and sufficient condition  for the convergence of $J$-flow in higher dimensions. For more results on the $J$-flow  and  the $J$-equation, see Fang-Lai-Song-Weinkove \cite{[FLSW]},  Lejmi-Szekelyhidi \cite{[LeSz]} and G. Chen \cite{[ChenG]} etc.

\subsection{The statements of main results}

In \cite{[Chen2]} X. X. Chen proposed the conjecture that the twisted Calabi flow (\ref{eq:000}) exists globally for any smooth initial K\"ahler potential.
To study Chen's conjecture via the continuity method, for any $\psi_0\in \cH(\oo_g)$ we define
\beqs
I_{\psi_0}&=&\{s'\in [0, 1]\;|\;\hbox{The solution of} \;(\ref{eq:000}) \; \hbox{with the initial data}\; \psi_0\; \hbox{exists for} \\&&  \hbox{ all time and converges smoothly to a twisted cscK metric for any}\;s\in [0, s'] \}.
\eeqs
To show Chen's conjecture, it suffices to show that $I_{\psi_0}=[0, 1]$ for any $\psi_0\in \cH(\oo_g)$. For $s=0$,  the equation (\ref{eq:000}) is the $J$-flow, which has long time existence and convergence by Song-Weinkove \cite{[SW]} since the $J$-equation $\tr_{\varphi}\oo_g=n$ has a solution $\varphi=0$.  We would like to ask if the twisted Calabi flow (\ref{eq:000}) has long time existence and convergence when $s$ is small. The first main result of this paper answers this question.

\smallskip
\begin{theo}\label{theo:main1}
Let $(M, g)$ be a compact K\"ahler manifold with a cscK metric
$\oo_g.$  For any $\psi_0\in \cH(\oo_g)$, there exists $s_0\in (0, 1]$ such that for any $s\in (0, s_0)$ the twisted Calabi flow (\ref{eq:000}) with the initial K\"ahler potential $\psi_0$ exists for all time and converges exponentially to the cscK metric $\oo_g$.

\end{theo}

Theorem \ref{theo:main1} shows that there exists $s_0>0$ such that $[0, s_0]\subset I_{\psi_0}.$ Since the cscK metric $\oo_g$ is also a twisted cscK metric,  the limit twisted cscK metric in Theorem \ref{theo:main1} is exactly  $\oo_g$ by the uniqueness of twisted cscK metrics(cf. Berman-Darvas-Lu \cite{[BDL]}).
The next result shows that if the twisted Calabi flow (\ref{eq:000}) has long time existence and converges at $s=s_0>0$, so does $s$ when $|s-s_0|$ is small.

\begin{theo}\label{theo:main3}
Let $(M, g)$ be a compact K\"ahler manifold with a cscK metric
$\oo_g.$
 For any $\psi_0\in \cH(\oo_g)$, if the twisted Calabi flow (\ref{eq:000}) at $s=s_0\in (0, 1)$ with the initial data $\psi_0$ exists for all time and converges smoothly to a twisted cscK metric, there exists $\dd>0$ such that the  flow (\ref{eq:000}) at any $s\in (s_0-\dd, s_0+\dd)\cap (0, 1)$ with the initial data $\psi_0$ exists for all time and converges exponentially to the cscK metric $\oo_g$.
\end{theo}

Theorem \ref{theo:main3} shows that   if $s_0>0$ satisfies $s_0\in I_{\psi_0}$, then there exists $\dd>0$ such that $(s_0-\dd, s_0+\dd)\subset I_{\psi_0}$. Combining Theorem \ref{theo:main1} with Theorem \ref{theo:main3}, we have

\begin{cor}On  a compact K\"ahler manifold $(M, g)$ with a cscK metric $g,$ for any $\psi_0\in \cH(\oo_g)$ the set $I_{\psi_0}$ is  open in $[0, 1]$.
\end{cor}

  The closeness of $I_{\psi_0}$ will be difficult, and it will imply Chen's conjecture on the cscK case.

\subsection{Outline of the proofs}

Theorem \ref{theo:main1} is motivated by the work of Hashimoto \cite{[Hashi]},  Zeng \cite{[Zeng]} on twisted cscK metrics and Fine \cite{[Fine]} on cscK metrics.  In \cite{[Hashi]}, Hashimoto  proved the existence of twisted cscK metrics
\beq
s(R(\varphi)-\un R)+(1-s)(n-\tr_{\varphi}\chi)=0
\eeq for  small $s$, and in \cite{[Zeng]} Zeng independently proved the case of $\chi=\oo_g$.  The methods of these works are similar and rely on the construction of approximate solutions of the linearized equation of twisted cscK metrics. Motivated by this method, we use the parabolic version to prove Theorem \ref{theo:main1}. We outline the proof as follows.

\emph{Step 1.} Given any K\"ahler potential $\psi_0$ and any $T>0$, there exists $s_0>0$ such that for $s\in (0, s_0)$ the solution of  twisted Calabi flow (\ref{eq:000}) with the initial data $\psi_0$ exists on $M\times [0, T]$. To prove this, we consider the twisted Calabi flow operator
\beq
L_s(\varphi):=\pd {\varphi}t-s(R(\varphi)-\un R)-(1-s)(n-\tr_{\varphi}\oo_g).
\eeq
Let $\varphi_0$ be the solution of the $J$-flow with the initial K\"ahler potential $\psi_0$, i.e. $L_0(\varphi_0(x, t))=0.$ Then $\|L_s(\varphi_0)\|_{C^{0, 0, \ga}(M\times [0, T], g)}\leq Cs. $
      We modify $\varphi_0$ to be $\td \varphi$ such that $\td \varphi$ satisfy
      \beq
      \|L_s(\td \varphi)\|_{C^{0, 0, \ga}(M\times [0, T], g)}\leq Cs^N
      \eeq for any given integer $N>0.$
 The linearized operator
$
DL_s|_{\td \varphi}:  C_0^{4, 1, \ga}(M\times [0, T], g)\ri C^{0, 0, \ga}(M\times [0, T], g)
$
is injective and surjective. Moreover, we have an upper bound of $ \|(DL_s|_{\td \varphi})^{-1}\|$.  Next, we show that the map $\Psi_f$  defined by
  \beq\Psi_f :  \varphi \ri \varphi+(DL_s|_{\td \varphi})^{-1}(f-L_s(\varphi))
\eeq is a contraction, which implies  the existence of the solution of $L_s(\varphi)=0$ by the fixed point theory.

\emph{Step 2.} There exists $\dd_0, s_0>0$ such that if the initial given K\"ahler potential $\psi_0$ satisfies $\|\psi_0\|_{C^4}\leq \dd_0$, then for any $s\in (0, s_0)$ the solution of  twisted Calabi flow (\ref{eq:000}) with the initial data $\psi_0$ exists on $M\times [0, \infty)$. Compared with the proof in Step 1, we have the following difficulties:
\begin{enumerate}
  \item[$(a).$] The H\"older norms of  $\td \varphi(x, t)$ constructed in Step 1 on $M\times [0, T)$ is infinity. To overcome this difficulty, we assume that the initial K\"ahler potential is sufficiently small such that we can show that $\varphi_0(x, t)$ and $\td \varphi_0(x, t)$ decay exponentially.
  \item[$(b).$] It is difficult to prove that the operators $L_s$ and $DL_s$ are invertible between H\"older spaces with functions decaying exponentially. To overcome this difficulty, we modify the operator by
      \beq
      \td L_s(\varphi)=L_s(\varphi)-c_s(\varphi, \td \varphi_0),
      \eeq where $c_s(\varphi, \td \varphi_0)$ is a  constant in $t$. We can choose $c_s(\varphi, \td \varphi_0)$ such that $ \td L_s(\varphi) $ and  $ D\td L_s(\varphi) $ map  exponentially decaying functions to exponentially decaying functions.
\end{enumerate}
The rest proof of Step 2 is similar to that in Step 1. We show that a modified operator $\td \Psi_f$ is a contraction and the solution of twisted Calabi flow exists on $M\times [0, \infty)$ by the fixed point theory.

 \emph{Step 3.} Theorem \ref{theo:main1} follows from Step 1 and Step 2. In fact, for given $\psi_0$ the $J$-flow with the initial data $\psi_0$ exists for all time and converges smoothly to $0$. By Step 2 there exists $T>0$ such that the norm of $\varphi_0(x, t)$ is bounded by $\dd_0/2$. On the other hand, by Step 1 for small $s$ the solution $\varphi_s(x, t)$ of twisted Calabi flow (\ref{eq:000}) exists and is sufficiently close to $\varphi_0(x, t)$ on $M\times [0, T]$. Thus, for small $s$ the norm of $\varphi_s(x, T)$ is bounded by $\dd_0$. By Step 2 the twisted Calabi flow with the initial data $\varphi_s(x, T)$ exists for all time $t\in [T, \infty)$ and converges exponentially to the cscK metric $\oo_g$.

 The proof of Theorem \ref{theo:main3} is similar to that of Theorem \ref{theo:main1}. Assume that the solution $\varphi_{s_0}(x, t)$ of twisted Calabi flow for $s=s_0\in (0, 1)$ exists for all time and converges. Then we replace the
 background K\"ahler potential $\td \varphi(x, t)$ by $\varphi_{s_0}(x, t)$ and the rest of the proof is similar to that of Theorem \ref{theo:main1}.\\

The organization of the paper is as follows. In Section 2 we introduce the twisted Calabi flow and its linearized equation. In Section 3 we show that the twisted Calabi flow near the $J$-flow exists for a given time interval. In Section 4 we show that there exists a uniform neighborhood near the cscK metric such that for any small $s$ the twisted Calabi flow has long time solution and converges. In Section 5, we combine the results in Section 3 and Section 4 to show Theorem \ref{theo:main1}. In Section 6 we modified the estimates in Section 3 and Section 4 to show Theorem \ref{theo:main3}. In the appendix, we recall the Schauder estimates and existence results of linear fourth-order and second-order parabolic equations from Metsch \cite{[Me]}, He-Zeng \cite{[HZ]} and Lieberman \cite{[Lieb]},  and show some necessary interpolation inequalities and existence results.

\section{Preliminaries}\label{sec2}
Let $(M, g)$ be a compact K\"ahler manifold with a given K\"ahler metric
$\oo_g$. We define  the space of K\"ahler potentials $\cH(\oo_g) $ by
\beq
\cH(\oo_g)=\{\varphi\in C^4(M)\;|\; \oo_g+\pbp\varphi>0\}.\no
\eeq
Let $\psi_0\in \cH(\oo_g)$.  Consider  a family of twisted Calabi flow with $s\in [0, 1]$
\beq
\pd {\varphi_s}t=s(R(\varphi_s)-\un R)+(1-s)(n-\tr_{\varphi_s}\oo_g),\quad \varphi_s(x, 0)=\psi_0. \label{eq:A0000a}
\eeq
We define the operator   $L_s$ by
\beq
L_s(\varphi)=\pd {\varphi}t-s(R(\varphi)-\un R)-(1-s)(n-\tr_{\varphi}\oo_g). \no
\eeq
Then the equation (\ref{eq:A0000a}) can be written as
\beq
L_s(\varphi_s)=0,\quad  \varphi_s(0)=\psi_0. \no
\eeq
When $s=0$, the equation (\ref{eq:A0000a}) becomes the $J$-flow:
\beq
\pd {\varphi_0}t=n-\tr_{\varphi_0}\oo_g,\quad \varphi_0(x, 0)=\psi_0. \label{eq:A0003}
\eeq

The equation (\ref{eq:A0003}) admits a solution $\varphi_0(x, t)$ for all time $t>0$ and the metrics $\oo_{\varphi_0(x, t)}$ converge smoothly to $\oo_g$ by Weinkove \cite{[Wenk2]}.  We can normalize $\varphi_0(x, t)$ such that $\lim_{t\ri +\infty}\varphi_0(x, t)=0.$ Assume that $\varphi(x, t)$ is a family of K\"ahler potentials. By direct calculation, the derivative of $L_s$ at $\varphi(x, t)$ can be written as
\beq
DL_s|_{\varphi}(u)=\pd {u}t+s(\Delta^2_{ \varphi}u+Ric_{\varphi, i\bar j
}u_{j\bar i})-(1-s)g_{\varphi}^{i\bar l}g_{\varphi}^{k\bar j}g_{i\bar j}u_{k\bar l}.\no
\eeq Therefore, the linearized equation $DL_s|_{\varphi}(u)=f$ can be written as
\beq
\frac 1s\pd {u}t+\Delta^2_{ \varphi}u=s^{-1}f+\frac {1-s}sg_{\varphi}^{i\bar l}g_{\varphi}^{k\bar j}g_{i\bar j}u_{k\bar l}-Ric_{\varphi, i\bar j
}u_{j\bar i},\quad \forall \;(x, t)\in M\times [0, T].\label{eq:A0004}
\eeq
To use the parabolic theory of fourth order, we define
\beqn \tau&:=&st,\quad w(x, \tau):=u(x, t),\quad \td f(x, \tau):=f(x, t),\no\\
a_{l\bar k}&:=&-sRic_{\varphi, l\bar k}+(1-s)g_{\varphi}^{i\bar l}g_{\varphi}^{k\bar j}g_{i\bar j}.\label{eq:A0006}
\eeqn
Thus, (\ref{eq:A0004}) is equivalent to the equality
\beq
\pd {w}{\tau}+\Delta^2_{\varphi}w=h(x, \tau):=s^{-1}\td f+s^{-1}a_{l\bar k}w_{k\bar l},\quad (x, \tau)\in M\times [0, sT]. \label{eq:A0007a}
\eeq Observe that $a_{l\bar k}$ can be written as
\beq
a_{l\bar k}=(1-s)g_{\varphi}^{k\bar l}-(1-s)g_{\varphi}^{i\bar l}g_{\varphi}^{k\bar j}\varphi_{i\bar j}
-sRic_{\varphi, l\bar k}, \label{eq:A0007b}
\eeq

To study the equation (\ref{eq:A0007a}), we will use the  parabolic theory of the fourth-order parabolic equation
\beq
\pd {u}{t}+\Delta^2_{g}u=f.\label{eq:A00b}
\eeq
The  Schauder estimates of higher-order parabolic equation are known in literatures;  see,   for instance, Q. Han \cite{[Han]}, Boccia \cite{[Bo]},  Dong-Zhang \cite{[DZ]}. In \cite{[HZ]} He-Zeng show the existence and Schauder estimates of (\ref{eq:A00b}) in weighted H\"older spaces, and in \cite{[Me]} Metsch shows such results in usual H\"older spaces. In the appendix, we collect these results from He-Zeng \cite{[HZ]} and Metsch \cite{[Me]} for fourth-order parabolic equations, and the results from Lieberman \cite{[Lieb]} for second-order parabolic equations. Moreover, we will show necessary interpolation inequalities and give an alterlative proof for the existence of solutions in usual H\"older spaces to fourth-order parabolic equations by using He-Zeng's method in the appendix.

\section{Finite time existence of twisted Calabi flow}
In this section, we prove that the twisted Calabi flow exists on a given finite time interval when $s$ is small, which is the first part of the proof of Theorem \ref{theo:main1}.

\begin{theo}\label{theo:001a}
Let $(M, g)$ be a compact K\"ahler manifold with a K\"ahler metric
$\oo_g$.   For  $\psi_0\in \cH(\oo_g)$, we denote by $\varphi_s(x, t)$
 the solution   of  (\ref{eq:000}) for $s\in [0, 1]$ with the initial K\"ahler potential $\psi_0$. For any $T>0$ and any $\psi_0\in \cH(\oo_g)$, there exists $s_0\in (0, 1]$ such that for any $s\in (0, s_0)$ the solution   $\varphi_s(x, t)$ exists for $t\in [0, T]$ and satisfies
 \beq
\|\varphi_s(x, t)-\varphi_0(x, t)\|_{C^{4, 1, \ga}(M\times [0, T], g)}\leq C(n, g, T, \varphi_0)s. \label{eq:A032}
\eeq

\end{theo}

We remark that in Theorem \ref{theo:001a}, we don't assume the existence of cscK metrics, and $T$ cannot be chosen to be $\infty$ since the estimates on $[0, \infty)$ will be infinity by Lemma \ref{lem:A004} and the proof in Section \ref{sec3} doesn't work.   To prove Theorem \ref{theo:001a}, we first construct the approximate solution $\td \varphi_{N, s}$, and we show that the linearized operator of $L_s$ at $\td \varphi_{N, s}$ is invertible. Then we construct a contraction map and obtain a solution of $L_s(\varphi)=0$ by the fixed point theory.

\subsection{The linearized equation}

In this subsection, we construct an approximate solution $\td \varphi_{N, s}$ such that $L_s(\td \varphi_{N, s})$ has higher order estimates with respect to $s$.
 Let $T>0, \psi_0\in \cH(\oo_g)$ and $\varphi_0(x, t)$ be the solution of the $J$-flow (\ref{eq:A0003}) on $M\times [0, \infty)$ with the initial data $\psi_0$, i.e. $L_0(\varphi_0)=0$. For any integer $N>0$, we define
\beq
\td \varphi_{N, s}(x, t):=\varphi_0(x, t)+\sum_{j=1}^N\, \frac {s^j}{j!} u_{j}(x, t), \label{eq:A026}
\eeq where $u_{j}(x, t)$ are  functions on $M\times [0, T].$ We  write
\beq  \psi_j:=\frac {\partial^j}{\partial s^j}\td \varphi_{N, s}  \no \label{Eq:A002}\eeq
and we denote by
$Q_{a, b}$ any type of finite many contractions of $ \Na^k\psi_l$ for
$2\leq k\leq a$ and $1\leq l\leq b$ with respect to the metric $g_{\td \varphi_{N, s}}$ i.e.
\beq
Q_{a, b}=\sum_{2\leq k_i\leq a, 1\leq l_i\leq b}\,c^{k_1k_2\cdots k_p}_{l_1l_2\cdots l_p}\Na^{k_1}\psi_{l_1}*
\Na^{k_2}\psi_{l_2}*\cdots \Na^{k_p}\psi_{l_p},\no
\eeq where $ c^{k_1k_2\cdots k_p}_{l_1l_2\cdots l_p}$ are some constants. By direct calculation, we have
\beqn
\pd {}s\Na^k\psi_l&=&\Na^k\psi_{l+1}+Q_{k+1, l},\label{eq:D2}\\
\pd {}sQ_{a, b}&=&Q_{a+1, b+1}. \label{eq:D3}
\eeqn

We would like to choose $u_j$ such that $L_s(\td \varphi_{N, s})$ has higher order  estimates with respect to $s$.
First, we have the following expansion of $R(\td \varphi_{N, s})$.

\begin{lem}\label{lem:A002} For $j\geq 2$, we have
\beqn
\frac {\partial^j}{\partial s^j}R(\td \varphi_{N, s})&=&-\Delta_{\td \varphi_{N, s}}^2\psi_j-Ric_{\td \varphi_{N, s}, \bb\bar \al
}\psi_{j, \al\bar \bb}+Ric_{\td \varphi_{N, s}
}*Q_{j+1, j-1}+Q_{j+3, j-1}.
\label{eq:A055}
\eeqn

\end{lem}
\begin{proof}For $j=1$ we have
\beq
\frac {\partial}{\partial s}R(\td \varphi_{N, s})=-\Delta_{\td \varphi_{N, s}}^2\psi_1-Ric_{\td \varphi_{N, s}, \bb\bar \al
}\psi_{1, \al\bar \bb}.\no
\eeq
 Then we calculate
\beqn
\frac {\partial^{j+1}}{\partial s^{j+1}}R(\td \varphi_{N, s})&=&-\frac {\partial^{j}}{\partial s^{j}}\Big(\Delta_{\td \varphi_{N, s}}^2\psi_1\Big)-\frac {\partial^{j}}{\partial s^{j}}\Big(Ric_{\td \varphi_{N, s}, \bb\bar \al
}\psi_{1, \al\bar \bb}\Big). \label{eq:A0008}
\eeqn

We calculate each term of the right-hand side of (\ref{eq:A0008}). Note that
\beqn
\pd {}s\Big(\Delta_{\td \varphi_{N, s}}^2\psi_j\Big)&=&\Big(\pd {}s\Delta_{\td \varphi_{N, s}}\Big)\Delta_{\td \varphi_{N, s}}\psi_j+\Delta_{\td \varphi_{N, s}}\Big(\pd {}s(\Delta_{\td \varphi_{N, s}}\psi_j)\Big)\no\\
&=&\Na^2\psi_1*\Na^4\psi_j+\Delta_{\td \varphi_{N, s}}^2\psi_{j+1}+\Delta_{\td \varphi_{N, s}}\Big(\Na^3\psi_1*\Na \psi_j\Big)\no\\
&=&\Delta_{\td \varphi_{N, s}}^2\psi_{j+1}+Q_{5, j}.\no
\eeqn
This implies that
\beq
\frac {\partial^j}{\partial s^j}\Big(\Delta_{\td \varphi_{N, s}}^2\psi_1\Big)=
\Delta_{\td \varphi_{N, s}}^2\psi_{j+1}+Q_{j+4, j}.  \label{eq:D9}
\eeq

 We claim that
\beqn
\frac {\partial^j}{\partial s^j}\Big(Ric_{\td \varphi_{N, s}, \bb\bar \al
}\psi_{1, \al\bar \bb}\Big)&=&Ric_{\td \varphi_{N, s}, \bb\bar \al
}\psi_{j+1, \al\bar \bb}+Ric_{\td \varphi_{N, s}
}*Q_{j+2, j}+Q_{j+3, j}. \label{eq:D1}
\eeqn
We can check that (\ref{eq:D1}) holds for $j=1$. Assume that (\ref{eq:D1}) holds for $j\geq 1$. Then
\beqn
\frac {\partial^{j+1}}{\partial s^{j+1}}\Big(Ric_{\td \varphi_{N, s}, \bb\bar \al
}\psi_{1, \al\bar \bb}\Big)&=&\pd {}s\Big(Ric_{\td \varphi_{N, s}, \bb\bar \al
}\psi_{j+1, \al\bar \bb}\Big)\no\\&&+\pd {}s\Big(Ric_{\td \varphi_{N, s}
}*Q_{j+2, j}\Big)+\pd {}sQ_{j+3, j}.\label{eq:D7}
\eeqn
By direct calculation and using (\ref{eq:D2})-(\ref{eq:D3}), we have
\beqn
\pd {}s\Big(Ric_{\td \varphi_{N, s}, \bb\bar \al
}\psi_{j+1, \al\bar \bb}\Big)&=&Ric_{\td \varphi_{N, s}, \bb\bar \al
}\psi_{j+2, \al\bar \bb}+Ric_{\td \varphi_{N, s}
}*Q_{3, j+1}+Q_{4, j+1}, \label{eq:D4}\\
\pd {}s\Big(Ric_{\td \varphi_{N, s}
}*Q_{j+2, j}\Big)&=&Ric_{\td \varphi_{N, s}
}*Q_{j+3, j+1}+Q_{j+2, j},  \label{eq:D5}\\
\pd {}sQ_{j+3, j}&=&Q_{j+4, j+1}.   \label{eq:D6}
\eeqn
Combining (\ref{eq:D4})-(\ref{eq:D6}) with (\ref{eq:D7}), we have
\beqn
\frac {\partial^{j+1}}{\partial s^{j+1}}\Big(Ric_{\td \varphi_{N, s}, \bb\bar \al
}\psi_{1, \al\bar \bb}\Big)&=&Ric_{\td \varphi_{N, s}, \bb\bar \al
}\psi_{j+2, \al\bar \bb}+Ric_{\td \varphi_{N, s}
}*Q_{j+3, j+1}+Q_{j+4, j+1}.\label{eq:D8}
\eeqn
Thus, (\ref{eq:D1}) is proved.

Combining (\ref{eq:D8}) with (\ref{eq:D9}), we have
\beqn
\frac {\partial^{j+1}}{\partial s^{j+1}}R(\td \varphi_{N, s})&=&-\Delta_{\td \varphi_{N, s}}^2\psi_{j+1}+Q_{j+4, j}-Ric_{\td \varphi_{N, s}, \bb\bar \al
}\psi_{j+1, \al\bar \bb}+Ric_{\td \varphi_{N, s}
}*Q_{j+2, j}+Q_{j+3, j}\no\\
&=&-\Delta_{\td \varphi_{N, s}}^2\psi_{j+1}-Ric_{\td \varphi_{N, s}, \bb\bar \al
}\psi_{j+1, \al\bar \bb}+Ric_{\td \varphi_{N, s}
}*Q_{j+2, j}+Q_{j+4, j}.\no
\eeqn The lemma is proved.

\end{proof}

Similarly, for $j\geq 2$ we have
 \beq
\frac {\partial^j}{\partial s^j}\tr_{\td \varphi_{N, s}}\oo_g=-g_{\td \varphi_{N, s}}^{\al\bar \dd}g_{\td \varphi_{N, s}}^{\eta\bar \bb}g_{\eta\bar \dd}\psi_{j, \al\bar \bb}+Q_{j+1, j-1}.\label{eq:A001}
\eeq Combining Lemma \ref{lem:A002} with the equality (\ref{eq:A001}), we have the expansion of $L_s(\td \varphi_{N, s}). $
\begin{lem}\label{lem:A003}We have
\beqs
\pd {}sL_s(\td \varphi_{N, s})&=&\frac {\partial}{\partial t}\psi_1-(1-s) g_{\td \varphi_{N, s}}^{\al\bar \dd}g_{\td \varphi_{N, s}}^{\eta\bar \bb}g_{\eta\bar \dd}\psi_{j, \al\bar \bb}+s\Delta_{\td \varphi_{N, s}}^2\psi_1\\&&+sRic_{\td \varphi_{N, s}, \bb\bar \al
}\psi_{1, \al\bar \bb}-(R(\td \varphi_{N, s})-\un R)+(n-\tr_{\td \varphi_{N, s}}g),
\eeqs and for $j\geq 2$ we have
 \beqn
\frac {\partial^j}{\partial s^j}L_s(\td \varphi_{N, s})
&=&\pd {\psi_j}t-(1-s)g_{\td \varphi_{N, s}}^{\al\bar \dd}g_{\td \varphi_{N, s}}^{\eta\bar \bb}g_{\eta\bar \dd}\psi_{j, \al\bar \bb}+ s\Delta_{\td \varphi_{N, s}}^2\psi_j+sRic_{\td \varphi_{N, s}, \bb\bar \al
}\psi_{j, \al\bar \bb}\no\\&&
+Q_{j+3, j-1}+Ric(\td \varphi_{N, s})*Q_{j+1, j-1}.\no
\eeqn

\end{lem}
\begin{proof}
Direct calculation shows that
\beqn
\pd {}sL_s(\td \varphi_{N, s})&=&\frac {\partial}{\partial t}\psi_1+(1-s)\frac {\partial}{\partial s}\tr_{\td \varphi_{N, s}}\oo_g-s\pd {}sR(\td \varphi_{N, s})\no\\&&-(R(\td \varphi_{N, s})-\un R)+(n-\tr_{\td \varphi_{N, s}}g)\no\\
&=&\frac {\partial}{\partial t}\psi_1-(1-s) g_{\td \varphi_{N, s}}^{\al\bar \dd}g_{\td \varphi_{N, s}}^{\eta\bar \bb}g_{\eta\bar \dd}\psi_{1, \al\bar \bb}+s\Big(\Delta_{\td \varphi_{N, s}}^2\psi_1+Ric_{\td \varphi_{N, s}, \bb\bar \al
}\psi_{1, \al\bar \bb}\Big)\no\\&&-(R(\td \varphi_{N, s})-\un R)+(n-\tr_{\td \varphi_{N, s}}g). \no
\eeqn
For $j\geq 2$, we have
\beqn
\frac {\partial^j}{\partial s^j}L_s(\td \varphi_{N, s})&=&\pd {\psi_j}t+(1-s)\frac {\partial^j}{\partial s^j}\tr_{\td \varphi_{N, s}}\oo_g-j\frac {\partial^{j-1}}{\partial s^{j-1}}\tr_{\td \varphi_{N, s}}\oo_g \no\\ &&
-s\frac {\partial^j}{\partial s^j}R(\td \varphi_{N, s})-js \frac {\partial^{j-1}}{\partial s^{j-1}}R(\td \varphi_{N, s}). \label{eq:A001a}
\eeqn
Combining  Lemma \ref{lem:A002} with  (\ref{eq:A001})-(\ref{eq:A001a}), we have
\beqs &&
\frac {\partial^j}{\partial s^j}L_s(\td \varphi_{N, s})\\
&=&\pd {\psi_j}t+(1-s)\Big(-g_{\td \varphi_{N, s}}^{\al\bar \dd}g_{\td \varphi_{N, s}}^{\eta\bar \bb}g_{\eta\bar \dd}\psi_{j, \al\bar \bb}+Q_{j+1, j-1}\Big)\\
&&-j\Big(-g_{\td \varphi_{N, s}}^{\al\bar \dd}g_{\td \varphi_{N, s}}^{\eta\bar \bb}g_{\eta\bar \dd}\psi_{j-1, \al\bar \bb}+Q_{j, j-2}\Big)\\
&&-s\Big(-\Delta_{\td \varphi_{N, s}}^2\psi_j-Ric_{\td \varphi_{N, s}, \bb\bar \al
}\psi_{j, \al\bar \bb}+Ric_{\td \varphi_{N, s}
}*Q_{j+1, j-1}+Q_{j+3, j-1}\Big)\\
&&-js\Big(-\Delta_{\td \varphi_{k, s}}^2\psi_{j-1}-Ric_{\td \varphi_{N, s}, \bb\bar \al
}\psi_{j-1, \al\bar \bb}+Ric_{\td \varphi_{N, s}
}*Q_{j, j-2}+Q_{j+2, j-2}
\Big)\\
&=&\pd {\psi_j}t-(1-s)g_{\td \varphi_{N, s}}^{\al\bar \dd}g_{\td \varphi_{N, s}}^{\eta\bar \bb}g_{\eta\bar \dd}\psi_{j, \al\bar \bb}+ s\Delta_{\td \varphi_{N, s}}^2\psi_j+sRic_{\td \varphi_{N, s}, \bb\bar \al
}\psi_{j, \al\bar \bb}\\&&
+Q_{j+3, j-1}+Ric(\td \varphi_{N, s})*Q_{j+1, j-1}.
\eeqs
The lemma is proved.

\end{proof}

Combining Lemma \ref{lem:A002} with Lemma \ref{lem:A003}, we can construct the high-order approximate solution $\td \varphi_{N, s}$ of the equation $L_s(\varphi)=0$ with respect to $s$.
\begin{lem}\label{lem007} Fix $\ga\in (0, 1)$ and integer $N\geq 1$.  Let $\varphi_0(x, t)$ be the solution of the $J$-flow (\ref{eq:A0003}) on $M\times [0, T]$.
 There exist $\{u_{j}(x, t)\}_{j=1}^N$ such that the function $\td \varphi_{N, s}$ defined by (\ref{eq:A026}) satisfies
\beqn
\|\td \varphi_{N, s}(x, t)-\varphi_0(x, t)\|_{C^{4+\ga}(M\times [0, T], g)}&\leq& C(\te, \Te,  T)s, \no\\
\|L_s(\td \varphi_{N, s}(x, t))\|_{C^{ \ga}(M\times [0, T], g)}&\leq& C(\te, \Te, T)s^{N+1},\no
\eeqn where $\te$ and $\Te$ are two constants satisfying the following conditions on $M\times [0, T]$
\beq
\oo_{\varphi_0(x, t)}\geq \te \oo_g, \quad \|\varphi_0\|_{C^{N'+\ga}(M\times [0, T], g)}\leq \Te.\label{eq:A100z}
\eeq Here $N'$ is an integer depending only on $N$, and the parabolic H\"older norm $\|\cdot\|_{C^{k+\ga}(M\times [0, T], g)}$ is defined in Definition \ref{defi:B001} the appendix.

\end{lem}

\begin{proof}
For any smooth function $f(s)$, we have the expansion
\beq f(s)=f(0)+f'(0)s+\cdots+\frac {f^{(N)}(0)}{N!}s^N+\int_0^s\,
\frac {f^{(N+1)}(\tau)}{N!}(s-\tau)^N\,d\tau.\label{eq:A054}\no
\eeq
Let $f(s):=L_s(\td \varphi_{N, s}).$ We would like to find functions
$u_1, u_2, \cdots, u_N$
 such that
\beqn
f(0)&=&f'(0)=\cdots=f^{(N)}(0)=0, \label{eq:A202}\\
u_j(x, 0)&=&0,\quad \forall\,x\in M, \quad j\in {1, 2, \cdots, N}.\label{eq:A070x}
\eeqn
By Lemma \ref{lem:A003} the equalities (\ref{eq:A202}) are equivalent to
\beqn
&&\pd {u_1}t-g_{\varphi_0}^{i\bar l}g_{\varphi_0}^{k\bar j}g_{i\bar j}u_{1, k\bar l}-\Big(R(\varphi_0)-\un R-n+\tr_{\varphi_0}\oo_g\Big)=0,\label{eq:A201ya}\\
&&\pd {u_j}t-g_{\varphi_0}^{\al\bar \dd}g_{\varphi_0}^{\eta\bar \bb}g_{\eta\bar \dd}u_{j, \al\bar \bb}+Q_{j+3, j-1}|_{s=0}+Ric(\varphi_0)*Q_{j+1, j-1}|_{s=0}=0.\no\\ \label{eq:A201y}
\eeqn where $2\leq j\leq N. $
By the standard second-order parabolic theory, we can find smooth functions $u_1, u_2, \cdots, u_N$  satisfying (\ref{eq:A201ya}), (\ref{eq:A201y}) and (\ref{eq:A070x}).
  Moreover, by Lemma \ref{lem:A204} below  the $L^2$ norm of $u_j$ is bounded and by Theorem \ref{theo:A4} the $L^{\infty}$  norm of $u_j$ is bounded. Therefore, by Theorem \ref{theo:A3} all higher order H\"older norms $\|u_j\|_{C^{k+\ga}(M, g)}$ can be bounded by the H\"older norms of $\varphi_0$.
In particular, $f^{(N+1)}(\tau)(\tau\in [0, s])$ is bounded and we have
 \beqs
\|f(s)\|_{C^{ \ga}(M\times [0, T], g)}&=&\Big|\int_0^s\,
\frac {f^{(N+1)}(\tau)}{N!}(s-\tau)^N\,d\tau\Big|_{C^{ \ga}(M\times [0, T], g)}\\
&\leq& C(\|\varphi_0\|_{C^{N'+\ga}}, T)s^{N+1},
\eeqs where $N'>0$ is an integer depending only on $N$ and $n$. The number $N'$ can be calculated explicitly, but we don't need it in the paper. The lemma is proved.

\end{proof}

The following result gives $L^2$ estimates of second-order parabolic equations.

\begin{lem}\label{lem:A204}Assume that
\begin{enumerate}
  \item[(1).] $\varphi_0(x, t)$ is a family of K\"ahler potentials satisfying
  (\ref{eq:A100z}) for two constants $\te$ and $\Te$;
  \item[(2).] $u(x, t)$ satisfies the equation on $M\times [0, T]$
  \beq
  \pd ut=b_{l\bar k}u_{k\bar l}+h,\quad u(0)=0, \no
  \eeq where $b_{l\bar k}=g_{\varphi_0}^{i\bar l}g_{\varphi_0}^{k\bar j}g_{i\bar j}$.
\end{enumerate}  Then we have
\beq
\|u(x, t)\|_{L^2(M, \oo_{\varphi_0(\tau)})}\leq C(\te, \Te, T)\max_{\tau\in [0, t]}\|h(x, \tau)\|_{L^2(M, \oo_{\varphi_0(\tau)})}.\no
\eeq
\end{lem}
\begin{proof}By the assumption $(1)$, we have $b_{i\bar j}\geq \la\oo_{\varphi_0}$ for a constant $\la:=\la(\Te)>0.$ Direct calculation shows that
\beqn
\frac d{dt}\int_M\;u^2\,\oo_{\varphi_0}^n&=& \int_M\; \Big(2u(b_{l\bar k}u_{k\bar l}+h)+
u^2\,\Delta_{\varphi_0}\pd {\varphi_0}t\Big)\,\oo_{\varphi_0}^n\no\\
&=&\int_M\; \Big(-2b_{l\bar k}u_ku_{\bar l}-2b_{l\bar k, \bar l}uu_k+2uh+
u^2\,\Delta_{\varphi_0}\pd {\varphi_0}t\Big)\,\oo_{\varphi_0}^n\no\\
&\leq&\int_M\; \Big(-2\la|\Na u|^2+C(\te, \Te)|u||\Na u|+2uh+
u^2\,\Delta_{\varphi_0}\pd {\varphi_0}t\Big)\,\oo_{\varphi_0}^n\no\\
&\leq&\int_M\; \Big(-\la|\Na u|^2+C(\te, \Te, \la)u^2\Big)\,\oo_{\varphi_0}^n+2\|u\|_{L^2(M, \oo_{\varphi_0})}\|h\|_{L^2(M, \oo_{\varphi_0})},\no\\\label{eq:A203}
 \eeqn where we used the inequality
 $$C(\te, \Te)|u||\Na u|\leq \la|\Na u|^2+C(\la, \te, \Te)u^2. $$
Let $E(t)=\|u\|_{L^2(M, \oo_{\varphi_0})}$. Then (\ref{eq:A203}) implies that
\beq
E'(t)\leq C(\Te, \te)E(t)+\|h\|_{L^2(M, \oo_{\varphi_0})}.\no
\eeq
Therefore, we have
\beq
E(t)\leq C(\te, \Te, T)\max_{\tau\in [0, t]}\|h(\cdot, \tau)\|_{L^2(M, \oo_{\varphi_0(\tau)})}.\no
\eeq The lemma is proved.

\end{proof}

\subsection{The $L^2$ estimates}

In this subsection, we give $L^2$ estimates of the linearized equation of the twisted Calabi flow on a given finite time interval, and  show that the linearized equation of twisted Calabi flow is invertible by using the  Schauder estimates of fourth-order parabolic equations.

 Let $T>0$ and consider the equation of $w(x, \tau)$
\beqn
\pd {w}{\tau}+\Delta^2_{\td \varphi}w&=&s^{-1}f(x, \tau)+s^{-1}a_{l\bar k}(x, \tau)w_{k\bar l},\quad (x, \tau)\in M\times [0, T],\label{eq:B003}\\
w(x, 0)&=&0 \label{eq:B004}
\eeqn where $a_{l\bar k}$ is defined by (\ref{eq:A0007b}).
First, we show that  the $L^2$ norm of $w$ is bounded by a constant depending on $T$.
\begin{lem}\label{lem:A004}Let $\td \varphi$ be a family of K\"ahler potentials satisfying the following conditions   on $M\times [0, T]$
\beq
\oo_{\td \varphi(x, t)}\geq \te \oo_g, \quad \|\td \varphi\|_{C^{N'+\ga}(M\times [0, T], g)}\leq \Te.\label{eq:A100}
\eeq
There exists $s_0\in (0, 1)$ depending only on $\te$ and $\Te$ satisfying the following property.
If $w(x, \tau)$ is a solution of (\ref{eq:B003})-(\ref{eq:B004}) with $s\in (0, s_0)$,
there exists a constant $C(\te,  \Te  )$ such that for any $\tau\in [0, T]$
\beq
\|w(\cdot, \tau)\|_{L^2(M, \oo_{\td \varphi(\tau)})}\leq C(\te, \Te)e^{C(\te, \Te)s^{-1}\tau}\sup_{t\in [0, \tau]}\|f(\cdot, t)\|_{L^2(M, \oo_{\td \varphi(t)})}. \no
\eeq

\end{lem}
\begin{proof}By direct calculation, we have
\beqn
\pd {}{\tau}\int_M\;w^2\oo_{\td \varphi}^n&=&\int_M\;2w\pd w{\tau}\,\oo_{ \td \varphi}^n+w^2
\pd {}{\tau}\oo_{ \td \varphi}^n\no\\
&\leq&\int_M\;2w\Big(-\Delta^2_{ \td \varphi}w+\frac 1sf
+\frac 1s a_{l\bar k}w_{k\bar l}\Big)\oo_{ \td \varphi}^n+C(\Te)\int_M\;w^2\oo_{ \td \varphi}^n\no\\
&\leq&-2\int_M\;(\Delta_{ \td \varphi}w)^2\oo_{ \td \varphi}^n+2s^{-1}\int_M\;a_{l\bar k}w_{k\bar l}w\,\oo_{\td \varphi}^n\no\\
&&+2s^{-1}\int_M\;fw+C(\Te)\int_M\;w^2\oo_{ \td \varphi}^n.\label{eq:A016a}
\eeqn
Note that by (\ref{eq:A100}) we have $g_{i\bar j}\geq C(\Te)g_{\td \varphi}$, and
by (\ref{eq:A0006}) we have
\beqn
a_{l\bar k}w_kw_{\bar l}&=&-sRic_{\td \varphi, l\bar k}w_kw_{\bar l}+(1-s)g_{\td \varphi}^{i\bar l}g_{\td \varphi}^{k\bar j}g_{i\bar j}w_kw_{\bar l}\no\\
&\geq&(1-s)C(\Te)|\Na w|^2_{\td \varphi}-sC(\te, \Te)|\Na w|^2_{\td \varphi},\no
\eeqn which implies that
\beq
\int_M\;a_{l\bar k}w_{k}w_{\bar l}\,\oo_{\td \varphi}^n
\geq\Big((1-s)C(\Te)-C(\te, \Te)s\Big)\int_M\;|\Na w|^2\,\oo_{\td \varphi}^n.\no
\eeq
Moreover,  by (\ref{eq:A100}) we have
\beq
|a_{l\bar k, \bar l}w_{k}w|\leq\ee |\Na w|^2+C(\te, \Te, \ee)w^2.\no
\eeq
Thus, we have
\beqn &&
\int_M\;a_{l\bar k}w_{k\bar l}w\,\oo_{\td \varphi}^n=-\int_M\;a_{l\bar k}w_{k}w_{\bar l}\,\oo_{\td \varphi}^n-\int_M\;a_{l\bar k, \bar l}w_{k}w\,\oo_{ \td \varphi}^n\no\\
&\leq&-((1-s)C(\Te)-C(\te, \Te)s-\ee) \int_M\;|\Na w|^2\,\oo_{ \td \varphi}^n+C(\te,  \Te, \ee )\int_M\;w^2\,\oo_{ \td \varphi}^n.\no\\&&\label{eq:A017a}
\eeqn  For small $s$ and $\ee$ depending only on $\Te$ and $\te$, we have $\la:=(1-s)C(\Te)-C(\te, \Te)s-\ee>0.$
Combining (\ref{eq:A016a}) with (\ref{eq:A017a}), we have
\beqn
\pd {}{\tau}\int_M\;w^2\oo_{ \td \varphi}^n&\leq&
-2\int_M\;(\Delta_{\td \varphi}w)^2\oo_{\td \varphi}^n-\la s^{-1}\int_M\;|\Na w|^2\,\oo_{ \td \varphi}^n\no\\&&+
 {C(\te,  \Te) }s^{-1}\int_M\;w^2\,\oo_{ \td \varphi}^n+2s^{-1}\Big(\int_M\,f^2\,\oo_{ \td \varphi}^n\Big)^{\frac 12}\Big(\int_M\;w^2\,\oo_{ \td \varphi}^n\Big)^{\frac 12}\no\\
&\leq&C(\te, \Te)s^{-1}\int_M\;w^2\,\oo_{ \td \varphi}^n+2s^{-1}\Big(\int_M\,f^2\,\oo_{ \td \varphi}^n\Big)^{\frac 12}\Big(\int_M\;w^2\,\oo_{ \td \varphi}^n\Big)^{\frac 12}. \no\\&&\label{eq:A018a}
\eeqn
Let
$ E(\tau)= \Big(\int_M\;w^2\oo_{\td \varphi}\Big)^{\frac 12}.
$
The inequality (\ref{eq:A018a}) implies that
\beq
E'(\tau)\leq C(\te,  \Te)s^{-1} E(\tau)+s^{-1}\|f(\cdot, \tau)\|_{L^2(M, \oo_{\td \varphi(\tau)})}.\no
\eeq
Thus, we have
\beq
E(\tau)\leq C(\te,  \Te)e^{C(\te, \Te)s^{-1}\tau}\sup_{t\in [0, \tau]}\|f(\cdot, t)\|_{L^2(M, \oo_{\td \varphi(t)})}.  \no
\eeq
The lemma is proved.

\end{proof}

To estimate the fourth-order parabolic equation, we introduce the following H\"older spaces.
\begin{defi}Let  $\ga\in (0, 1), k\in \NN, b>a\geq 0$ and $\psi\in \cH(\oo_g)$.  We  define  $C^{k, [k/4], \ga }(M\times [a, b], g, \psi)$ the space of functions $f\in C^{k, [k/4], \ga }(M\times [a, b], g)$ with $f(x, 0)=\psi(x)$ for all $x\in M$. Here the space $C^{k, [k/4], \ga }(M\times [a, b], g)$ is defined in Definition \ref{defi:A001} and Definition \ref{defi:A002} in the appendix.

\end{defi}

\begin{lem}\label{lem004}  Let $\td \varphi$ the function and $s_0$ the constant in Lemma \ref{lem:A004}. Then for any $s\in (0, s_0)$, the operator
\beq
DL_s|_{\td \varphi}:  C^{4, 1, \ga}(M\times [0, T], g, 0)\ri C^{0, 0, \ga}(M\times [0, T], g)\no
\eeq
is injective and surjective. Moreover, we have
\beq
\|(DL_s|_{\td \varphi})^{-1}\|\leq C(\te, \Te, n, g, T) s^{-\ka}\no
\eeq where $\ka=\ka(\ga, n)>1.$

\end{lem}
\begin{proof}

For any $f\in C^{0, 0, \ga}(M\times [0, T], g)$, we would like to solve the equation
\beq
DL_s|_{\td \varphi}(u)=f,\quad  u(x, 0)=0.\label{eq:A035}
\eeq
As in Section \ref{sec2}, the equality (\ref{eq:A035}) can be written as
\beqn
\pd {w}{\tau}+\Delta^2_{\td\varphi}w&=&h(x, \tau):=s^{-1}\td f+s^{-1}a_{l\bar k}w_{k\bar l},\quad (x, \tau)\in M\times [0, sT], \label{eq:A083}\\
w(x, 0)&=&0,\label{eq:A084}
\eeqn where $\tau, w$ and $\td f$ are defined by
\beq \tau:=st,\quad w(x, \tau):=u(x, t),\quad \td f(x, \tau):=f(x, t).\no
 \eeq By Theorem \ref{theo:A2x}, for any $f\in C^{0, 0, \ga}(M\times [0, T], g)$ there exists a solution $w(x, t)\in C^{4, 1, \ga}(M\times [0, T], g)$  of (\ref{eq:A083})-(\ref{eq:A084}).
By Theorem \ref{theo:Me}, the solution $w$ satisfies
\beqn &&
\|w\|_{{C^{4, 1, \ga}(M\times [0, sT], g)}}\leq C(\te, \Te)\Big( \|h\|_{{C^{0, 0, \ga}(M\times [0, sT], g)}}+\max_{t\in [0, sT]}\|w\|_{L^2(M, g)}\Big)\no\\
&=&C(\te, \Te)s^{-1}\Big( \|\td f\|_{{C^{0, 0, \ga}(M\times [0, sT], g)}}+\|\Na^2 w\|_{{C^{0, 0, \ga}(M\times [0, sT], g)}}\Big)\no\\&&+C(\te, \Te)\max_{t\in [0, sT]}\|w\|_{L^2(M, g)}.\label{eq:A043a}
\eeqn Note that by Lemma \ref{lem:A5} we have the interpolation inequality
\beq
\|\Na^2 w\|_{{C^{0, 0, \ga}(M\times [0, sT], g)}}\leq \ee s \|w\|_{{C^{4, 1, \ga}(M\times [0, T], g)}}+C(n, g) \ee^{-\ka'}s^{-\ka'}\|w\|_{L^2(M, g)}. \label{eq:A044a}
\eeq where  $\ka'>1. $
Combining the inequalities (\ref{eq:A043a})-(\ref{eq:A044a}), we have
\beqn&&
\|w\|_{{C^{4, 1, \ga}(M\times [0, sT], g)}}\leq C(\te, \Te)s^{-1} \|\td f\|_{{C^{0, 0, \ga}(M\times [0, sT], g)}}\no\\&&+C(\te, \Te, n, g)\ee  \|w\|_{{C^{4, 1, \ga}(M\times [0, sT], g)}}+C(\te, \Te, n, g) s^{-\ka'-1} \max_{t\in [0, sT]}\|w\|_{L^2(M, g)}.\no\\\label{eq:A200}\eeqn
By Lemma \ref{lem:A004} there exists $s_0(\te, \Te)>0$ such that for any $s\in (0, s_0)$ we have
\beq
\max_{t\in [0, sT]}\|w\|_{L^2(M, g)}\leq C(\te, \Te)e^{C(\te, \Te) T}\sup_{t\in [0, sT]}\|\td f(\cdot, t)\|_{L^2(M, \oo_{\td \varphi(t)})}.\label{eq:A201}
\eeq
Thus, (\ref{eq:A200}) and (\ref{eq:A201}) imply that
\beqn &&
\|w\|_{{C^{4, 1, \ga}(M\times [0, sT], g)}}\no\\&\leq& C(\te, \Te, n, g)s^{-1} \|\td f\|_{{C^{0, 0, \ga}(M\times [0, sT], g)}}+C(\te, \Te, n, g)s^{-\ka'-1} \max_{t\in [0, sT]}\|w\|_{L^{2}(M, g)}\no\\
&\leq&C(\te, \Te, n, g)s^{-1} \|\td f\|_{{C^{0, 0, \ga}(M\times [0, sT], g)}}+C(\te, \Te, n, g) s^{-\ka'-1}\max_{t\in [0, sT]}\|\td f\|_{L^{2}(M, g)}\no\\
&\leq&C(\te, \Te, n, g, T)s^{-\ka'-1} \|\td f\|_{{C^{0, 0, \ga}(M\times [0, sT], g)}}. \label{eq:A013}
\eeqn
By Lemma \ref{lem:B008z}, we have
\beqn
\|\td f\|_{{C^{0, 0, \ga}(M\times [0, sT], g)}}&\leq&s^{-\frac {\ga}4}\| f\|_{{C^{0, 0, \ga}(M\times [0, T], g)}},\label{eq:A014}\\
\|w\|_{{C^{4, 1, \ga}(M\times [0, sT], g)}}&\geq&\|u\|_{{C^{4, 1, \ga}(M\times [0, T], g)}}.\label{eq:A015}
\eeqn Thus, (\ref{eq:A013})-(\ref{eq:A015}) imply that
\beq
\|u\|_{{C^{4, 1, \ga}(M\times [0, T], g)}}\leq C(\te, \Te, n, g, T)s^{-\ka'-1-\frac {\ga}4} \| f\|_{{C^{0, 0, \ga}(M\times [0, T], g)}}.\no
\eeq
The lemma is proved.

\end{proof}

\subsection{The contraction map}
In this subsection, we construct a contraction map, which will be used in the next section to show the existence of the solution of twisted Calabi flow by the fixed point theory. First, we have the following two lemmas, which give the differences of the linearized operator $DL_s$ at two K\"ahler potentials.

\begin{lem}\label{lem010}For any $\varphi_1, \varphi_2\in \cH(\oo_g)$ with
\beq
\oo_{\varphi_1}\geq \te\oo_g, \quad \oo_{\varphi_1}\geq \te\oo_g,\quad \|\varphi_1\|_{C^4(M, g)}\leq \Te, \quad \|\varphi_2\|_{C^4(M, g)}\leq \Te,\no
\eeq
for any function $v$, we have
\beqn
&&
|\Delta_{\varphi_1}^2v-\Delta_{\varphi_2}^2v|\no\\
&\leq&C(\te, \Te)\Big(|\Na^2v|\cdot |\Na^4(\varphi_1-\varphi_2)|
+|\Na^2v|\cdot |\Na^3(\varphi_1-\varphi_2)|+|\Na^2v|\cdot |\Na^2(\varphi_1-\varphi_2)|\no\\
&&+|\Na^3v|\cdot |\Na^3(\varphi_1-\varphi_2)|+|\Na^3v|\cdot|\Na^2(\varphi_1-\varphi_2)|+|\Na^4v|\cdot |\Na^2(\varphi_1-\varphi_2)|\Big).\label{eq:A056}
\eeqn
\end{lem}
\begin{proof}

(1). Note that
\beqn
|g_{\varphi_1}^{i\bar j}-g_{\varphi_2}^{i\bar j}|&=&|g_{\varphi_1}^{i\bar l}(g_{\varphi_1, k\bar l}-g_{\varphi_2, k\bar l})g_{\varphi_2}^{k\bar j}|=|g_{\varphi_1}^{i\bar l}g_{\varphi_2}^{k\bar j}\partial_k\partial_{\bar l}(\varphi_1-\varphi_2)|\no\\&\leq& C(\te)|\Na^2(\varphi_1-\varphi_2)|. \label{eq:A060}
\eeqn
Thus, we have
\beqn &&
|\partial_{p}g_{\varphi_1}^{i\bar j}-\partial_{p}g_{\varphi_2}^{i\bar j}|=
|\partial_{p}( g_{\varphi_1}^{i\bar l}(g_{\varphi_1, k\bar l}-g_{\varphi_2, k\bar l})g_{\varphi_2}^{k\bar j}   )|\no\\
&\leq&|g_{\varphi_1}^{i\bar l}(\partial_{p}g_{\varphi_1, k\bar l}-\partial_{p}g_{\varphi_2, k\bar l})g_{\varphi_2}^{k\bar j}|+| \partial_{p} g_{\varphi_1}^{i\bar l}(g_{\varphi_1, k\bar l}-g_{\varphi_2, k\bar l})g_{\varphi_2}^{k\bar j}  |+|  g_{\varphi_1}^{i\bar l}(g_{\varphi_1, k\bar l}-g_{\varphi_2, k\bar l})\partial_{p}g_{\varphi_2}^{k\bar j}  |\no\\
&\leq&C(\te, \Te)\Big(|\Na^3(\varphi_1-\varphi_2)|+|\Na^2(\varphi_1-\varphi_2)|\Big).\label{eq:A061}
\eeqn
Moreover, we estimate the second derivatives.
\beqn &&
|\partial_{\bar q}\partial_{p}g_{\varphi_1}^{i\bar j}-\partial_{\bar q}\partial_{p}g_{\varphi_2}^{i\bar j}|=|\partial_{\bar q}(\partial_{p}g_{\varphi_1}^{i\bar j}-\partial_{p}g_{\varphi_2}^{i\bar j})|\no\\
&\leq&\Big|\partial_{\bar q}\Big(g_{\varphi_1}^{i\bar l}(\partial_{p}g_{\varphi_1, k\bar l}-\partial_{p}g_{\varphi_2, k\bar l})g_{\varphi_2}^{k\bar j}\Big)\Big|+\Big| \partial_{\bar q}\Big(\partial_{p} g_{\varphi_1}^{i\bar l}(g_{\varphi_1, k\bar l}-g_{\varphi_2, k\bar l})g_{\varphi_2}^{k\bar j} \Big) \Big|\no\\
&&+\Big| \partial_{\bar q}\Big( g_{\varphi_1}^{i\bar l}(g_{\varphi_1, k\bar l}-g_{\varphi_2, k\bar l})\partial_{p}g_{\varphi_2}^{k\bar j} \Big) \Big|\no
\\
&:=&J_1+J_2+J_3.\label{eq:A057}
\eeqn We estimate each term $J_i.$
\beqn &&
\Big|\partial_{\bar q}\Big(g_{\varphi_1}^{i\bar l}(\partial_{p}g_{\varphi_1, k\bar l}-\partial_{p}g_{\varphi_2, k\bar l})g_{\varphi_2}^{k\bar j}\Big)\Big|
\leq
\Big|\partial_{\bar q}g_{\varphi_1}^{i\bar l}(\partial_{p}g_{\varphi_1, k\bar l}-\partial_{p}g_{\varphi_2, k\bar l})g_{\varphi_2}^{k\bar j}\Big|\no\\&&+
\Big|g_{\varphi_1}^{i\bar l}\partial_{\bar q}(\partial_{p}g_{\varphi_1, k\bar l}-\partial_{p}g_{\varphi_2, k\bar l})g_{\varphi_2}^{k\bar j} \Big| +
\Big| g_{\varphi_1}^{i\bar l}(\partial_{p}g_{\varphi_1, k\bar l}-\partial_{p}g_{\varphi_2, k\bar l})\partial_{\bar q}g_{\varphi_2}^{k\bar j} \Big|\no\\
&\leq&C(\te, \Te)\Big(|\Na^4(\varphi_1-\varphi_2)|+|\Na^3(\varphi_1-\varphi_2)|\Big).\label{eq:A058}
\eeqn
Similarly, we have
\beq
J_2+J_3\leq C(\te, \Te)\Big(|\Na^2(\varphi_1-\varphi_2)|+|\Na^3(\varphi_1-\varphi_2)|\Big).\label{eq:A059}
\eeq
Combining (\ref{eq:A057})-(\ref{eq:A059}), we have
\beq
|\partial_{\bar q}\partial_{p}g_{\varphi_1}^{i\bar j}-\partial_{\bar q}\partial_{p}g_{\varphi_2}^{i\bar j}|\leq C(\te, \Te)\Big(|\Na^4(\varphi_1-\varphi_2)|+|\Na^3(\varphi_1-\varphi_2)|
+|\Na^2(\varphi_1-\varphi_2)|\Big).\label{eq:A062}
\eeq

(2). By direct calculation, we have
\beqn \Delta_{\varphi_1}^2v&=&
g^{i\bar j}_{\varphi_1}\partial_i\partial_{\bar j}(g_{\varphi_1}^{k\bar l}\partial_k\partial_{\bar l}v)
=g_{\varphi_1}^{i\bar j}\partial_i\Big(\partial_{\bar j}g_{\varphi_1}^{k\bar l}\partial_k\partial_{\bar l}v+g_{\varphi_1}^{k\bar l}\partial_{\bar j}\partial_k\partial_{\bar l}v\Big)\no\\
&=&g_{\varphi_1}^{i\bar j}\partial_i\partial_{\bar j}g_{\varphi_1}^{k\bar l}\partial_k\partial_{\bar l}v+g_{\varphi_1}^{i\bar j}\partial_{\bar j}g_{\varphi_1}^{k\bar l}\partial_i\partial_k\partial_{\bar l}v\no\\&&+g_{\varphi_1}^{i\bar j}\partial_ig_{\varphi_1}^{k\bar l}\partial_{\bar j}\partial_k\partial_{\bar l}v+g_{\varphi_1}^{i\bar j}g_{\varphi_1}^{k\bar l}\partial_i\partial_{\bar j}\partial_k\partial_{\bar l}v. \label{eq:A051}
\eeqn
Thus, we have
\beqn
\Delta_{\varphi_1}^2v-\Delta_{\varphi_2}^2v
&=&\Big(g_{\varphi_1}^{i\bar j}\partial_i\partial_{\bar j}g_{\varphi_1}^{k\bar l}\partial_k\partial_{\bar l}v-g_{\varphi_2}^{i\bar j}\partial_i\partial_{\bar j}g_{\varphi_2}^{k\bar l}\partial_k\partial_{\bar l}v\Big)\no\\
&&+\Big(g_{\varphi_1}^{i\bar j}\partial_{\bar j}g_{\varphi_1}^{k\bar l}\partial_i\partial_k\partial_{\bar l}v-g_{\varphi_2}^{i\bar j}\partial_{\bar j}g_{\varphi_2}^{k\bar l}\partial_i\partial_k\partial_{\bar l}v\Big)\no\\
&&+\Big(g_{\varphi_1}^{i\bar j}\partial_ig_{\varphi_1}^{k\bar l}\partial_{\bar j}\partial_k\partial_{\bar l}v-g_{\varphi_2}^{i\bar j}\partial_ig_{\varphi_2}^{k\bar l}\partial_{\bar j}\partial_k\partial_{\bar l}v\Big)\no\\
&&+\Big(g_{\varphi_1}^{i\bar j}g_{\varphi_1}^{k\bar l}\partial_i\partial_{\bar j}\partial_k\partial_{\bar l}v-g_{\varphi_2}^{i\bar j}g_{\varphi_2}^{k\bar l}\partial_i\partial_{\bar j}\partial_k\partial_{\bar l}v\Big)\no\\
&:=&I_1+I_2+I_3+I_4. \label{eq:A052}
\eeqn
We estimate each $I_i$.
\beqn
I_1&=&|g_{\varphi_1}^{i\bar j}\partial_i\partial_{\bar j}g_{\varphi_1}^{k\bar l}\partial_k\partial_{\bar l}v-g_{\varphi_2}^{i\bar j}\partial_i\partial_{\bar j}g_{\varphi_2}^{k\bar l}\partial_k\partial_{\bar l}v|\no\\
&\leq&|g_{\varphi_1}^{i\bar j}\partial_i\partial_{\bar j}g_{\varphi_1}^{k\bar l}\partial_k\partial_{\bar l}v-g_{\varphi_1}^{i\bar j}\partial_i\partial_{\bar j}g_{\varphi_2}^{k\bar l}\partial_k\partial_{\bar l}v|\no\\
&&+|g_{\varphi_1}^{i\bar j}\partial_i\partial_{\bar j}g_{\varphi_2}^{k\bar l}\partial_k\partial_{\bar l}v-g_{\varphi_2}^{i\bar j}\partial_i\partial_{\bar j}g_{\varphi_2}^{k\bar l}\partial_k\partial_{\bar l}v|\no\\
&\leq& C(\te, \Te)|\Na^2v|\Big(|\Na^4(\varphi_1-\varphi_2)|
+|\Na^3(\varphi_1-\varphi_2)|+|\Na^2(\varphi_1-\varphi_2)|\Big), \no
\eeqn where we used (\ref{eq:A060}) and (\ref{eq:A062}). Moreover, using (\ref{eq:A060}) and (\ref{eq:A061}) we have
\beqn &&
I_2+I_3=2\Big|g_{\varphi_1}^{i\bar j}\partial_{\bar j}g_{\varphi_1}^{k\bar l}\partial_i\partial_k\partial_{\bar l}v-g_{\varphi_2}^{i\bar j}\partial_{\bar j}g_{\varphi_2}^{k\bar l}\partial_i\partial_k\partial_{\bar l}v\Big|\no\\
&\leq&2\Big|g_{\varphi_1}^{i\bar j}\partial_{\bar j}g_{\varphi_1}^{k\bar l}\partial_i\partial_k\partial_{\bar l}v-g_{\varphi_1}^{i\bar j}\partial_{\bar j}g_{\varphi_2}^{k\bar l}\partial_i\partial_k\partial_{\bar l}v\Big|+2\Big|g_{\varphi_1}^{i\bar j}\partial_{\bar j}g_{\varphi_2}^{k\bar l}\partial_i\partial_k\partial_{\bar l}v-g_{\varphi_2}^{i\bar j}\partial_{\bar j}g_{\varphi_2}^{k\bar l}\partial_i\partial_k\partial_{\bar l}v\Big|\no\\
&\leq&C(\te, \Te)|\Na^3v|\Big(|\Na^3(\varphi_1-\varphi_2)|+|\Na^2(\varphi_1-\varphi_2)|\Big).\label{eq:A063}
\eeqn
and
\beqn
I_4&=&\Big|g_{\varphi_1}^{i\bar j}g_{\varphi_1}^{k\bar l}\partial_i\partial_{\bar j}\partial_k\partial_{\bar l}v-g_{\varphi_2}^{i\bar j}g_{\varphi_2}^{k\bar l}\partial_i\partial_{\bar j}\partial_k\partial_{\bar l}v\Big|\no\\
&\leq&\Big|g_{\varphi_1}^{i\bar j}g_{\varphi_1}^{k\bar l}\partial_i\partial_{\bar j}\partial_k\partial_{\bar l}v-g_{\varphi_2}^{i\bar j}g_{\varphi_1}^{k\bar l}\partial_i\partial_{\bar j}\partial_k\partial_{\bar l}v\Big|+\Big|g_{\varphi_2}^{i\bar j}g_{\varphi_1}^{k\bar l}\partial_i\partial_{\bar j}\partial_k\partial_{\bar l}v-g_{\varphi_2}^{i\bar j}g_{\varphi_2}^{k\bar l}\partial_i\partial_{\bar j}\partial_k\partial_{\bar l}v\Big|\no\\
&\leq&C(\te, \Te)|\Na^4v| \cdot |\Na^2(\varphi_1-\varphi_2)|.\label{eq:A064}
\eeqn
Combining the inequalities (\ref{eq:A051})-(\ref{eq:A064}), we have
\beqn
&&
|\Delta_{\varphi_1}^2v-\Delta_{\varphi_2}^2v|\no\\
&\leq&C(\te, \Te)\Big(|\Na^2v|\cdot |\Na^4(\varphi_1-\varphi_2)|
+|\Na^2v|\cdot |\Na^3(\varphi_1-\varphi_2)|+|\Na^2v|\cdot |\Na^2(\varphi_1-\varphi_2)|\no\\
&&+|\Na^3v|\cdot |\Na^3(\varphi_1-\varphi_2)|+|\Na^3v|\cdot|\Na^2(\varphi_1-\varphi_2)|+|\Na^4v|\cdot |\Na^2(\varphi_1-\varphi_2)|\Big). \no
\eeqn The lemma is proved.
\end{proof}

\begin{lem}\label{lem011}For any $\varphi_1, \varphi_2\in \cH(\oo_g)$ with
\beq
\oo_{\varphi_1}\geq \te\oo_g, \quad \oo_{\varphi_1}\geq \te\oo_g,\quad \|\varphi_1\|_{C^4(M, g)}\leq \Te, \quad \|\varphi_2\|_{C^4(M, g)}\leq \Te,\no
\eeq
 we have
\beq
\Big|Ric_{\varphi_1, i\bar j
}v_{ j\bar i}-Ric_{\varphi_2, i\bar j
}v_{ j\bar i}\Big|\leq C(\te, \Te)\Big(|\Na^2(\varphi_1- \varphi_2)|+|\Na^3(\varphi_1- \varphi_2)|+|\Na^4(\varphi_1- \varphi_2)|\Big)\cdot |\Na^2 v|.\no
\eeq
\end{lem}

\begin{proof}By direct calculation, we have
\beqn &&
\Big|Ric_{\varphi_1, i\bar j
}v_{ j\bar i}-Ric_{\varphi_2, i\bar j
}v_{ j\bar i}\Big|=\Big|g_{\varphi_1}^{i\bar k}g_{\varphi_1}^{l\bar j}Ric_{\varphi_1, i\bar j
}v_{ l\bar k}-g_{\varphi_2}^{i\bar k}g_{\varphi_2}^{l\bar j}Ric_{\varphi_2, i\bar j
}v_{ l\bar k}\Big|\no\\
&\leq&\Big|g_{\varphi_1}^{i\bar k}g_{\varphi_1}^{l\bar j}Ric_{\varphi_1, i\bar j
}v_{ l\bar k}-g_{\varphi_2}^{i\bar k}g_{\varphi_1}^{l\bar j}Ric_{\varphi_1, i\bar j
}v_{ l\bar k}\Big|+\Big|g_{\varphi_2}^{i\bar k}g_{\varphi_1}^{l\bar j}Ric_{\varphi_1, i\bar j
}v_{ l\bar k}-g_{\varphi_2}^{i\bar k}g_{\varphi_2}^{l\bar j}Ric_{\varphi_1, i\bar j
}v_{ l\bar k}\Big|\no\\
&&+\Big|g_{\varphi_2}^{i\bar k}g_{\varphi_2}^{l\bar j}Ric_{\varphi_1, i\bar j
}v_{ l\bar k}-g_{\varphi_2}^{i\bar k}g_{\varphi_2}^{l\bar j}Ric_{\varphi_2, i\bar j
}v_{ l\bar k}\Big|\no\\
&\leq&C(\te, \Te)\Big(|\Na^2(\varphi_1-\varphi_2)|\cdot |\Na^2 v|+|Ric_{\varphi_1, i\bar j}-Ric_{\varphi_2, i\bar j}|\cdot |\Na^2 v|\Big).\label{eq:A069}
\eeqn
Note that
\beqs Ric_{\varphi, i\bar j}&=&
- \partial_i\partial_{\bar j}(\log \det (g_{k\bar l}+\varphi_{k\bar l}))=- \partial_i\Big(g_\varphi^{k\bar l}\partial_{\bar j} g_{k\bar l}+g_\varphi^{k\bar l}\partial_{\bar j}\partial_k\partial_{\bar l}\varphi\Big)\\
&=&- \Big(\partial_ig_\varphi^{k\bar l}\partial_{\bar j} g_{k\bar l}+g_\varphi^{k\bar l}\partial_i\partial_{\bar j} g_{k\bar l}+\partial_ig_\varphi^{k\bar l}\partial_{\bar j}\partial_k\partial_{\bar l}\varphi+g_\varphi^{k\bar l}\partial_i\partial_{\bar j}\partial_k\partial_{\bar l}\varphi\Big)\\
&=&g_\varphi^{k\bar q}g_\varphi^{p\bar l}(\partial_ig_{ p\bar q}+\partial_i\partial_p\partial_{\bar q}\varphi)\partial_{\bar j} g_{k\bar l}-g_\varphi^{k\bar l}\partial_i\partial_{\bar j} g_{k\bar l}\\&&+g_\varphi^{k\bar q}g_\varphi^{p\bar l}(\partial_ig_{ p\bar q}+\partial_i\partial_p\partial_{\bar q}\varphi)\partial_{\bar j}\partial_k\partial_{\bar l}\varphi-g_\varphi^{k\bar l}\partial_i\partial_{\bar j}\partial_k\partial_{\bar l}\varphi\\
&=& g_\varphi^{k\bar q}g_\varphi^{p\bar l}\partial_ig_{ p\bar q}\partial_{\bar j} g_{k\bar l}+
g_\varphi^{k\bar q}g_\varphi^{p\bar l}\partial_i\partial_p\partial_{\bar q}\varphi\partial_{\bar j} g_{k\bar l}+g_\varphi^{k\bar q}g_\varphi^{p\bar l}\partial_ig_{ p\bar q}\partial_{\bar j}\partial_k\partial_{\bar l}\varphi\\&&+g_\varphi^{k\bar q}g_\varphi^{p\bar l}\partial_i\partial_p\partial_{\bar q}\varphi\partial_{\bar j}\partial_k\partial_{\bar l}\varphi-g_\varphi^{k\bar l}\partial_i\partial_{\bar j} g_{k\bar l}-g_\varphi^{k\bar l}\partial_i\partial_{\bar j}\partial_k\partial_{\bar l}\varphi.
\eeqs Thus, we can write
\beq
Ric_{\varphi_1, i\bar j}-Ric_{\varphi_2, i\bar j}:=\sum_{i=1}^6I_i. \label{eq:A066}
\eeq where $I_i$ can be estimated by
\beqn
|I_1|&=& | g_{\varphi_1}^{k\bar q}g_{\varphi_1}^{p\bar l}\partial_ig_{ p\bar q}\partial_{\bar j} g_{k\bar l}-g_{\varphi_2}^{k\bar q}g_{\varphi_2}^{p\bar l}\partial_ig_{ p\bar q}\partial_{\bar j} g_{k\bar l}|
\leq| g_{\varphi_1}^{k\bar q}g_{\varphi_1}^{p\bar l}\partial_ig_{ p\bar q}\partial_{\bar j} g_{k\bar l}-g_{\varphi_2}^{k\bar q}g_{\varphi_1}^{p\bar l}\partial_ig_{ p\bar q}\partial_{\bar j} g_{k\bar l}|\no\\
&&+| g_{\varphi_2}^{k\bar q}g_{\varphi_1}^{p\bar l}\partial_ig_{ p\bar q}\partial_{\bar j} g_{k\bar l}-g_{\varphi_2}^{k\bar q}g_{\varphi_2}^{p\bar l}\partial_ig_{ p\bar q}\partial_{\bar j} g_{k\bar l}|
\leq C(\te, \Te)|\Na^2(\varphi_1-\varphi_2)|,
\eeqn  and
\beqn &&
|I_2|=|g_{\varphi_1}^{k\bar q}g_{\varphi_1}^{p\bar l}\partial_i\partial_p\partial_{\bar q}\varphi_1\partial_{\bar j} g_{k\bar l}-g_{\varphi_2}^{k\bar q}g_{\varphi_2}^{p\bar l}\partial_i\partial_p\partial_{\bar q}{\varphi_2}\partial_{\bar j} g_{k\bar l}|\no\\&
\leq&|g_{\varphi_1}^{k\bar q}g_{\varphi_1}^{p\bar l}\partial_i\partial_p\partial_{\bar q}\varphi_1\partial_{\bar j} g_{k\bar l}-g_{\varphi_2}^{k\bar q}g_{\varphi_1}^{p\bar l}\partial_i\partial_p\partial_{\bar q}{\varphi_1}\partial_{\bar j} g_{k\bar l}|+|g_{\varphi_2}^{k\bar q}g_{\varphi_1}^{p\bar l}\partial_i\partial_p\partial_{\bar q}\varphi_1\partial_{\bar j} g_{k\bar l}\no\\&&-g_{\varphi_2}^{k\bar q}g_{\varphi_2}^{p\bar l}\partial_i\partial_p\partial_{\bar q}{\varphi_1}\partial_{\bar j} g_{k\bar l}|+|g_{\varphi_2}^{k\bar q}g_{\varphi_2}^{p\bar l}\partial_i\partial_p\partial_{\bar q}\varphi_1\partial_{\bar j} g_{k\bar l}-g_{\varphi_2}^{k\bar q}g_{\varphi_2}^{p\bar l}\partial_i\partial_p\partial_{\bar q}{\varphi_2}\partial_{\bar j} g_{k\bar l}|\no\\&\leq& C(\te, \Te)\Big(|\Na^2(\varphi_1-\varphi_2)|+|\Na^3(\varphi_1-\varphi_2)|\Big).
\eeqn Similarly, we have
\beqn
\sum_{i=3}^6|I_i|&\leq&|g_{\varphi_1}^{k\bar q}g_{\varphi_1}^{p\bar l}\partial_ig_{ p\bar q}\partial_{\bar j}\partial_k\partial_{\bar l}\varphi_1-g_{\varphi_2}^{k\bar q}g_{\varphi_2}^{p\bar l}\partial_ig_{ p\bar q}\partial_{\bar j}\partial_k\partial_{\bar l}\varphi_2|\no\\
&&+|g_{\varphi_1}^{k\bar q}g_{\varphi_1}^{p\bar l}\partial_i\partial_p\partial_{\bar q}\varphi_1\partial_{\bar j}\partial_k\partial_{\bar l}\varphi_1-g_{\varphi_2}^{k\bar q}g_{\varphi_2}^{p\bar l}\partial_i\partial_p\partial_{\bar q}\varphi_2\partial_{\bar j}\partial_k\partial_{\bar l}\varphi_2|\no\\
&&+|g_{\varphi_1}^{k\bar l}\partial_i\partial_{\bar j} g_{k\bar l} -g_{\varphi_2}^{k\bar l}\partial_i\partial_{\bar j} g_{k\bar l}     |+|g_{\varphi_1}^{k\bar l}\partial_i\partial_{\bar j}\partial_k\partial_{\bar l}\varphi_1-g_{\varphi_2}^{k\bar l}\partial_i\partial_{\bar j}\partial_k\partial_{\bar l}\varphi_2    |\no\\
&\leq&C(\te, \Te)\Big(|\Na^2(\varphi_1-\varphi_2)|+
|\Na^3(\varphi_1-\varphi_2)|+|\Na^4(\varphi_1-\varphi_2)|\Big).\label{eq:A067}
\eeqn
Combining (\ref{eq:A066})-(\ref{eq:A067}), we have
\beq
|Ric_{\varphi_1, i\bar j}-Ric_{\varphi_2, i\bar j}|\leq C( \te, \Te)\Big(|\Na^2(\varphi_1-\varphi_2)|+|\Na^3(\varphi_1-\varphi_2)|+
|\Na^4(\varphi_1-\varphi_2)|\Big).\label{eq:A068}
\eeq Therefore, by (\ref{eq:A068}) and (\ref{eq:A069}) we have
\beqs &&
\Big|Ric_{\varphi_1, i\bar j
}v_{ j\bar i}-Ric_{\varphi_2, i\bar j
}v_{ j\bar i}\Big|\\&
\leq& C( \te, \Te)\Big(|\Na^2(\varphi_1-\varphi_2)|+|\Na^3(\varphi_1- \varphi_2)|+|\Na^4(\varphi_1-\varphi_2)|\Big)\cdot |\Na^2 v|.
\eeqs The lemma is proved.
\end{proof}

\begin{lem}\label{lem013}Let $\td \varphi$ the function and $s_0$ the constant in Lemma \ref{lem:A004},  and $f\in C^{0, 0, \ga}(M\times [0, T], g)$. Define the operator \beqn\Psi_f :  C^{4, 1, \ga}(M\times [0, T], g, \psi_0)&\ri& C^{4, 1, \ga}(M\times [0, T], g, \psi_0)\no\\
  \varphi&\ri& \varphi+(DL_s|_{\td \varphi})^{-1}(f-L_s(\varphi)). \no
\eeqn There exist $\al=\al(\ga, n)>1$ and $\dd=\dd(g, n, \te, \Te, T)>0$ satisfying the following properties.
If $h_1(x, t) $ and $ h_2(x, t)$ satisfy
\beq
\|h_1(x, t)-\td \varphi\|_{C^{4, 1, \ga}(M\times [0, T], g)}\leq s^{\al}\dd, \quad \|h_2(x, t)-\td \varphi\|_{C^{4, 1, \ga}(M\times [0, T], g)}\leq s^{\al}\dd,\label{eq:A021}
\eeq then for any $s\in (0, s_0)$ we have
\beq
\|\Psi_f(h_1)-\Psi_f(h_2)\|_{C^{4, 1, \ga}(M\times [0, T], g)}\leq \frac 12 \|h_1-h_2\|_{C^{4, 1, \ga}(M\times [0, T], g)}.\no
\eeq
\end{lem}

\begin{proof}  Let $h_{\xi}=\xi h_1+(1-\xi)h_2$. Then
\beqn &&
\Psi_f(h_1)-\Psi_f(h_2)=\int_0^1\,\pd {}{\xi}\Psi_f(h_{\xi})\,d\xi\no\\
&=&h_1-h_2-\int_0^1\,(DL_s|_{\td \varphi})^{-1}DL_s|_{h_s}(h_1-h_2)\,d\xi\no\\
&=&-\int_0^1\,(DL_s|_{\td \varphi})^{-1}\Big(DL_s|_{h_s}-DL_s|_{\td \varphi}\Big)(h_1-h_2)\,d\xi.\label{eq:A049a}
\eeqn Let $v=h_1-h_2$.  Note that
\beqn
\Big(DL_s|_{h_s}-DL_s|_{\td \varphi}\Big)v&=&
s\Big(\Delta^2_{h_s}-\Delta^2_{\td \varphi}\Big)v+s\Big(Ric_{h_s, i\bar j
}v_{ j\bar i}-Ric_{\td\varphi, i\bar j
}v_{ j\bar i}\Big)\no\\&&-(1-s)\Big(g_{h_s}^{i\bar l}g_{h_s}^{k\bar j}g_{i\bar j}v_{ k\bar l}-g_{\td\varphi}^{i\bar l}g_{\td\varphi}^{k\bar j}g_{i\bar j}v_{ k\bar l}\Big).\no
\eeqn
By Lemma \ref{lem010} and Lemma \ref{lem011}, we have
\beqn &&
\Big|\Big(DL_s|_{h_s}-DL_s|_{\td \varphi}\Big)v\Big|\no\\&\leq &
s\Big|\Delta^2_{h_s}v-\Delta^2_{\td \varphi}v\Big|+s\Big|Ric_{h_s, i\bar j
}v_{ j\bar i}-Ric_{\td\varphi, i\bar j
}v_{ j\bar i}\Big|+(1-s)\Big|g_{h_s}^{i\bar l}g_{h_s}^{k\bar j}g_{i\bar j}v_{ k\bar l}-g_{\td\varphi}^{i\bar l}g_{\td\varphi}^{k\bar j}g_{i\bar j}v_{ k\bar l}\Big|\no\\
&\leq&s C(\te, \Te)\Big(|\Na^2v|\cdot |\Na^4(h_s-\td \varphi)|
+|\Na^2v|\cdot |\Na^3(h_s-\td \varphi)|+|\Na^2v|\cdot |\Na^2(h_s-\td \varphi)|\no\\
&&+|\Na^3v|\cdot |\Na^3(h_s-\td \varphi)|+|\Na^3v|\cdot|\Na^2(h_s-\td \varphi)|+|\Na^4v|\cdot |\Na^2(h_s-\td \varphi)|\Big)\no\\
&&+(1-s)C(\te, \Te)|\Na^2v|\cdot |\Na^2(h_s-\td \varphi)|.\label{eq:A019}
\eeqn
Therefore, we have
\beqn &&
\Big\|\Big(DL_s|_{h_s}-DL_s|_{\td \varphi}\Big)v\Big\|_{C^{0, 0, \ga}(M\times [0, T], g)}\no\\ &\leq&
s C(\te, \Te)\Big(\|\Na^2v\|_{C^{0, 0, \ga}(M\times [0, T], g)}\cdot \|\Na^4(h_s-\td \varphi)\|_{C^{0, 0, \ga}(M\times [0, T], g)}\no\\&&
+\|\Na^2v\|_{C^{0, 0, \ga}(M\times [0, T], g)}\cdot \|\Na^3(h_s-\td \varphi)\|_{C^{0, 0, \ga}(M\times [0, T], g)}\no\\&&+\|\Na^2v\|_{C^{0, 0, \ga}(M\times [0, T], g)}\cdot \|\Na^2(h_s-\td \varphi)\|_{C^{0, 0, \ga}(M\times [0, T], g)}\no\\
&&+\|\Na^3v\|_{C^{0, 0, \ga}(M\times [0, T], g)}\cdot \|\Na^3(h_s-\td \varphi)\|_{C^{0, 0, \ga}(M\times [0, T], g)}\no\\&&+\|\Na^3v\|_{C^{0, 0, \ga}(M\times [0, T], g)}\cdot\|\Na^2(h_s-\td \varphi)\|_{C^{0, 0, \ga}(M\times [0, T], g)}\no\\&&+\|\Na^4v\|_{C^{0, 0, \ga}(M\times [0, T], g)}\cdot \|\Na^2(h_s-\td \varphi)\|_{C^{0, 0, \ga}(M\times [0, T], g)}\Big)\no\\
&&+(1-s)C(\te, \Te)\|\Na^2v\|_{C^{0, 0, \ga}(M\times [0, T], g)}\cdot \|\Na^2(h_s-\td \varphi)\|_{C^{0, 0, \ga}(M\times [0, T], g)}\no\\
&\leq& C(\te, \Te)\|v\|_{C^{4, 1, \ga}(M\times [0, T], g)}\cdot \|h_s-\td \varphi\|_{C^{4, 1, \ga}(M\times [0, T], g)}. \label{eq:A020}
\eeqn
Combining Lemma \ref{lem004} with (\ref{eq:A021}) and (\ref{eq:A020}),  we have
\beqn &&
\Big\|(DL_s|_{\td \varphi})^{-1}\Big(DL_s|_{h_s}-DL_s|_{\td \varphi}\Big)(h_1-h_2)\Big\|_{C^{4, 1, \ga}(M\times [0, T], g)}\no\\
&\leq& C(\te, \Te, n, g, T) s^{-\ka(\ga, n)} \Big\|\Big(DL_s|_{h_s}-DL_s|_{\td \varphi}\Big)(h_1-h_2)\Big\|_{C^{0, 0, \ga}(M\times [0, T], g)}\no\\
&\leq& C(\te, \Te, n, g, T) s^{-\ka(\ga, n)} \|h_s-\td \varphi\|_{C^{4, 1, \ga}(M\times [0, T], g)}\cdot  \|h_1-h_2\|_{C^{4, 1, \ga}(M\times [0, T], g)}\no\\
&\leq&C(\te, \Te, n, g, T)\dd s^{\al-\ka(\ga, n)}  \|h_1-h_2\|_{C^{4, 1, \ga}(M\times [0, T], g)}. \label{eq:A022z}
\eeqn
Let $\al=\ka(\ga, n)$ and we choose $\dd$ such that $C(\te, \Te, n, g, T)\dd<\frac 12$. By (\ref{eq:A022z}) and (\ref{eq:A049a}) we have
\beq
\|\Psi_f(h_1)-\Psi_f(h_2)\|_{C^{4, 1, \ga}(M\times [0, T], g)}\leq \frac 12 \|h_1-h_2\|_{C^{4, 1, \ga}(M\times [0, T], g)}.\no
\eeq
The lemma is proved.

\end{proof}

\subsection{Proof of Theorem \ref{theo:001a} } \label{sec3}
In this subsection, we prove Theorem \ref{theo:001a}. The proof is  divided into several steps.

(1). Fix $\psi_0\in \cH(\oo_g).$
  We denote by $\varphi_0(x, t)(t\in [0, \infty))$ the solution of $J$-flow with the initial data $\psi_0,$ i.e. $L_0(\varphi(x, t))=0$.   By Lemma \ref{lem007}, for any integer $N>0$ there exists $\{u_{j}(x, t)\}_{j=1}^N$ on $M\times [0, T]$ such that the function $\td \varphi:=\td \varphi_{N, s}(x, t) $ defined by (\ref{eq:A026}) satisfies the properties
\beqn
\oo_{\td \varphi}\geq \te\oo_g, \quad \|\td \varphi\|_{C^{N'}(M\times [0, T], g)}&\leq& \Te,\\
\|L_s(\td \varphi_{N, s}(x, t))\|_{C^{0, 0, \ga}(M\times [0, T], g)}&\leq& C(\te, \Te, T)s^{N+1}, \label{eq:A210}\\
\|\td \varphi_{N, s}-\varphi_0(x, t)\|_{C^{4+\ga}(M\times [0, T], g)}&\leq&C(\te,
\Te, T)s. \label{eq:A031}
\eeqn Moreover, by the construction of $\td \varphi_{N, s}$ in Lemma  \ref{lem007} we have that $\td \varphi_{N, s}(x, 0)=\varphi_0(x, 0) =\psi_0$.

(2). By Lemma \ref{lem004}, there exists $s_0(\te, \Te)>0$ such that for any $s\in (0, s_0)$,
\beq
DL_s|_{\td \varphi}:  C^{4, 1, \ga}(M\times [0, T], g, 0)\ri C^{0, 0, \ga}(M\times [0, T], g).\no
\eeq
is injective and surjective. Moreover, we have
\beq
\|(DL_s|_{\td \varphi})^{-1}\|\leq C(\te, \Te, n, g, T) s^{-\ka(\ga, n)}.\no
\eeq where $\ka(\ga, n)>1$.
By Lemma \ref{lem013}, for any $f\in C^{0, 0, \ga}(M\times [0, T], g)$ the map \beqn\Psi_f :  C^{4, 1, \ga}(M\times [0, T], g, \psi_0)&\ri& C^{4, 1, \ga}(M\times [0, T], g, \psi_0)\no\\
  \varphi&\ri& \varphi+(DL_s|_{\td \varphi})^{-1}(f-L_s(\varphi)).\no
\eeqn is a contraction and satisfies
\beq
\|\Psi_f(h_1)-\Psi_f(h_2)\|_{C^{4, 1, \ga}(M\times [0, T], g)}\leq \frac 12 \|h_1-h_2\|_{C^{4, 1, \ga}(M\times [0, T], g)}, \no
\eeq where $h_1(x, t) $ and $ h_2(x, t)$ satisfy
\beq
\|h_1(x, t)-\td \varphi\|_{C^{4, 1, \ga}(M\times [0, T], g)}\leq s^{\al}\dd, \quad \|h_2(x, t)-\td \varphi\|_{C^{4, 1, \ga}(M\times [0, T], g)}\leq s^{\al}\dd. \label{eq:A029z}
\eeq  Here $\al=\al(\ga, n)>1$ and $\dd=\dd(g, n, \te, \Te, T)>0$. \\

(3).
We next show that    for any $s\in (0, s_0)$, if $v\in C^{0, 0, \ga}(M\times [0, T], g)$ satisfying \beq \|v-L_s(\td \varphi)\|_{C^{0, 0, \ga}(M\times [0, T], g)}\leq \eta s^{\bb} \label{eq:A030y}\eeq for some $\bb(\ga, n)>0$,  we can find $\varphi\in  C^{4, 1, \ga}(M\times [0, T], g, \psi_0)$ with $L_s(\varphi)=v. $
For any $k\in \NN$ we define
\beq
\varphi_k=\Psi_v^{k-1}(\td \varphi).\no
\eeq
Then $\varphi_1=\td \varphi$ and  we have
\beqs
\|\varphi_2-\varphi_1\|_{C^{4, 1, \ga}(M\times [0, T], g)}&=&
\|\varphi_2-\td \varphi\|_{C^{4, 1, \ga}(M\times [0, T], g)}\\&=&\|(DL_s|_{\td \varphi})^{-1}(v-L_s(\td \varphi))\|_{C^{4, 1, \ga}(M\times [0, T], g)}\\
&\leq&C(\te, \Te, n, g, T) s^{-\ka}\|v-L_s(\td \varphi)\|_{C^{0, 0, \ga}(M\times [0, T], g)}\\
&\leq &C(\te, \Te, n, g, T) \eta s^{\bb-\ka}\leq\frac 12 \dd s^{\al},
\eeqs where we choose
$
\bb-\ka=\al
$ and $\eta$ satisfying $C(\te, \Te, n, g, T) \eta<\frac 12 \dd. $
We claim that for any $k\geq 1$ we have
\beq
\|\varphi_{k+1}-\varphi_{k}\|_{C^{4, 1, \ga}(M\times [0, T], g)}\leq \frac {1}{2^{k}}\dd s^{\al}. \label{eq:A028z}
\eeq
By induction, we assume that for $k=1, 2,\cdots, j-1$ the inequality (\ref{eq:A028z}) holds. Then for $k=j$ we have
\beq
\|\varphi_{j}-\varphi_1\|_{C^{4, 1, \ga}(M\times [0, T], g)}\leq\sum_{i=1}^{j-1}\,
\|\varphi_{i+1}-\varphi_i\|_{C^{4, 1, \ga}(M\times [0, T], g)}\leq\sum_{i=1}^{j-1}\,
\frac 1{2^{i}} \dd s^{\al}\leq \dd s^{\al}. \label{eq:H032}
\eeq
Thus, $\varphi_j$ satisfies the condition (\ref{eq:A029z}) and we have
\beqn &&
\|\varphi_{j+1}-\varphi_j\|_{C^{4, 1, \ga}(M\times [0, T], g)}=\|\Psi(\varphi_j)-\Psi(\varphi_{j-1})\|_{C^{4, 1, \ga}(M\times [0, T], g)}\no\\&\leq &
\frac 12 \|\varphi_j-\varphi_{j-1}\|_{C^{4, 1, \ga}(M\times [0, T], g)}\leq \frac 1{2^j} \dd s^{\al}.\no
\eeqn Thus, (\ref{eq:A028z}) is proved. Therefore, $\varphi_k$ converges to a limit
function $\varphi_{\infty}$ in $C^{4, 1, \ga}(M\times [0, T], g, \psi_0)$ and $L_{s}(\varphi_{\infty})=v$. By (\ref{eq:H032}), we have
\beq
\|\varphi_{\infty}-\td \varphi\|_{C^{4, 1, \ga}(M\times [0, T], g)}\leq \dd s^{\al}. \label{eq:A031a}
\eeq

(4).
 By (\ref{eq:A210}), for any  given integer $N>0$ the function $\td \varphi:=\td \varphi_{N, s}$ satisfies 
\beq
\|L_s(\td \varphi)\|\leq C(\te, \Te, T) s^N.\no
\eeq Taking $N>\bb(\ga, n)$, the function $v=0$ satisfies the condition  (\ref{eq:A030y}). Therefore, we can find $\varphi_{\infty}(x, t)\in C^{4, 1, \ga}(M\times [0, T], g, \psi_0)$ such that
$$L_s(\varphi_{\infty})=0.$$
Moreover, (\ref{eq:A031}) and (\ref{eq:A031a}) imply that
\beqn
\|\varphi_{\infty}-\varphi_0\|_{C^{4, 1, \ga}(M\times [0, T], g)}&\leq&
\|\varphi_{\infty}-\td \varphi\|_{C^{4, 1, \ga}(M\times [0, T], g)}+\|\varphi_0-\td \varphi\|_{C^{4, 1, \ga}(M\times [0, T], g)}\no\\
&\leq&\dd(g, n, \te, \Te, T) s^{\al}+C(\te,
\Te, T)s \no\\&\leq& C(g, n, \te,
\Te, T)s, \no
\eeqn where we used  $\al>1$ in the last inequality.
The theorem is proved.

\section{Stability of twisted Calabi flow for small $s$}
In this section, we prove
\begin{theo}\label{theo:002}
Let $(M, g)$ be a compact K\"ahler manifold with a cscK metric
$\oo_g$.  There exist $\dd>0 $ and $ s_0>0$ such that for any $\psi_0\in \cH(\oo_g)$  satisfying
\beq
\|\psi_0\|_{C^{4}(M, g)}\leq \dd,\no
\eeq the twisted Calabi flow (\ref{eq:000}) for any $s\in (0, s_0)$ with the initial K\"ahler potential $\psi_0$ exists for all time and converges exponentially to the cscK metric $\oo_g$.

\end{theo}

Theorem \ref{theo:002} shows that there exists a uniform neighbourhood near the cscK metric such that the twisted Calabi flow for any small $s$ with initial data in the neighbourhood will exist for all time and converge smoothly. Note that in
\cite{[ChenHe1]} Chen-He showed that there exists a   neighbourhood near a cscK metric such that the  Calabi flow   with initial data in the neighbourhood will exist for all time and converge smoothly, and in \cite{[HuangZheng]} Huang-Zheng proved similar stability results for Calabi flow near extremal K\"ahler metrics.
 By using Chen-He's method, the authors in \cite{[HL]} show similar stability theorem for twisted Calabi flow for fixed $s\in (0, 1)$. However, the neighbourhood of the initial K\"ahler potentials will shrink to a point as $s\ri 0$ so Chen-He's argument doesn't work for Theorem \ref{theo:002}.

The proof of Theorem \ref{theo:002} is similar to that of Theorem \ref{theo:001a}.
Instead of using usual H\"older spaces, we introduce   weighted H\"older spaces with functions decaying exponentially. We show the linearized operator of twisted Calabi flow is invertible between weighted H\"older spaces and we construct a contraction map between weighted H\"older spaces, which implies the existence of twisted Calabi flow by the fixed point theory.

\subsection{The exponential decay of $J$-flow}

Consider the $J$-flow
\beq
\pd {\varphi}t=n-\tr_{\varphi}\oo_g,\quad \varphi(0)=\psi_0. \label{eq:B001z}
\eeq We normalize $\varphi(x, t)$ such that $I(\varphi)=0$, where the functional $I(\varphi)$ is defined by
\beq
I(\varphi)=\int_0^1\,\int_M\;\pd {\psi_t}t\,\oo_{\psi_t}^ndt.\no
\eeq Here $\psi_t\in \cH(\oo_g)(t\in [0, 1])$ is a smooth path connecting $0$ and $\varphi$. It is showed by Chen \cite{[Chen-Mabuchi]}  that the $J$ functional
\beq
J(\varphi)=-\int_0^1\,\int_M\;\pd {\psi_t}t\,\oo_g\wedge \oo_{\psi_t}^n\no
\eeq
is convex in any $C^{1, 1}$ geodesic and $J$ has at most one critical point in $\cH_0(\oo_g):=\cH(\oo_g)\cap\{I(\varphi)=0\}$.
Since the equation $\tr_{\varphi}\oo_g=n$ has a solution $\varphi=0$ in $\cH_0(\oo_g)$,  the solution  of the flow (\ref{eq:B001z}) on $M\times [0, \infty)$, which we denote by $\varphi_0(x, t)$,   exists for all time and converges smoothly to $0$ by Weinkove \cite{[Wenk2]}. Therefore, for any $\ee_0>0$ there exists $t_1>0$ such that
\beq
\|\varphi_0(x, t)\|_{C^k(M, \oo_g)}\leq \ee_0,\quad \forall\; t\geq t_1. \no
\eeq

The next result gives some conditions which will be used to show the exponential decay of solutions of second-order parabolic equations.

\begin{lem}\label{lem:B004a}Let $ k>0$ be an integer and $\ga\in (0, 1), \La_1, \La_2>0$. There exist integer $N(k)\geq 4$,  and constants $\ee_0(g)>0, \eta_0(g)>0$ satisfying the following properties.      Suppose that
\begin{enumerate}
  \item[(1).] $\varphi_0(x, t)(t\in [0, \infty))$ is a family of K\"ahler potentials satisfying
      \beq
      \|\varphi_0(x, t)\|_{C^N(M\times [0, \infty), \oo_g)}\leq \ee_0, \label{eq:A101}
      \eeq

\item[(2).] The function $h(x, t)$ satisfies the condition for $\eta\in (0, \eta_0)$
\beq
\|h(\cdot, t)\|_{C^k(M\times [t, t+1], g)}\leq \La_1 e^{-\eta t}, \quad \forall\;t\geq 0,\label{eq:A103}
\eeq
\item[(3).] $v(x, t)$ satisfies the equation
\beqn
\pd vt&=&b_{j\bar i}v_{i\bar j}+h,\quad \forall\; (x, t)\in M\times [0, \infty),\label{eq:B009}\\
v(x, 0)&=&v_0,\quad \|v_0\|_{C^k(M, g)}\leq \La_2, \label{eq:A400}
\eeqn where $b_{i\bar j}=g_{\varphi_0}^{i\bar l}g_{\varphi_0}^{k\bar j}g_{i\bar j}$.

\end{enumerate}
 Then
there exists a constant $v_{\infty}\in \RR$ such that the solution $\td v(x, t)=
v(x, t)-v_{\infty}$ of (\ref{eq:B009})  satisfies
\beq
\|\td v\|_{C^{k+\ga}(M\times [t, t+1], g)}\leq C(n, g, \eta, k, t_0)(\La_1+\La_2) e^{-\eta t}, \quad \forall \;t\geq 0.\no
\eeq

\end{lem}
\begin{proof}(1). By the assumption (\ref{eq:A101}), $b_{i\bar j}=g_{\varphi_0}^{i\bar l}g_{\varphi_0}^{k\bar j}(g_{\varphi_0, i\bar j}-\varphi_{0, i\bar j})$ satisfies
 \beq
b_{i\bar j}\geq \te g_{\varphi_0, i\bar j},\quad |\Na_{\varphi_0}b_{i\bar j}|\leq |\Na^3\varphi_0|\leq \ee_0, \label{eq:A300}
\eeq where $\te>0.$
  Let $\un v(t)=\int_M\;v\,\oo_{\varphi_0}^n$ and $\td v=v-\un v(t).$ Using (\ref{eq:A101}) and (\ref{eq:A300}), we have
\beqn
\frac d{dt}\int_M\;\td v^2\,\oo_{\varphi_0}^n
&=& \int_M\;\Big(2\td vb_{l\bar k}\td v_{k\bar l}+2\td vh+\td v^2\Delta_{\varphi_0}\pd {\varphi_0}t\Big)\,\oo_{\varphi_0}^n\no\\
&\leq&\int_M\;\Big(-2b_{l\bar k}\td v_{\bar l}\td v_k-2b_{l\bar k, \bar l}\td v\td v_{k}+2\td vh\Big)\,\oo_{\varphi_0}^n+\ee_0\int_M\;\td v^2\,\oo_{\varphi_0}^n\no\\
&\leq&-2\te\int_M\,|\Na \td v|^2\,\oo_{\varphi_0}^n+\ee_0\int_M\;|\Na \td v||\td v|\,\,\oo_{\varphi_0}^n\no\\
&&+2\|h\|_{L^2(M, \oo_{\varphi_0})}\|\td v\|_{L^2(M, \oo_{\varphi_0})}+\ee_0\int_M\;\td v^2\,\oo_{\varphi_0}^n. \label{eq:A301}
\eeqn
By the assumption (\ref{eq:A101}), we have the Poincar\'e inequality
\beq \int_M\;|\Na \td v|^2\,\oo_{\varphi_0}^n\geq \la \int_M\;| \td v|^2\,\oo_{\varphi_0}^n, \label{eq:A302}
\eeq where the constant $\la=\la(g)>0$. Therefore,  (\ref{eq:A301}) and (\ref{eq:A302}) imply that
\beqn
\frac d{dt}\int_M\;\td v^2\,\oo_{\varphi_0}^n&\leq& -(2\te-\frac 12\ee_0)\int_M\;|\Na \td v|^2\,\oo_{\varphi_0}^n+2\ee_0\int_M\;\td v^2\,\oo_{\varphi_0}^n\no\\&&+2\|h\|_{L^2(M, \oo_{\varphi_0})}\|\td v\|_{L^2(M, \oo_{\varphi_0})}\no\\
&\leq&-(2\la \te-C(\la)\ee_0)\int_M\;\td v^2\,\oo_{\varphi_0}^n+2\|h\|_{L^2(M, \oo_{\varphi_0})}\|\td v\|_{L^2(M, \oo_{\varphi_0})}.\no\\\label{eq:A102}
\eeqn
We  choose $\ee_0$  small such that $\eta_1:=\frac 12(2\la \te-C(\la)\ee_0)>0$. Let $E(t)=\|\td v\|_{L^2(M, \oo_{\varphi_0})}$. The inequality (\ref{eq:A102}) implies that
 \beq
E'(t)\leq -\eta_1 E(t)+\|h\|_{L^2(M, \oo_{\varphi_0})}.\no
\eeq
Combining this with (\ref{eq:A103}), we have
\beqn
E(t)&\leq& e^{-\eta_1 t}E(0)+e^{-\eta_1 t}\int_0^{t}\,e^{\eta_1 \tau}\|h(\cdot, \tau)\|_{L^2(M, \oo_{\varphi_0})}\,d\tau\no\\
&\leq&e^{-\eta_1 t}E(0)+\La_1 e^{-\eta_1 t}\int_0^{t}\,e^{(\eta_1-\eta) \tau}\,d\tau\no\\
&\leq&e^{-\eta_1 t}E(0)+\frac {\La_1}{\eta_1-\eta}e^{-\eta t}.\no
\eeqn Note that by the assumption (\ref{eq:A400}) we have
\beq
E(0)\leq \|v_0\|_{L^2(M, g)}+|\un v(0)|\leq 2\La_2.\no
\eeq
Thus, for $\eta\in (0, \eta_1)$ we have
\beq
\|\td v(\cdot, t)\|_{L^2(M, \oo_{\varphi_0})}\leq C( \eta_1, \eta)(E(0)+\La_1) e^{-\eta t}\leq C( \eta_1, \eta)(\La_1+\La_2) e^{-\eta t}.\label{eq:B033}
\eeq

(2). We estimate $\un v(t)$.  Direct calculation shows that
\beqn
\frac {d}{dt}\int_M\; v \,\oo_{\varphi_0}^n&=&
\int_M\;\Big( b_{i\bar j}v_{j\bar i}+h-v\Delta_{\varphi_0}\tr_{\varphi_0}\oo_g \Big)\,\oo_{\varphi_0}^n\no\\
&=&\int_M\;\Big( \Delta_{\varphi_0}v-\varphi_{0, i\bar j}v_{j\bar i}+h-v\Delta_{\varphi_0}\tr_{\varphi_0}(\oo_{\varphi_0}-\varphi_{0, i\bar j}) \Big)\,\oo_{\varphi_0}^n\no
\\&\leq&\int_M\;h
\,\oo_{\varphi_0}^n+C(g, \ee_0)\int_M\;|\td v|\,\oo_{\varphi_0}^n.\no
\eeqn where we used the assumption (\ref{eq:A101}) in the last inequality.
Using (\ref{eq:A103}) and (\ref{eq:B033}), we have
\beq
\Big|\frac {d}{dt}\int_M\; v \,\oo_{\varphi_0}^n\Big|\leq C( \eta_1, \eta, g, \ee_0)(\La_1+\La_2) e^{-\eta t}. \label{eq:B031}
\eeq
Since $\varphi_0(x, t)$ converges smoothly to $0$ as $t\ri +\infty$,
 there exists $v_{\infty}\in \RR$ such that   $\lim_{t\ri \infty}\int_M\; v \,\oo_{\varphi_0}^n=v_{\infty}$. Thus, (\ref{eq:B031}) implies that
\beq
|\un v(t)-v_{\infty}|\leq C( \eta_1, \eta, g, \ee_0)(\La_1+\La_2) e^{-\eta t}. \label{eq:B032}
\eeq Moreover, taking $t=0$ in (\ref{eq:B032}) and using (\ref{eq:A400}) we have
\beq
|v_{\infty}|\leq C( \eta_1, \eta, g, \ee_0)(\La_1+\La_2) e^{-\eta t}. \label{eq:I004}
\eeq
Combining (\ref{eq:B032}) with (\ref{eq:B033}), we have
\beqn
\| v(\cdot, t)-v_{\infty}\|_{L^2(M, \oo_g)}&\leq& \| \td v(\cdot, t)\|_{L^2(M, \oo_g)}+ |\un v(t)-v_{\infty}|\no\\&\leq& C( \eta_1, \eta, g, \ee_0)(\La_1+\La_2) e^{-\eta t}. \label{eq:I001}
\eeqn

(3). Let $\hat v(x, t):=v(x, t)-v_{\infty}$.   Since $\hat v(x, t)$ satisfies the equation (\ref{eq:B009}), by the maximal principle for any $t\geq 0$ we have
\beq
\|\hat v(x, t)\|_{L^{\infty}(M)}\leq \|\hat v(x, 0)\|_{L^{\infty}(M)}+\La_1 t
\leq 2\La_2+\La_1 t.
\eeq
Note that by (\ref{eq:A400}) and  (\ref{eq:I004}) $\hat v(x, 0)$ satisfies
\beq
\|\hat v(x, 0)\|_{C^{2+\ga}(M, g)}\leq \|v_0\|_{C^{2+\ga}(M, g)}+|v_{\infty}|\leq
C( \eta_1, \eta, g, \ee_0)(\La_1+\La_2).
\eeq Fix $t_1>0$.
By Theorem 4.28 of Lieberman \cite{[Lieb]}, we have
\beqn
\|\hat v\|_{C^{2+\ga}(M\times [0, t_1], g)}&\leq& C( \eta_1, \eta, g, \ee_0)\Big(\|\hat v(x, t)\|_{L^{\infty}(M\times [0, t_1])}+\|h\|_{C^{\ga}(M\times [0, t], g)}\no\\&&+\|\hat v(x, 0)\|_{C^{2+\ga}(M, g)}\Big)\no\\
&\leq&C( \eta_1, \eta, g, \ee_0, t_1)(\La_1+\La_2).  \label{eq:I002}
\eeqn
Moreover,  by Theorem \ref{theo:A4} and (\ref{eq:I001}) for $t\geq \frac {t_1}2$  we have
\beqn
\|\hat v(\cdot, t)\|_{L^{\infty}(M, g)}&\leq& C(\te, \ee_0, g, n, t_1) \|\hat v\|_{L^2(M\times [t-\frac {t_1}2, t], \oo_g)}\no\\&\leq& C(\te, \ee_0, g, n, \eta_1, \eta, t_1 )(\La_1+\La_2) e^{-\eta t}.\no
\eeqn
By Theorem \ref{theo:A3}, for $t\geq {t_1}$ we have
\beqn &&
\|\hat v\|_{C^{2+\ga}(M\times [t, t+1], g)}\no\\&\leq& C(\te, \ee_0, g, n)\Big(\max_{\tau\in [t-\frac {t_1}2, t+1]}\|\hat v(\cdot, \tau)\|_{L^{\infty}(M, g)}+\|h\|_{C^{\ga}(M\times [t-\frac {t_1}2, t+1], g)  }\Big)\no\\
&\leq& C(\te, \ee_0, g, n, \eta_1, \eta, t_1 )(\La_1+\La_2) e^{-\eta t}. \label{eq:B040z}
\eeqn
Combining (\ref{eq:I002}) with (\ref{eq:B040z}), we have  the inequality
\beq
\|\hat v\|_{C^{2+\ga}(M\times [t, t+1], g)}\leq C(\te, \ee_0, g, n, \eta_1, \eta, t_1 )(\La_1+\La_2) e^{-\eta t},\quad \forall\, t\geq 0. \label{eq:I003}
\eeq

(4).
Let $v_1=|\Na v|_{\varphi_0}^2$. We calculate the equation of $v_1$:
\beqn
\pd {v_1}t&=&v_{ i}(b_{l\bar k}v_{k\bar l}+h)_{\bar i}+v_{ \bar i}(b_{l\bar k}v_{k\bar l}+h)_{i}-\Big(\pd {\varphi_0}t\Big)_{i\bar j}v_iv_{\bar j}\no\\
&=&b_{l\bar k, \bar i}v_{k\bar l}v_{ i}+b_{l\bar k}v_{k\bar l\bar i}v_{i}+
b_{l\bar k, i}v_{k\bar l}v_{ \bar i}+b_{l\bar k}v_{k\bar li}v_{ \bar i}-\Big(\pd {\varphi_0}t\Big)_{i\bar j}v_iv_{\bar j}\no\\
&&+v_{ i}h_{\bar i}+v_{ \bar i}h_{ i}\no\\
&=&b_{l\bar k}(|\Na v|^2)_{k\bar l}-b_{l\bar k}(v_{ik}v_{\bar i\bar l}+v_{i\bar l}v_{\bar ik})-b_{l\bar k}R_{i\bar lk\bar j}v_jv_{ \bar i}+b_{l\bar k, \bar i}v_{k\bar l}v_{ i}\no\\&&+b_{l\bar k, i}v_{k\bar l}v_{ \bar i}-\Big(\pd {\varphi_0}t\Big)_{i\bar j}v_iv_{\bar j}+v_{ i}h_{\bar i}+v_{ \bar i}h_{ i}.\no
\eeqn
Therefore, we have
\beq
\pd {v_1}t= b_{l\bar k}v_{1, k\bar l}+ h_1,\no
\eeq where $h_1$ can be written as
\beq
  h_1:=\Na^2v*\Na^2v+T(\varphi_0)*\Na v*\Na v+\Na^3\varphi_0*\Na^2v*\Na v+\Na v*\Na h.\no
\eeq Here $T(\varphi_0)$ denotes a tensor depending only on $\varphi_0$.
By (\ref{eq:I003}) and the assumption (\ref{eq:A103}),  for $t\geq 0$ the function $h_1$ satisfies $\|h_1\|_{C^{\ga}(M\times [t, t+1], g)}\leq \La_1' e^{-2 \eta t}$ with $\La_1'=C(\te, \ee_0, g, n,  \eta_1, \eta)(\La_1+\La_2)^2$. Moreover, by (\ref{eq:I003}) we have
\beq
\|v_1(\cdot, t)\|_{L^{\infty}(M)}\leq \La_2'=C(\te, \ee_0, g, n, \eta_1, \eta )(\La_1+\La_2) e^{-2\eta t},\quad \forall\;t\geq 0.\no
\eeq
Following the same argument as in (2), we have
\beq
\| v_1\|_{C^{2+\ga}(M\times [t, t+1], g)}\leq C(\te, \ee_0, g, n, \eta_1, \eta, t_1)(\La_1+\La_2)^2e^{-2\eta t},\quad \forall\, t\geq 0. \no
\eeq
Similarly, we can estimate $|\Na^j v|(1\leq j\leq k)$, and we omit the details here.

\end{proof}

Using Lemma \ref{lem:B004a}, we show that any $C^k$ norm of the solution $\varphi_0$  of $J$-flow decays exponentially.

\begin{lem}\label{lem:B002} Let $k\in \NN$ and $\varphi_0(x, t)$  the solution of $J$-flow. There exists an integer $N(k)>0$ and  two constants $ \ee_0(g), \eta(g)>0$ such that if $\varphi_0(x, t)$ converges smoothly to $0$ as $t\ri +\infty$ and satisfies the condition \beq
      \|\varphi_0(x, t)\|_{C^N(M\times [0, \infty), \oo_g)}\leq \ee_0, \label{eq:A303}
      \eeq
then for $t\geq 0$ we have
\beq
\|\varphi_0(x, t)\|_{C^k(M, g)}\leq C(n, g, k,  t_0)\ee_0 e^{-\eta t}. \label{eq:B044}
\eeq
Moreover, if $\oo_g$ is a cscK metric, then for $t\geq 0$ we have
\beq
\|R(\varphi_0(x, t))-\un R\|_{C^k(M, g)}\leq C(n, g, k, t_0)\ee_0 e^{-\eta t}. \label{eq:B045}
\eeq
\end{lem}

\begin{proof} The function $v:=\pd {\varphi_0}t$ satisfies the equation
$$
\pd vt=b_{l\bar k}v_{k\bar l},
$$ where $b_{l\bar k}=g_{\varphi_0}^{i\bar l}g_{\varphi_0}^{k\bar j}g_{i\bar j}.$ Note that $v(x, 0)$ satisfies
 \beq
 \|v(x, 0)\|_{C^k(M, g)}=\|n-\tr_{\varphi_0}g\|_{C^k(M, g)}=\|g^{i\bar j}_{\varphi_0}\varphi_{0, i\bar j}\|_{C^k(M, g)}\leq \ee_0.\no
 \eeq
 Therefore, by Lemma \ref{lem:B004a} there exist $\eta(g)>0$ and $v_{\infty}\in \RR$ such that for $t\geq 0$ we have
\beq
\|v-v_{\infty}\|_{C^k(M\times [t, t+1], g)}\leq C(n, g, k)\ee_0 e^{-\eta t}. \label{eq:A305}
\eeq
 Since $\varphi_{0}$ satisfies $I(\varphi_0(x, t))=0$ for all $t\geq 0,$ we have
 \beq
 \frac d{dt}I(\varphi_0(x, t))=\int_M\;v\,\oo_{\varphi_0}^n=0.\label{eq:A304}
 \eeq
 Letting $t\ri +\infty$ in (\ref{eq:A304}) and using (\ref{eq:A305}), we have $v_{\infty}=0$ and then (\ref{eq:A305}) implies that for $t\geq t_0$,
 \beq \|v(\cdot, t)\|_{C^k(M\times [t, t+1], g)}\leq C\ee_0 e^{-\eta t}.\label{eq:A304a}\eeq
 Since $\varphi_0$ converges smoothly to $0$ as $t\ri +\infty$, for $t\geq 0$ we have
\beq
|\varphi_0(t)|\leq \int_t^{\infty}\,\max_{x\in M}|v(x, \tau)|\,d\tau\leq C(n, g)\ee_0e^{-\eta t}.\no
\eeq
Moreover, we have
\beq
\pd {}t|\Na \varphi_0|^2=-v_{j\bar i}\varphi_{0, i }\varphi_{0, \bar j }+\varphi_{0, i }v_{\bar i}+\varphi_{0, \bar i }v_{ i}. \no
\eeq
 Thus, by (\ref{eq:A303}) and (\ref{eq:A304a}) for $t\geq 0$ we have
 \beqn
 \pd {}t|\Na \varphi_0|&\leq& \frac 12\Big(|\Na^2 v|+|\Na v|\Big)|\Na \varphi_0|+|\Na v|\no\\
 &\leq& C(n, g)\ee_0 e^{-\eta t}|\Na \varphi_0|+C(n, g)\ee_0 e^{-\eta t}\leq  C(n, g)\ee_0e^{-\eta t}.  \label{eq:B043}
 \eeqn
 Since $\varphi_0(x, t)$   converges smoothly to $0$ as $t\ri +\infty$, we have $\lim_{t\ri \infty}|\Na \varphi_0|=0.$ Thus, (\ref{eq:B043}) implies that
 \beq
 \|\Na \varphi_0\|_{C^0(M)}\leq C(n, g)\ee_0 e^{-\eta t},\quad \forall\, t\geq 0.\no
 \eeq
 Similarly, we can show higher order estimates (\ref{eq:B044}) and we omit the details.

 To prove (\ref{eq:B045}), we take $\psi_{\tau}=\tau\varphi_0$ for $\tau\in [0, 1]$.
 Thus, we have that $\psi_0=0$ and $R(\psi_0)=R(\oo_g)=\un R$ by the assumption. By (\ref{eq:B044}), for $t\geq 0$ we have
 \beqn
 \|R(\varphi_0(t))-\un R\|_{C^k(M, g)}&=&\|R(\psi_1)-R(\psi_0)\|_{C^k(M, g)}\no\\&\leq& \int_0^1\,\|(\Delta_{\psi_{\tau}}\varphi_0+Ric_{\psi_{\tau},
 i\bar j}\varphi_{0, j\bar i})\|_{C^k(M, g)}\,d\tau \no\\
 &\leq& C(n, g, k)\ee_0 e^{-\eta t},\no
 \eeqn where  we used the inequality (\ref{eq:B044}). The lemma is proved.

\end{proof}

The next result shows that if the norm of $\varphi_0(x, t)$ is sufficiently small, then we can choose $u_j$ such that the function
\beq
\td \varphi_{N, s}(x, t):=\varphi_0(x, t)+\sum_{j=1}^N\, \frac {s^j}{j!} u_{j}(x, t),\label{eq:B014}
\eeq
  has uniform estimates.
\begin{lem}\label{lem:B005} Let  $\oo_g$ be a cscK metric and $\varphi_0(x, t)(t\in [0, \infty))$ be the solution of $J$-flow in Lemma \ref{lem:B002}. For any integer $N>0$,  there exist $\{u_{j}(x, t)(1\leq j\leq N)\}$ on $M\times [0, \infty)$ with
  \beq
  u_j(x, 0)=c_j\in \RR,\quad \forall\,x\in M,
  \eeq such that for any $s\in [0, 1]$ the metric $\td \varphi_{N, s}$ defined by (\ref{eq:B014}) satisfies
\beqn
\|e^{\eta t}L_s(\td \varphi_{N, s}(x, t))\|_{C^{0, 0, \ga}(M\times [0, \infty), g)}&\leq& C(n, g)\ee_0 s^{N+1}, \label{eq:C013}\\
\|e^{\eta t}(\td \varphi_{N, s}(x, t)-\varphi_0(x, t)) \|_{C^{4, 1, \ga}(M\times [0, \infty), g)}&\leq &C(n, g)\ee_0 s.\label{eq:C014}
\eeqn In particular, there exist $\te>0$  such that for $t\geq 0$ we have
\beq
\oo_{\td \varphi_{N, s}}\geq \te\oo_g, \quad \|e^{\eta t}\td \varphi_{N, s}(x, t)\|_{C^{4, 1, \ga}(M\times [0, \infty), g)}\leq C(n, g)\ee_0.\no
\eeq

\end{lem}
\begin{proof}We follow the argument of Lemma \ref{lem007}. Since $u_1$ satisfies the equation
\beq
\pd {u_1}t=g_{\varphi_0}^{i\bar l}g_{\varphi_0}^{k\bar j}g_{i\bar j}u_{1, k\bar l}+(R(\varphi_0)-\un R-n+\tr_{\varphi_0}\oo_g),\quad u_1(x, 0)=0. \no
\eeq
By Lemma \ref{lem:B002}, the function $h_1:=R(\varphi_0)-\un R-n+\tr_{\varphi_0}\oo_g$ satisfies
\beq
\|h_1\|_{C^k(M\times [t, t+1], g)}\leq C(n, g,     k, t_1)\ee_0 e^{-\eta t},\quad \forall\, t\geq 0.\no
\eeq
By Lemma \ref{lem:B004a}, there exists $c_1\in \RR$ such that $\td u_1:=u_1-c_1$ satisfies
\beq
\|\td u_1\|_{C^k(M\times [t, t+1], g)}\leq C(n, g,     k, t_1)\ee_0e^{-\eta t},\quad \forall\,t\geq 0.\no
\eeq
Similarly, since $u_j(j\geq 2)$ satisfies the equation
\beq
\pd {u_j}t=g_{\varphi_0}^{\al\bar \dd}g_{\varphi_0}^{\eta\bar \bb}g_{\eta\bar \dd}u_{j, \al\bar \bb}+Q_{4, j-1}(\Na^ku_l)+Ric(\varphi_0)*Q_{2, j-1}(\Na^{k}u_{l}),\no
\eeq there exists $c_j\in \RR$ such that $\td u_j:=u_j-c_j$ satisfies
\beq
\|\td u_j\|_{C^k(M\times [t, t+1], g)}\leq C(n, g,     k, t_0)\ee_0 e^{-\eta t},\quad \forall\,t\geq 0. \no \eeq
 We define the function $\td \varphi_{N, s}$ by replacing $u_j$ with $\td u_j$ in (\ref{eq:B014}). The rest of the proof is the same as that in Lemma \ref{lem007}.
The lemma is proved.
\end{proof}

\subsection{The linearized operator}
In this subsection, we introduce a weighted H\"older space which consists of functions of zero average with respect to a given family of metrics.  Such functions will be used to get the Poincar\'e inequality.   Let $\td \varphi:=\td \varphi_{N, s} $ be the family of metrics in Lemma \ref{lem:B005}. For any $\psi\in \cH(\oo_g)$, we define the  operator $A_s$  by
\beq
A_s(\varphi)=s(R(\varphi)-\un R)+(1-s)(n-\tr_{\varphi}\oo_g).\no
\eeq The twisted Calabi flow can be written as
\beq
\pd {\varphi}t-A_s(\varphi)=0.\label{eq:B0003}
\eeq
We normalize $\varphi(x, t)$ such that the average is  zero
\beq
\psi(x, t)=\varphi(x, t)-\un \varphi(t),\quad \un \varphi(t)=\frac 1V\int_M\; \varphi(x, t)\,\oo_{\td \varphi(x, t)}^n,\no
\eeq where $V$ is the volume of $\oo_g.$
Therefore,  (\ref{eq:B0003}) can be written as
\beq
\pd {\psi}t-A_s(\psi)+\un \varphi'(t)=0,\no
\eeq where $\un \varphi'(t)=\frac d{dt}\un \varphi(t).$
Note that
\beq
\un \varphi'(t)=\frac 1V\int_M\; \Big(A_s(\psi)+\psi\Delta_{\td \varphi}\pd {\td \varphi}t\Big)\,\oo_{\td \varphi}^n. \no
\eeq
Therefore, we define the operator
\beq
\td L_s(\psi):=\pd {\psi}t-A_s(\psi)+\frac 1V\int_M\; \Big(A_s(\psi)+\psi\Delta_{\td \varphi}\pd {\td \varphi}t\Big)\,\oo_{\td \varphi}^n, \no
\eeq where $\psi(x, t)$ satisfies
\beq
\int_M\;\psi(x, t)\,\oo_{\td \varphi}^n=0.\no
\eeq

We introduce the following function spaces.
\begin{defi}\label{defi:C1}Let $\ga\in (0, 1)$ and $k\geq 0$ be an integer.  Assume that $g$ and $g_{\varphi(x, t)}(t\in [a, \infty))$ are K\"ahler metrics on $M$.
\begin{enumerate}
  \item[(1).]
For any $\eta>0$ and $f(x, t)\in C^{k, [k/4], \ga }(M\times [a, +\infty), g)$, we define
\beq
\|f\|_{C_{\eta}^{k, [k/4], \ga }(M\times [a, +\infty), g)}:=\|e^{\eta t} f\|_{C^{k, [k/4], \ga }(M\times [a, +\infty), g)} \label{eq:B006}
\eeq  and
\beqn &&
C_{\eta}^{k, [k/4], \ga }(M\times [a, +\infty), g)\no\\&:=&\Big\{f\in C^{k, [k/4], \ga }(M\times [a, +\infty), g)\,\Big |\,\|f\|_{C_{\eta}^{k, [k/4], \ga }(M\times [a, +\infty), g)}<\infty\Big\}.\no
\eeqn
  \item[(2).]
We  define  $C_{\eta, 0}^{k, [k/4], \ga }(M\times [a, +\infty), g, g_{\varphi})$ the space of functions $f\in C^{k, [k/4], \ga }_{\eta}(M\times [a, +\infty), g)$ with
\beq
\int_M\; f(x, t)\,\oo_{\varphi(x, t)}^n=0,\quad \forall\;t\geq a\no
\eeq and the norm $\|\cdot \|_{C_{\eta, 0}^{k, [k/4], \ga }(M\times [a, +\infty), g, g_{\varphi})} $ on $C_{\eta, 0}^{k, [k/4], \ga }(M\times [a, +\infty), g, g_{\varphi})$ is defined to be the same as in  (\ref{eq:B006}).

  \item[(3).]Let $\psi\in \cH(\oo_g)$.  We  define  $C_{\eta, 0, \psi}^{k, [k/4], \ga }(M\times [a, +\infty), g, g_{\varphi})$ the space of functions $f\in C^{k, [k/4], \ga }_{\eta}(M\times [a, +\infty), g)$ with
\beq
\int_M\; f(x, t)\,\oo_{\varphi}^n=0,\quad f(x, 0)=\psi(x),\quad \forall\; x\in M,\no
\eeq and the norm $\|\cdot \|_{C_{\eta, 0, \psi}^{k, [k/4], \ga }(M\times [a, +\infty), g, g_{\varphi})} $ on $C_{\eta, 0, \psi}^{k, [k/4], \ga }(M\times [a, +\infty), g, g_{\varphi})$ is defined to be the same as in  (\ref{eq:B006}).

\end{enumerate}

\end{defi}

By direct calculation and using Definition \ref{defi:C1}, we have the properties of $\td L_s.$
 \begin{lem}\label{lem:B003}
\begin{enumerate}
\item[(1).] For any $u\in {C_{\eta, 0, \psi_0}^{4, 1, \ga }(M\times [0, +\infty), g, g_{\varphi})}$, we have
    \beq
    \td L_s(u)\in C_{\eta, 0}^{0, 0, \ga }(M\times [0, +\infty), g, g_{\varphi}).\no
    \eeq
\item[(2).] The derivative of $\td L_s$ at $\td \varphi$ is given by \beqn
D\td L_{s}|_{\td \varphi}(u)&=&\pd {u}t+s(\Delta^2_{ \td \varphi}u+Ric_{\td \varphi, i\bar j
}u_{j\bar i})-(1-s)g_{\td \varphi}^{i\bar l}g_{\td \varphi}^{k\bar j}g_{i\bar j}u_{k\bar l}\no\\
&&+\frac 1V\int_M\;\Big(a_{l\bar k}u_{k\bar l}+u\Delta_{\td \varphi}\pd {\td \varphi}t\Big)\,\oo_{\td \varphi}^n \no
\eeqn where $a_{l\bar k}$ is defined by
\beqn
a_{l\bar k}&=&(1-s)g_{\td\varphi}^{i\bar l}g_{\td \varphi}^{k\bar j}g_{i\bar j}
-sRic_{\td \varphi, l\bar k}\no\\&=&(1-s)g_{\td \varphi}^{k\bar l}-(1-s)g_{\td\varphi}^{i\bar l}g_{\td \varphi}^{k\bar j}\td \varphi_{i\bar j}
-sRic_{\td \varphi, l\bar k}. \label{eq:B015}
\eeqn Moreover, $D\td L_{s}|_{\td \varphi}$ maps $C_{\eta, 0}^{4, 1, \ga }(M\times [0, +\infty), g, g_{\varphi})$ to $C_{\eta, 0}^{0, 0, \ga }(M\times [0, +\infty), g, g_{\varphi})$.
\end{enumerate}

\end{lem}

\subsection{The $L^2$ estimates}

In this subsection, we show that the linearized operator $D\td L_s$ of twisted Calabi flow is invertible between weighted H\"older spaces for small $s$. First, we have the following result, which shows that the modified linearized operator $G_s$ is invertible for small $s$.

\begin{lem}\label{lem:B006}There exist $\ee_0(g), \eta_0(g)>0$ and $s_0\in (0, 1)$ satisfying the following properties. Suppose that
\begin{enumerate}
  \item[(1).] $\td \varphi(x, t)(t\in [0, \infty))$ is a family of K\"ahler potentials satisfying
      \beq
  \oo_{\td \varphi(x, t)}\geq \te\oo_g,\quad  \|\td \varphi(x, t)\|_{C_{\eta}^{4, 1, \ga}(M\times [0, \infty), g)}\leq \ee_0, \label{eq:B100v}
      \eeq where $\te, \ee_0>0.$

\item[(2).]   Define the operator $G_s: C_{\eta, 0}^{4, 1, \ga }(M\times [0, +\infty), g, g_{\td \varphi})\ri
C_{\eta}^{0, 0, \ga }(M\times [0, +\infty), g, g_{\td \varphi})$ by $G_s(w)=f$, where $w$ and $f$ satisfies the equation
\beqn
\pd {w}{\tau}+\Delta^2_{\td \varphi}w&=&s^{-1} f(x, \tau)+s^{-1}a_{l\bar k}(x, \tau)w_{k\bar l}-c_s(\tau).\no
\eeqn Here $a_{i\bar j}$ is defined by (\ref{eq:B015}) and $c_s(\tau)$ is defined by
\beq
c_s(\tau)=\int_M\;\Big(s^{-1}a_{l\bar k}(x, \tau)w_{k\bar l}+w\Delta_{\td \varphi}\pd {\td \varphi}{\tau}\Big)\,\oo_{\td \varphi}^n. \no
\eeq

\end{enumerate}
 Then we have \begin{enumerate}
                \item[(a).] For any $w\in C_{\eta, 0}^{4, 1, \ga }(M\times [0, +\infty), g, g_{\td \varphi})$, we have
\beq
\int_M\;G_s(w)\,\oo_{\td \varphi}^n=0.\no
\eeq
Thus, $G_s$ is a well-defined operator from $C_{\eta, 0}^{4, 1, \ga }(M\times [0, +\infty), g, g_{\td \varphi})$ to $C_{\eta, 0}^{0, 0, \ga }(M\times [0, +\infty), g, g_{\td \varphi})$.
                \item[(b).] Let $s\in (0, s_0)$ and $\eta\in (0, \eta_0)$.  For any $f\in C_{\eta, 0}^{0, 0, \ga }(M\times [0, +\infty), g, g_{\td \varphi})$, there exists a unique solution $w\in C_{\eta, 0, 0}^{4, 1, \ga }(M\times [0, +\infty), g, g_{\td \varphi})$ of the equation $G(w)=f$ and we have
\beq
\|w\|_{C_{\eta, 0, 0}^{4, 1, \ga }(M\times [0, +\infty), g, g_{\td \varphi})}\leq C(n, g, \te)s^{-k(\ga, n)}
\|f\|_{C_{\eta, 0}^{0, 0, \ga }(M\times [0, +\infty), g, g_{\td \varphi})}. \no
\eeq

              \end{enumerate}

\end{lem}
\begin{proof}(1). By the assumption $w\in  C_{\eta, 0}^{4, 1, \ga }(M\times [a, +\infty), g, g_{\td \varphi})$, we have
\beqn
0&=&\frac d{d\tau}\int_M\; w\,\oo_{\td \varphi}^n=\int_M\; \Big(\pd w{\tau}+w\Delta_{\td \varphi}\pd {\td \varphi}{\tau}\Big)\,\oo_{\td\varphi}^n\no\\
&=&\int_M\; \Big(s^{-1} f(x, \tau)+s^{-1}a_{l\bar k}(x, \tau)w_{k\bar l}+w\Delta_{\td \varphi}\pd {\td \varphi}{\tau}\Big)\,\oo_{\td \varphi}^n-c_s(\tau)\no\\
&=&\int_M\;  s^{-1} f(x, \tau)\,\oo_{\td \varphi}^n. \label{eq:B011}
\eeqn Therefore, we have
$$\int_M\;   G(w)\,\oo_{\td \varphi}^n=\int_M\;   f\,\oo_{\td \varphi}^n=0.$$
Thus, we have  $G(w)\in C_{\eta, 0}^{0, 0, \ga }(M\times [a, +\infty), g, g_{\td \varphi})$.

(2). Given $f\in C_{\eta, 0}^{0, 0, \ga }(M\times [a, +\infty), g, g_{\td \varphi})$, we would like to find $w\in C_{\eta, 0, 0}^{4, 1, \ga }(M\times [a, +\infty), g, g_{\td \varphi})$ such that $G(w)=f. $ By Theorem \ref{theo:A2x}, the
equation
\beq
G(w)=f, \quad w(x, 0)=0 \no
\eeq has a solution $w\in C^{4, 1, \ga }(M\times [a, +\infty), g)$. By the calculation in (\ref{eq:B011}), we have
\beq
\frac d{d\tau}\int_M\;w\,\oo_{\td \varphi}^n=0,\quad \forall \;\tau>0.\no \eeq Thus, $\int_M\;w(x, \tau)\,\oo_{\td \varphi}^n=0$ for all $\tau>0$ and $w\in C_{\eta, 0, 0}^{4, 1, \ga }(M\times [a, +\infty), g, g_{\td \varphi})$.

(3).  Let
$z=e^{\eta \tau}w(x, \tau). $
  We have
  \beq
  z(x, 0)=0,\quad \int_M\;z(x, \tau)\,\oo_{\td \varphi(\tau)}^n=0. \label{eq:B030}
  \eeq
  Note that $z$ satisfies the equation on $M\times [0, \infty)$
\beqn
\pd {z}{\tau}+\Delta^2_{\td \varphi}z&=&\td h:=s^{-1} \td f(x, \tau)+s^{-1}a_{l\bar k}(x, \tau)z_{k\bar l}+\eta z-e^{\eta \tau}c_s(\tau),\no\\
z(x, 0)&=&0, \no
\eeqn where $\td f:=e^{\eta \tau}f.$ We estimate the $L^2$ norm of $z$. By direct calculation, we have
\beqn &&
\frac d{d\tau}\int_M\;z^2\oo_{\td \varphi}^n=\int_M\;\Big(2z\pd z{\tau}\,\oo_{ \td \varphi}^n+z^2
\pd {}{\tau}\oo_{ \td \varphi}^n\Big)\no\\
&\leq&\int_M\;2z\Big(-\Delta^2_{ \td \varphi}z+s^{-1}\td f
+s^{-1} a_{l\bar k}z_{k\bar l}+\eta z-e^{\eta \tau}c_s(\tau)\Big)\oo_{ \td \varphi}^n+\ee_0\int_M\;z^2\oo_{ \td \varphi}^n\no\\
&\leq&-2\int_M\;(\Delta_{ \td \varphi}z)^2\oo_{ \td \varphi}^n+2s^{-1}\int_M\;a_{l\bar k}z_{k\bar l}z\,\oo_{\td \varphi}^n\no\\
&&+2s^{-1}\int_M\;\td fz+(2\eta+\ee_0)\int_M\;z^2\oo_{ \td \varphi}^n.\label{eq:A016}
\eeqn
Note that by (\ref{eq:B015}) and (\ref{eq:B100v}) we have
\beq
g_{\varphi, i\bar j}=g_{i\bar j}+\varphi_{i\bar j}\leq (1+\ee_0)g_{i\bar j}.
\eeq
Thus, we have
\beqn
a_{l\bar k}&\geq& \Big(\frac {1-s}{1+\ee_0}-sC(g)\Big)g_{\varphi, l\bar k}, \label{eq:A500}\\
|a_{l\bar k, \bar l}|&=&|-(1-s)g_{\td\varphi}^{i\bar l}g_{\td \varphi}^{k\bar j}\td \varphi_{i\bar j\bar l}
+s(Ric_{\td \varphi, l\bar k})_{\bar l}|\leq C(g, \te)\ee_0.  \label{eq:A501}
\eeqn
 Combining (\ref{eq:A016})-(\ref{eq:A501}), we have
\beqn &&
\int_M\;a_{l\bar k}z_{k\bar l}z\,\oo_{\td \varphi}^n=-\int_M\;a_{l\bar k}z_{k}z_{\bar l}\,\oo_{\td \varphi}^n-\int_M\;a_{l\bar k, \bar l}z_{k}z\,\oo_{ \td \varphi}^n\no\\
&\leq&- \Big(\frac {1-s}{1+\ee_0}-sC(g)\Big) \int_M\;|\Na z|^2\,\oo_{\td \varphi}^n+C(g, \te)\ee_0\int_M\;( |\Na z|^2+ z^2)\,\oo_{ \td \varphi}^n\no\\
&=&-\Big(\frac {1-s}{1+\ee_0}-sC(g)-C(g, \te)\ee_0\Big)  \int_M\;|\Na z|^2\,\oo_{ \td \varphi}^n+C(g, \te)\ee_0\int_M\;z^2\,\oo_{ \td \varphi}^n.\no\\ \label{eq:A017}
\eeqn  Since $g_{\td \varphi}$ is equivalent to $g$ for any $t>0$ by (\ref{eq:B100v}), there exists a constant $\la(g)>0$ independent of $t$ such that the Poincar\'e inequality holds
\beq
\int_M\;|\Na z|^2\,\oo_{ \td \varphi}^n\geq \la(g)\int_M\;z^2\,\oo_{ \td \varphi}^n,\label{eq:A073}
\eeq where we used (\ref{eq:B030}).
Combining  (\ref{eq:A073})(\ref{eq:A017}) with  (\ref{eq:A016}), we have
\beqn &&
\frac {d}{dt}\int_M\;z^2\oo_{ \td \varphi}^n\leq
-2\int_M\;(\Delta_{\td \varphi}z)^2\oo_{\td \varphi}^n-2 s^{-1}\Big(\frac {1-s}{1+\ee_0}-sC(g)-C(g, \te)\ee_0\Big)\int_M\;|\Na z|^2\,\oo_{ \td \varphi}^n\no\\&&+
C(g, \te) \ee_0 s^{-1}\int_M\;z^2\,\oo_{ \td \varphi}^n+2s^{-1}\Big(\int_M\,\td f^2\,\oo_{ \td \varphi}^n\Big)^{\frac 12}\Big(\int_M\;z^2\,\oo_{ \td \varphi}^n\Big)^{\frac 12}+(2\eta+\ee_0)\int_M\;z^2\oo_{ \td \varphi}^n\no\\
&\leq& - 2\la  s^{-1}\Big(\frac {1-s}{1+\ee_0}-sC(g)-C(g, \te)\ee_0\Big)\int_M\;z^2\,\oo_{ \td \varphi}^n +\Big(C(g, \te)\ee_0 s^{-1}+2\eta+\ee_0\Big)\int_M\;z^2\,\oo_{ \td \varphi}^n\no\\&&+2s^{-1}\Big(\int_M\,\td f^2\,\oo_{ \td \varphi}^n\Big)^{\frac 12}\Big(\int_M\;z^2\,\oo_{ \td \varphi}^n\Big)^{\frac 12}
\no \\
&\leq&-A s^{-1}\int_M\;z^2\,\oo_{ \td \varphi}^n+2s^{-1}\Big(\int_M\,\td f^2\,\oo_{ \td \varphi}^n\Big)^{\frac 12}\Big(\int_M\;z^2\,\oo_{ \td \varphi}^n\Big)^{\frac 12}, \label{eq:A018}
\eeqn where \beq A=2\la \Big(\frac {1-s}{1+\ee_0}-sC(g)-C(g, \te)\ee_0\Big) -C(g, \te)\ee_0 -(2\eta+\ee_0)s.\no\eeq
There exist $\ee_1(g)>0$ and $s_0(g)\in (0, 1)$ such that for any $\ee_0\in (0, \ee_1)$ and $s\in (0, s_0)$ we have $A\geq \la.$
Let
$ E(\tau)= \Big(\int_M\;z^2\oo_{\td \varphi}\Big)^{\frac 12}.
$
The inequality (\ref{eq:A018}) implies that
\beq
E'(\tau)\leq  -\frac {\la}2 s^{-1} E(\tau)+s^{-1}\|\td f(\cdot, \tau)\|_{L^2(M, \oo_{\td \varphi(\tau)})}.\no
\eeq
Let $A'=\frac {\la}2 s^{-1}.$ Then we have
\beqn
E(\tau)&\leq& e^{-A'\tau}E(0)+s^{-1}e^{-A'\tau}\int_0^{\tau}\,\Big(e^{A't}\|\td f(\cdot, t)\|_{L^2(M, \oo_{\td \varphi(t)})}\Big)\,dt\no\\
&=& 2\la^{-1}\max_{t\in [0, \tau]}\|\td f(\cdot, t)\|_{L^2(M, \oo_{\td \varphi(t)})}, \label{eq:B010}
\eeqn where we used $E(0)=0$ by (\ref{eq:B030}).

(4).
By Theorem \ref{theo:Me} and  $z(x, 0)=0(\forall x\in M)$, we have
\beqn &&
\|z\|_{C^{4, 1, \ga}(M\times [0, \infty), g)}\no\\&\leq& C(n, g, \te)\Big(\|\td h\|_{C^{0, 0, \ga}(M\times [0, \infty), g)}+\max_{t\in [0, \infty)}\|z(\cdot, t)\|_{L^2(M, g)}\Big)\no\\
&\leq&C(n, g, \te)\Big(s^{-1}\|\td f\|_{C^{0, 0, \ga}(M\times [0, \infty), g)}+
s^{-1}\|\Na^2z\|_{C^{0, 0, \ga}(M\times [0, \infty), g)}
 +\eta \|z\|_{C^{0, 0, \ga}(M\times [0, \infty), g)}\no\\&&+\|e^{\eta\tau}c_s(w, \tau)\|_{C^{0, 0, \ga}(M\times [0, \infty), g)}+
\max_{\tau\in [0, \infty)}\|z(\cdot, \tau)\|_{L^2(M, g)}\Big).\label{eq:B007}
\eeqn
Note that by (\ref{eq:A039x}) we have
\beq
\|z\|_{C^{0, 0, \ga}(M\times [0, \infty), g)}\leq \ee  \|z\|_{{C^{4, 1, \ga}(M\times [0, \infty), g)}}+C(n, g) \ee^{-\ka'(\ga, n)}\max_{\tau\in [0, \infty)}\|z(\cdot, \tau)\|_{L^2(M, g)},\label{eq:B016}
\eeq where $\ka'(\ga, n)>0$ and by (\ref{eq:A039x}) and (\ref{eq:B010}) we have
\beqn &&
\|\Na^2 z\|_{{C^{0, 0, \ga}(M\times [0, \infty), g)}}\no\\&\leq&\ee s \|z\|_{{C^{4, 1, \ga}(M\times [0, \infty), g)}}+C(n, g) \ee^{-\ka'}s^{-\ka'}\max_{\tau\in [0, \infty)}\|z(\cdot, \tau)\|_{L^2(M, g)}\no\\
&\leq& \ee s \|z\|_{{C^{4, 1, \ga}(M\times [0, \infty), g)}}+C(n, g) \ee^{-\ka'}s^{-\ka'}\max_{\tau\in [0, \infty)}\|\td f(\cdot, \tau)\|_{L^2(M, g)}, \label{eq:B017}
\eeqn where
$
\ka'=\ka'(\ga, n)>0.
$
Combining (\ref{eq:B007})-(\ref{eq:B017}), we have
\beq
\|z\|_{C^{4, 1, \ga}(M\times [0, \infty), g)}
\leq C(n, g, \te) \Big( s^{-\ka'-1}\|\td f\|_{C^{0, 0, \ga}(M\times [0, \infty), g)}+\|e^{\eta\tau}c_s(\tau)\|_{C^{0, 0, \ga}(M\times [0, \infty), g)}\Big).\label{eq:B012}
\eeq

(5). We estimate $\|e^{\eta\tau}c_s( \tau)\|_{C^{0, 0, \ga}(M\times [0, \infty), g)}. $ Note that
 \beqn
 c_s( \tau)&=&\int_M\;\Big(s^{-1}a_{l\bar k}(x, \tau)w_{k\bar l}+w\Delta_{\td \varphi}\pd {\td \varphi}t\Big)\,\oo_{\td \varphi}^n\no\\
&=&\int_M\;\Big(s^{-1}a_{l\bar k, \bar lk}(x, \tau)+\Delta_{\td \varphi}\pd {\td \varphi}t\Big)w\,\oo_{\td \varphi}^n=\int_M\;b(\td \varphi)w\,\,\oo_{\td \varphi}^n, \no
 \eeqn where $$b(\td \varphi):=s^{-1}a_{l\bar k, \bar lk}(x, \tau)+\Delta_{\td \varphi}\pd {\td \varphi}t.$$
  By (\ref{eq:B015}) we have
\beq
a_{l\bar k, \bar lk}=-(1-s)g_{\td\varphi}^{i\bar l}g_{\td \varphi}^{k\bar j}\td \varphi_{i\bar j\bar lk}
-s\Delta_{\td \varphi}R(\td \varphi),\no
\eeq which implies that
\beqn &&
\|b(\td \varphi)\|_{C^{0, 0, \ga}_{\eta}(M\times [0, \infty), g)}\no\\&=&\Big\|-s^{-1}(1-s)g_{\td\varphi}^{i\bar l}g_{\td \varphi}^{k\bar j}\td \varphi_{i\bar j\bar lk}
-\Delta_{\td \varphi}R(\td \varphi)+\Delta_{\td \varphi}\pd {\td \varphi}t\Big\|_{C^{0, 0, \ga}_{\eta}(M\times [0, \infty), g)}\no\\
&\leq&C(n, \te)s^{-1}\|\td \varphi\|_{_{C^{4, 1, \ga}_{\eta}(M\times [0, \infty), g)}}+\|\Na^2(R(\td \varphi-\un R))\|_{C^{0, 0, \ga}_{\eta}(M\times [0, \infty), g)}\no\\&&+\Big\|\Na^2\pd {\td \varphi}t\Big\|_{C^{0, 0, \ga}_{\eta}(M\times [0, \infty), g)}\no\\
&\leq&C(n, g,  \te)\ee_0 s^{-1}, \label{eq:B101}
\eeqn where we used the assumption (\ref{eq:B100v}). Thus, we have
\beqn
|e^{\eta \tau}c_s( \tau)|&=&\Big|\int_M\; e^{\eta \tau}b(\td \varphi)w\,\,\oo_{\td \varphi}^n\Big|\leq C(n, \te)\ee_0s^{-1}\|w\|_{L^2(M, \oo_{\td \varphi})}. \label{eq:A200z}
\eeqn
On the other hand, we have
\beqn &&
|c_s(\tau_1)-c_s(\tau_2)|\no\\&\leq &\Big|\int_M\;b(\td \varphi(x, \tau_1))w(x, \tau_1)\,\,\oo_{\td \varphi(x, \tau_1)}^n-\int_M\;b(\td \varphi(x, \tau_2))w(x, \tau_2)\,\,\oo_{\td \varphi(x, \tau_2)}^n\Big|\no\\
&\leq& \int_M\; |b(\td \varphi(x, \tau_1))-b(\td \varphi(x, \tau_2)) |\cdot |w(x, \tau_1)|\,\,\oo_{\td \varphi(x, \tau_1)}^n\no\\
&&+\int_M\; |b(\td \varphi(x, \tau_2)) |\cdot |w(x, \tau_1)-w(x, \tau_2)|\,\,\oo_{\td \varphi(x, \tau_1)}^n\no\\
&&+\int_M\; |b(\td \varphi(x, \tau_2)) |\cdot |w(x, \tau_2)|\,\,|\oo_{\td \varphi(x, \tau_1)}^n-\oo_{\td \varphi(x, \tau_2)}^n|. \no
\eeqn
   Therefore, we have
\beqn &&
[e^{\eta \tau}c_s(\tau)]_{\ga, M\times [0, \infty)}\leq \|b\|_{C^{0, 0, \ga}(M\times [0, \infty), g)}\cdot \max_{\tau\in [0, \infty)}\|z(\cdot, \tau)\|_{L^2(M, \td \varphi(\tau))}\no\\&&+\|b\|_{C^0(M\times [0, \infty))}\|z\|_{C^{0, 0, \ga}(M\times [0, \infty), g)  }
+\|b\|_{C^{0, 0, \ga}(M\times [0, \infty), g)} [\Na^2\varphi]_{\ga, M\times [0, \infty)}\no\\&&\cdot \max_{\tau\in [0, \infty)}\|z(\cdot, \tau)\|_{L^2(M, \td \varphi(\tau))}\no\\
&\leq&C(n, g,  \te)\ee_0 s^{-1}\max_{\tau\in [0, \infty)}\|z(\cdot, \tau)\|_{L^2(M, g)}+C(n, g, \te)\ee_0 s^{-1}\no\\&&\cdot \Big(\ee_1  \|z\|_{{C^{4, 1, \ga}(M\times [0, \infty), g)}}+C(n, g) \ee_1^{-\ka'}\max_{\tau\in [0, \infty)}\|z(\cdot, \tau)\|_{L^2(M, g)} \Big). \label{eq:A081}
\eeqn
Taking $\ee_1=s$ and combining (\ref{eq:A081}) with (\ref{eq:A200z}) and (\ref{eq:B010}), we have
\beqn &&
\|e^{\eta\tau}c_s( \tau)\|_{C^{0, 0, \ga}(M\times [0, \infty), g)}\leq C(n, g,  \te)\ee_0 s^{-1}\max_{\tau\in [0, \infty)}\|z(\cdot, \tau)\|_{L^2(M, g)}\no\\&&+C(n, g,  \te)\ee_0  \|z\|_{{C^{4, 1, \ga}(M\times [0, \infty), g)}}  +C(n, g, \te)\ee_0 s^{-\ka'-1}\max_{\tau\in [0, \infty)}\|z(\cdot, \tau)\|_{L^2(M, g)}\no\\
&\leq& C(n, g, \te)\ee_0 s^{-\ka'-1}\max_{\tau\in [0, \infty)}\|\td f(\cdot, \tau)\|_{L^2(M, g)}+C(n, g, \te)\ee_0  \|z\|_{{C^{4, 1, \ga}(M\times [0, \infty), g)}}.\no\\ \label{eq:A082}
\eeqn Combining (\ref{eq:A082}) with (\ref{eq:B012}) and choosing $\ee_0$ small, we have
\beq
\|z\|_{C^{4, 1, \ga}(M\times [0, \infty), g)}
\leq C(n, g,  \te) s^{-\ka'-1}\|\td f\|_{C^{0, 0, \ga}(M\times [0, \infty), g)}.\no
\eeq

\end{proof}

Lemma \ref{lem:B006} implies that the operator $D\td L_{s}$ is invertible between
weighted H\"older spaces.

\begin{lem}\label{lem:B001} Let $\td \varphi$  the function and $s_0, \eta_0$ the constants in Lemma \ref{lem:B006}. Then for any $s\in (0, s_0)$ and $\eta\in (0, \eta_0)$ the operator $D\td L_{s}|_{\td \varphi}$ defined in Lemma \ref{lem:B003} is an operator from $C_{\eta s, 0, 0}^{4, 1, \ga}(M\times [0, \infty), g, g_{\td \varphi})$ to $C_{\eta s, 0}^{0, 0, \ga}(M\times [0, \infty), g, g_{\td \varphi})$ with
\beq
\|(D\td L_{s}|_{\td \varphi})^{-1}\|\leq C(n, g) s^{-\ka(\ga, n)},\no
\eeq where $\ka(\ga, n)>0. $

\end{lem}

\begin{proof}

For any $f\in C_{\eta s, 0}^{0, 0, \ga}(M\times [0, \infty ), g, g_{\td \varphi})$, we would like to find
$u \in C_{\eta s, 0, 0}^{4, 1, \ga}(M\times [0, \infty ), g, g_{\td \varphi})$ with
\beq
D\td L_s|_{\td \varphi}(u)=f.\label{eq:B035}
\eeq By Lemma \ref{lem:B003}, the equation (\ref{eq:B035}) is equivalent to
\beqn &&
\pd {u}t+s(\Delta^2_{ \td \varphi}u+Ric_{\td \varphi, i\bar j
}u_{j\bar i})-(1-s)g_{\td \varphi}^{i\bar l}g_{\td \varphi}^{k\bar j}g_{i\bar j}u_{k\bar l}\no\\
&&+\int_M\;\Big(a_{l\bar k}u_{k\bar l}+u\Delta_{\td \varphi}\pd {\td \varphi}t\Big)\,\oo_{\td \varphi}^n=f,\no
\eeqn
which can be written as
\beqn
\frac 1s\pd {u}t+\Delta^2_{\td\varphi}u&=&\frac 1sf-Ric_{\td\varphi, i\bar j
}u_{j\bar i}+\frac {1-s}sg_{\td\varphi}^{i\bar l}g_{\td\varphi}^{k\bar j}g_{i\bar j}u_{k\bar l}\no\\&&-\int_M\;\Big(s^{-1}a_{l\bar k}u_{k\bar l}+u\frac 1s\Delta_{\td \varphi}\pd {\td \varphi}t\Big)\,\oo_{\td \varphi}^n. \label{eq:B036}
\eeqn
Let
\beqn \tau&:=&st,\quad w(x, \tau):=u(x, t),\quad \td f(x, \tau):=f(x, t),\no\\
a_{l\bar k}&:=&-sRic_{\td\varphi, l\bar k}+(1-s)g_{\td\varphi}^{i\bar l}g_{\td\varphi}^{k\bar j}g_{i\bar j}.\no
\eeqn
Thus, (\ref{eq:B036}) implies that for $(x, \tau)\in M\times [0, \infty) $ we have
\beqn
\pd {w}{\tau}+\Delta^2_{\td\varphi}w&=&h(x, \tau):=s^{-1}\td f+s^{-1}a_{l\bar k}w_{k\bar l}-c_{s, w, \td \varphi}(\tau),\label{eq:A085}\\
w(x, 0)&=&0,\label{eq:A086}
\eeqn where $c_{s, w, \td \varphi}(\tau)$ is defined by
\beq
c_{s, w, \td \varphi}(\tau)=\int_M\;\Big(s^{-1}a_{l\bar k}w_{k\bar l}+w\Delta_{\td \varphi}\pd {\td \varphi}{\tau}\Big)\,\oo_{\td \varphi}^n.\no
\eeq
 By Lemma \ref{lem:B006}, for any $\td f\in C^{0, 0, \ga}_{\eta, 0}(M\times [0, \infty), g, g_{\td \varphi})$ there exists a solution $w(x, t)\in C^{4, 1, \ga}_{\eta, 0, 0}(M\times [0, \infty), g, g_{\td \varphi})$  of (\ref{eq:A085})-(\ref{eq:A086}) with
\beq
\|w\|_{{C_{\eta, 0}^{4, 1, \ga}(M\times [0, \infty), g, g_{\td \varphi})}}\leq C(n, g) s^{-\ka'}\|\td f\|_{{C_{\eta, 0}^{0, 0, \ga}(M\times [0, \infty), g, g_{\td \varphi})}}\no
\eeq for some $\ka'=\ka'(\ga, n)>0$.
Therefore, by Lemma \ref{lem:B008z} in the appendix we have
\beqn
\|e^{\eta st}u\|_{{C^{4, 1, \ga}(M\times [0, \infty), g)}}&\leq&\|e^{\eta \tau}w\|_{{C^{4, 1, \ga}(M\times [0, \infty), g)}}=\|w\|_{{C_{\eta, 0}^{4, 1, \ga}(M\times [0, \infty), g, g_{\td \varphi})}}\no\\
&\leq&C(n, g) s^{-\ka'}\|\td f\|_{{C_{\eta, 0}^{0, 0, \ga}(M\times [0, \infty), g, g_{\td \varphi})}}\no\\
&\leq&C(n, g) s^{-\ka'}\|e^{\eta \tau}\td f\|_{{C^{0, 0, \ga}(M\times [0, \infty), g)}}\no\\
&\leq&C(n, g) s^{-\ka'-\frac {\ga}4}\|e^{\eta st} f\|_{{C^{0, 0, \ga}(M\times [0, \infty), g)}}.\no
\eeqn
Thus, we have
\beq
\|u\|_{{C_{\eta s, 0}^{4, 1, \ga}(M\times [0, \infty), g, g_{\td \varphi})}}\leq
C(n, g) s^{-\ka'-\frac {\ga}4}\| f\|_{{C_{\eta s, 0}^{0, 0, \ga}(M\times [0, \infty), g, g_{\td \varphi})}}.\no
\eeq
The lemma is proved.

\end{proof}

\subsection{The contraction map}

In this subsection, we construct a contraction map between weighted H\"older spaces.

\begin{lem}\label{lem:B007}Let $\td \varphi$  the function and $s_0, \eta_0$ the constants in Lemma \ref{lem:B006}.  For any $f\in C_{\eta s, 0}^{0, 0, \ga}(M\times [0, \infty), g, g_{\td \varphi})$, we define \beqn \td\Psi_f :  C_{\eta s, 0, \psi_0}^{4, 1, \ga}(M\times [0, \infty), g, g_{\td \varphi})&\ri& C_{\eta s, 0, \psi_0}^{4, 1, \ga}(M\times [0,  \infty), g, g_{\td \varphi})\no\\
  \varphi&\ri& \varphi+(D\td L_s|_{\td \varphi})^{-1}(f-\td L_s(\varphi)).\no
\eeqn There exists $\dd(n, g), \al(\ga, n)>0$ satisfying the following properties.
If $s\in (0, s_0), \eta\in (0, \eta_0)$ and $h_1, h_2\in C_{\eta s, 0, \psi_0}^{4, 1, \ga}(M\times [0, \infty), g,  g_{\td \varphi})$ satisfy
\beq
\|h_1(x, t)-\td \varphi\|_{C^{4, 1, \ga}(M\times [0,  \infty), g)}\leq s^{\al}\dd, \quad \|h_2(x, t)-\td \varphi\|_{C^{4, 1, \ga}(M\times [0,  \infty), g)}\leq s^{\al}\dd, \no
\eeq then we have
\beq
\|\td \Psi_f(h_1)-\td \Psi_f(h_2)\|_{C^{4, 1, \ga}_{\eta s, 0}(M\times [0,  \infty), g, g_{\td \varphi})}\leq \frac 12 \|h_1-h_2\|_{C^{4, 1, \ga}_{\eta s, 0}(M\times [0,  \infty), g, g_{\td \varphi})}.\no
\eeq
\end{lem}

\begin{proof} We follow the argument of Lemma \ref{lem013}.

 (1). Let $h_{\xi}=\xi h_1+(1-\xi)h_2$. Then
\beq
\Psi_f(h_1)-\Psi_f(h_2)
=-\int_0^1\,(D\td L_s|_{\td \varphi})^{-1}\Big(D\td L_s|_{h_s}-D\td L_s|_{\td \varphi}\Big)(h_1-h_2)\,d\xi.\no
\eeq Let $v=h_1-h_2$. By Lemma \ref{lem:B003}, we have
\beq
\Big|\Big(D\td L_s|_{h_s}-D\td L_s|_{\td \varphi}\Big)v\Big|
\leq\Big|\Big(DL_s|_{h_s}-DL_s|_{\td \varphi}\Big)v\Big|+|c_{s, v, h_s}(\tau)-c_{s, v, \td \varphi}(\tau)|, \label{eq:A020x1}
\eeq where $c_{s, v, \td \varphi}(\tau)$ is defined by
\beq
c_{s, v, \td \varphi}(\tau)=\int_M\;\Big((1-s)g_{\td\varphi}^{i\bar l}g_{\td \varphi}^{k\bar j}g_{i\bar j}v_{k\bar l}
-sRic_{\td \varphi, l\bar k}v_{k\bar l}+v\Delta_{\td \varphi}\pd {\td \varphi}t\Big)\,\oo_{\td \varphi}^n. \label{eq:cs}
\eeq
Using (\ref{eq:A019}), we have
\beqn &&\Big\|e^{\eta s t}\Big(DL_s|_{h_s}-DL_s|_{\td \varphi}\Big)v\Big\|_{C^{0, 0, \ga}(M\times [0, \infty), g)}\no\\
&\leq&
s C(\te, \Te)\Big(\|e^{\eta s t}\Na^2v\|_{C^{0, 0, \ga}(M\times [0, \infty), g)}\cdot \|\Na^4(h_s-\td \varphi)\|_{C^{0, 0, \ga}(M\times [0, \infty), g)}\no\\&&
+\|e^{\eta s t}\Na^2v\|_{C^{0, 0, \ga}(M\times [0, \infty), g)}\cdot \|\Na^3(h_s-\td \varphi)\|_{C^{0, 0, \ga}(M\times [0, \infty), g)}\no\\&&+\|e^{\eta s t}\Na^2v\|_{C^{0, 0, \ga}(M\times [0, \infty), g)}\cdot \|\Na^2(h_s-\td \varphi)\|_{C^{0, 0, \ga}(M\times [0, \infty), g)}\no\\
&&+\|e^{\eta s t}\Na^3v\|_{C^{0, 0, \ga}(M\times [0, \infty), g)}\cdot \|\Na^3(h_s-\td \varphi)\|_{C^{0, 0, \ga}(M\times [0, \infty), g)}\no\\&&+\|e^{\eta s t}\Na^3v\|_{C^{0, 0, \ga}(M\times [0, \infty), g)}\cdot\|\Na^2(h_s-\td \varphi)\|_{C^{0, 0, \ga}(M\times [0, \infty), g)}\no\\&&+\|e^{\eta s t}\Na^4v\|_{C^{0, 0, \ga}(M\times [0, \infty), g)}\cdot \|\Na^2(h_s-\td \varphi)\|_{C^{0, 0, \ga}(M\times [0, \infty), g)}\Big)\no\\
&&+(1-s)C(\te, \Te)\|e^{\eta s t}\Na^2v\|_{C^{0, 0, \ga}(M\times [0, \infty), g)}\cdot \|\Na^2(h_s-\td \varphi)\|_{C^{0, 0, \ga}(M\times [0, \infty), g)}\no\\
&\leq& C(\te, \Te)\|e^{\eta s t}v\|_{C^{4, 1, \ga}(M\times [0, \infty), g)}\cdot \|h_s-\td \varphi\|_{C^{4, 1, \ga}(M\times [0, \infty), g)}, \label{eq:A019x1}
\eeqn and by Lemma \ref{lem:B010} below we have
 \beqn &&
\Big\|e^{\eta s t}\Big(c_{s, v, h_s}(\tau)-c_{s, v, \td \varphi}(\tau)    \Big)\Big\|_{C^{0, 0, \ga}(M\times [0, \infty), g)}\no\\&\leq& C(\te, \Te)\|e^{\eta s t}v\|_{C^{4, 1, \ga}(M\times [0, \infty), g)}\cdot \|h_s-\td \varphi\|_{C^{4, 1, \ga}(M\times [0, \infty), g)}. \label{eq:A022x1}
\eeqn
Combining (\ref{eq:A020x1}), (\ref{eq:A019x1}) with (\ref{eq:A022x1}), we have
\beqn &&
\Big\|\Big(D\td L_s|_{h_s}-D\td L_s|_{\td \varphi}\Big)v\Big\|_{C_{\eta s, 0}^{0, 0, \ga}(M\times [0, \infty), g, g_{\td \varphi})}\no\\
&\leq&\Big\|e^{\eta s t}\Big(DL_s|_{h_s}-DL_s|_{\td \varphi}\Big)v\Big\|_{C^{0, 0, \ga}(M\times [0, \infty), g)}\no\\&&+\Big\|e^{\eta s t}\Big( c_{s, v, h_s}(\tau)-c_{s, v, \td \varphi}(\tau)    \Big)\Big\|_{C^{0, 0, \ga}(M\times [0, \infty), g)}\no\\
&\leq& C(\te, \Te)\|e^{\eta s t}v\|_{C^{4, 1, \ga}(M\times [0, \infty), g)}\cdot \|h_s-\td \varphi\|_{C^{4, 1, \ga}(M\times [0, \infty), g)}. \label{eq:E001}
\eeqn
Using Lemma \ref{lem:B001} and (\ref{eq:E001}), we have
\beqn &&
\Big\|(D\td L_s|_{\td \varphi})^{-1}\Big(D\td L_s|_{h_s}-D \td L_s|_{\td \varphi}\Big)(h_1-h_2)\Big\|_{C_{\eta s, 0}^{4, 1, \ga}(M\times [0, \infty), g, g_{\td \varphi})}\no\\
&\leq& C(n) s^{-\ka(\ga, n)} \Big\|\Big(D\td L_s|_{h_s}-D \td L_s|_{\td \varphi}\Big)(h_1-h_2)\Big\|_{C_{\eta s, 0}^{0, 0, \ga}(M\times [0, \infty), g, g_{\td \varphi})}\no\\
&\leq& C(n, g) s^{-\ka(\ga, n)} \|h_s-\td \varphi\|_{C^{4, 1, \ga}(M\times [0, \infty), g)}\cdot  \|h_1-h_2\|_{C_{\eta s, 0}^{4, 1, \ga}(M\times [0, \infty), g, g_{\td \varphi})}\no\\
&\leq&C(n, g)\dd s^{\al-\ka(\ga, n)}  \|h_1-h_2\|_{C_{\eta s, 0}^{4, 1, \ga}(M\times [0, \infty), g, g_{\td \varphi})}.\no
\eeqn
Choosing $\al=\ka(\ga, n)$ and small $\dd$ with $C(n, g)\dd<\frac 12$, we have
\beq
\|\Psi_f(h_1)-\Psi_f(h_2)\|_{C_{\eta s, 0}^{4, 1, \ga}(M\times [0, \infty), g, g_{\td \varphi})}\leq \frac 12 \|h_1-h_2\|_{C_{\eta s, 0}^{4, 1, \ga}(M\times [0, \infty), g, g_{\td \varphi})}.\no
\eeq
The lemma is proved.

\end{proof}

\begin{lem}\label{lem:B010}Under the assumption of Lemma \ref{lem:B007}, the function $c_{s, v, \varphi}(\tau)$ defined by (\ref{eq:cs})
satisfies
\beqn &&
\Big\|e^{\eta s t}\Big(c_{s, v, \varphi_1}(\tau)-c_{s, v, \varphi_2} (\tau)   \Big)\Big\|_{C^{0, 0, \ga}(M\times [0, \infty), g)}\no\\
&\leq& C(\te, \Te)\|e^{\eta s t}v\|_{C^{4, 1, \ga}(M\times [0, \infty), g)}\cdot \|\varphi_1- \varphi_2\|_{C^{4, 1, \ga}(M\times [0, \infty), g)}.\no
\eeqn

\end{lem}
\begin{proof}
  Let
\beq
B(v, \varphi)=(1-s)g_{\varphi}^{i\bar l}g_{\varphi}^{k\bar j}g_{i\bar j}v_{k\bar l}
-sRic_{ \varphi, l\bar k}v_{k\bar l}+\Delta_{\varphi}v\pd { \varphi}t. \no
\eeq We estimate each term of $B. $
 By Lemma \ref{lem011}, we have
\beq
\Big|Ric_{\varphi_1, i\bar j
}v_{ j\bar i}-Ric_{\varphi_2, i\bar j
}v_{ j\bar i}\Big|\leq C(\te, \Te)\Big(|\Na^2(\varphi_1- \varphi_2)|+|\Na^3(\varphi_1- \varphi_2)|+|\Na^4(\varphi_1- \varphi_2)|\Big)\cdot |\Na^2 v|.\no
\eeq
Moreover, by (\ref{eq:A060}) we have
\beq
|g_{\varphi_1}^{i\bar l}g_{\varphi_1}^{k\bar j}g_{i\bar j}v_{k\bar l}-g_{\varphi_2}^{i\bar l}g_{\varphi_2}^{k\bar j}g_{i\bar j}v_{k\bar l}|
\leq|\Na^2(\varphi_1- \varphi_2)|\cdot |\Na^2 v|.\no
\eeq
Note that
\beqn &&
\Big|\Delta_{\varphi_1}v\pd { \varphi_1}t-\Delta_{\varphi_2}v\pd { \varphi_2}t\Big|\leq\Big|\Delta_{\varphi_1}v\pd { \varphi_1}t-\Delta_{\varphi_2}v\pd { \varphi_1}t\Big|+\Big|\Delta_{\varphi_2}v\pd { \varphi_1}t-\Delta_{\varphi_2}v\pd { \varphi_2}t\Big|\no\\
&\leq&|\Na^2(\varphi_1-\varphi_2)|\cdot|\Na^2v|\cdot \Big|\pd {\varphi_1}t\Big|
+|\Delta_{\varphi_2}v|\cdot\Big|\pd { }t(\varphi_1-\varphi_2)\Big|. \no
\eeqn
Therefore, we have
\beqn &&
|B(v, \varphi_1)-B(v, \varphi_2)|\leq (1-s)|\Na^2(\varphi_1- \varphi_2)|\cdot |\Na^2 v|\no\\&&+s C(\te, \Te)\Big(|\Na^2(\varphi_1- \varphi_2)|+|\Na^3(\varphi_1- \varphi_2)|+|\Na^4(\varphi_1- \varphi_2)|\Big)\cdot |\Na^2 v|\no\\&&+|\Na^2(\varphi_1-\varphi_2)|\cdot|\Na^2v|\cdot \Big|\pd {\varphi_1}t\Big|
+|\Delta_{\varphi_2}v_2|\cdot\Big|\pd { }t(\varphi_1-\varphi_2)\Big|. \label{eq:C010v}
\eeqn
Note that
\beqn &&
|c_{s, v, \varphi_1}(\tau)-c_{s, v, \varphi_2} (\tau) |=\Big|\int_M\;B(v, \varphi_1)\,\oo_{\varphi_1}^n-\int_M\,B(v, \varphi_2)\,\oo_{\varphi_2}^n\Big|\no\\
&\leq&\Big|\int_M\;\Big(B(v, \varphi_1)-B(v, \varphi_2)\Big)\,\oo_{\varphi_1}^n\Big|+\Big|\int_M\;B(v, \varphi_2)\,\oo_{\varphi_1}^n-\int_M\,B(v, \varphi_2)\,\oo_{\varphi_2}^n\Big|.\no\\\label{eq:C011v}
\eeqn
Let $\varphi_\xi=\varphi_1+\xi(\varphi_2-\varphi_1)(\xi\in [0, 1])$. We have
\beqn
 \oo_{\varphi_2}^n-\oo_{\varphi_1}^n&=&\int_0^1\,\Delta_{\varphi_s}(\varphi_2-\varphi_1)\,
 \oo_{\varphi_s}^n. \label{eq:C012v}
\eeqn
Combining (\ref{eq:C010v})-(\ref{eq:C012v}), we have
\beqn &&
\Big\|e^{\eta s t}\Big(c_{s, v, \varphi_1}(\tau)-c_{s, v, \varphi_2} (\tau)    \Big)\Big\|_{C^{0, 0, \ga}(M\times [0, \infty), g)}\no\\
&\leq&\|e^{\eta s t}(B(v, \varphi_1)-B(v, \varphi_2))\|_{{C^{0, 0, \ga}(M\times [0, \infty), g)}}\no\\&&+\|e^{\eta s t}(B(v, \varphi_1)\|_{{C^{0, 0, \ga}(M\times [0, \infty), g)}}\cdot \|\Na^2(\varphi_1-\varphi_2)\|_{{C^{0, 0, \ga}(M\times [0, \infty), g)}}
\no\\&\leq& C(\te, \Te)\|e^{\eta s t}v\|_{C^{4, 1, \ga}(M\times [0, \infty), g)}\cdot \|\varphi_1- \varphi_2\|_{C^{4, 1, \ga}(M\times [0, \infty))}. \no
\eeqn
The lemma is proved.

\end{proof}

\subsection{Proof of Theorem \ref{theo:002}}
The proof of Theorem \ref{theo:002} is similar to that in Section \ref{sec3}, and we sketch the proof here. Fix $\psi_0\in \cH(\oo_g).$ We denote by $\varphi_0(x, t)(t\in [0, \infty))$ the solution of $J$-flow with the initial data $\psi_0$ in Lemma \ref{lem:B002}.   By Lemma \ref{lem:B005} we can find $\td \varphi_{N, s}$ satisfying (\ref{eq:C013})-(\ref{eq:C014}). Moreover, we have
\beq
\td \varphi_{N, s}(0)-\varphi_0(x, 0)=\td \varphi_{N, s}(0)-\psi_0=c\in \RR, \quad \forall\; x\in M.\no
\eeq
We denote by $\td \varphi:=\td \varphi_{N, s}$ for simplicity.
 By Lemma \ref{lem:B001}, there exist two constants $s_0\in (0, 1)$ and $ \eta_0>0$ such that for any $s\in (0, s_0)$ and $\eta\in (0, \eta_0)$ the operator
\beq
D\td L_s|_{\td \varphi}:  C_{\eta s, 0, 0}^{4, 1, \ga}(M\times [0, \infty), g, g_{\td \varphi})\ri C_{\eta s, 0}^{0, 0, \ga}(M\times [0, \infty), g, g_{\td \varphi}).\no
\eeq
is injective and surjective. Moreover, we have
\beq
\|(D\td L_s|_{\td \varphi})^{-1}\|\leq C(n, T) s^{-\ka(\ga, n)}.\no
\eeq
By Lemma \ref{lem:B007}, for any $f\in C_{\eta s, 0}^{0, 0, \ga}(M\times [0, \infty), g, g_{\td \varphi})$ the map \beqn \td \Psi_f :  C_{\eta s, 0, \psi_0}^{4, 1, \ga}(M\times [0, \infty), g, g_{\td \varphi})&\ri& C_{\eta s, 0, \psi_0}^{4, 1, \ga}(M\times [0, \infty), g, g_{\td \varphi})\no\\
  \varphi&\ri& \varphi+(D\td L_s|_{\td \varphi})^{-1}(f-\td L_s(\varphi))\no
\eeqn is a contraction, i.e., if $h_1(x, t), h_2(x, t)\in C_{\eta s, 0, \psi_0}^{4, 1, \ga}(M\times [0, \infty), g, g_{\td \varphi})$ satisfy
\beq
\|h_1(x, t)-\td \varphi\|_{C^{4, 1, \ga}(M\times [0,  \infty), g, g_{\td \varphi})}\leq s^{\al}\dd, \quad \|h_2(x, t)-\td \varphi\|_{C^{4, 1, \ga}(M\times [0,  \infty), g, g_{\td \varphi})}\leq s^{\al}\dd,\no
\eeq  we have
\beq
\|\td \Psi_f(h_1)-\td \Psi_f(h_2)\|_{C^{4, 1, \ga}_{\eta s, 0}(M\times [0,  \infty), g, g_{\td \varphi})}\leq \frac 12 \|h_1-h_2\|_{C^{4, 1, \ga}_{\eta s, 0}(M\times [0,  \infty), g, g_{\td \varphi})}.\no
\eeq
  Therefore, there exists $\bb(\ga, n)>0$ such that if $v\in C_{\eta s, 0}^{0, 0, \ga}(M\times [0, \infty), g, g_{\td \varphi})$ satisfying \beq \|v-\td L_s(\td \varphi)\|_{C_{\eta s, 0}^{0, 0, \ga}(M\times [0, \infty), g, g_{\td \varphi})}\leq \eta s^{\bb}, \label{eq:A030v}\eeq  we can find $\varphi\in  C_{\eta s, 0, \psi_0}^{4, 1, \ga}(M\times [0, \infty), g, g_{\td \varphi})$ with $\td L_s(\varphi)=v. $
 By (\ref{eq:C013}), for any $N>0$  we can assume that the function $\td \varphi:=\td \varphi_{N, s}$   satisfy the condition
\beq
\|\td L_s(\td \varphi)\|_{C_{\eta s, 0}^{0, 0, \ga}(M\times [0, \infty), g, g_{\td \varphi})}\leq C(n, g) s^N.\no
\eeq Taking $N>\bb$, the function $v=0$ satisfies the condition  (\ref{eq:A030v}). Therefore, we can find $\varphi_{\infty}(x, t)\in C_{\eta s, 0, \psi_0 }^{4, 1, \ga}(M\times [0, T], g, g_{\td \varphi})$ such that
$$\td L_s(\varphi_{\infty})=0.$$
The theorem is proved.

\section{Proof of Theorem \ref{theo:main1}}\label{sec5}

Combining Theorem \ref{theo:001a} with Theorem \ref{theo:002}, we have the result:
\begin{theo}
Let $(M, g)$ be a compact K\"ahler manifold with a cscK metric
$\oo_g.$ Consider  a family of twisted Calabi flow with $s\in [0, 1]$
\beq
\pd {\varphi_s}t=s(R(\varphi_s)-\un R)+(1-s)(n-\tr_{\varphi_s}\oo_g).\no
\eeq For any $\psi_0\in \cH(\oo_g)$, there exists $s_0\in (0, 1]$ such that for any $s\in [0, s_0)$ the twisted Calabi flow (\ref{eq:000}) with the initial K\"ahler potential $\psi_0$ exists for all time and converges smoothly to the  cscK metric
$\oo_g$.

\end{theo}
\begin{proof} Let $\psi_0\in \cH(\oo_g)$.
Consider the $J$ flow
\beq
\pd {\varphi}t=n-\tr_{\varphi}g,\quad \varphi(0)=\psi_0. \label{eq:B001}
\eeq We normalize $\psi_0$ such that $I(\psi_0)=0$. Note that $I(\varphi(x, t))$ is a constant independent of $t$, we have $I(\varphi(x, t))=0$ for all $t$.
Since the equation $\tr_{\varphi}g=n$ has a solution $\varphi=0$ in $\cH_0(\oo_g)$,  the solution $\varphi_0(x, t)$ of the flow (\ref{eq:B001}) on $M\times [0, \infty)$ exists for all time and converges smoothly to $0$. Therefore, for $\dd_0>0$ in Theorem \ref{theo:002} there exists $T>0$ such that
\beq
\|\varphi_0(x, t)\|_{C^{4, 1, \ga}(M, \oo_g)}\leq \frac 12\dd_0,\quad \forall\; t\geq T. \label{eq:B034}
\eeq
By Theorem \ref{theo:001a},  there exists $s_1>0$ such that for any $s\in (0, s_1)$  the solution $\varphi_s(x, t)$ of twisted Calabi flow with the K\"ahler potential $ \psi_0$  exists for $t\in [0, T]$ and satisfies
\beq
\|\varphi_s(x, t)- \varphi_0(x, t)\|_{C^{4, 1, \ga}(M\times [0, T], g)}\leq C(n, g, \te, \Te, T)s.\label{eq:B034a}
\eeq
Therefore, (\ref{eq:B034}) and (\ref{eq:B034a}) imply that
\beqs
\|\varphi_s(x, T)\|_{C^{4}(M, g)}&\leq& \|\varphi_s(x, T)- \varphi_0(x, T)\|_{C^{4}(M, g)}+\| \varphi_0(x, T)\|_{C^{4}(M, g)}\\
&\leq &C(n, g, \te, \Te, T)s+\frac 12\dd_0.
\eeqs Thus, we choose $s_1'\in (0, s_1)$ such that $C(n, g, \te, \Te, T)s_1'<\frac 12\dd_0$, and  for any $s\in (0, s_1')$ we have
\beq \|\varphi_s(x, T)\|_{C^{4}(M, g)}\leq \dd_0.\no
\eeq
By Theorem \ref{theo:002}, for $s\in (0, s_1')$ the solution $\varphi_s(x, t)$ exists for all time $t>T$ and converges exponentially  to the cscK metric $\oo_g$. The theorem is proved.

\end{proof}

\section{Proof of Theorem \ref{theo:main3}}

\subsection{Finite time existence}
In this section, we prove that if a twisted Calabi flow exists on a finite time interval, so does the nearby twisted Calabi flow. The proof is similar to that of Theorem \ref{theo:001a}, but here we don't need to construct the approximate solution $\td \varphi_{N, s} $ as we replace $\td \varphi_{N, s} $ by the solution of the given twisted Calabi flow.

\begin{theo}\label{theo:main3a}
Let $\al, s_0\in (0, 1), T>0$ and $(M, g)$ be a compact K\"ahler manifold with a K\"ahler metric
$\oo_g$.   For any $\psi_0\in \cH(\oo_g)$, there exists a constant $\dd_0\in (0, s_0)$ satisfying the following property. If the twisted Calabi flow (\ref{eq:000}) for $s=s_0$ with the initial data $\psi_0$ exists for $t\in [0, T]$, then the twisted Calabi flow for any $s\in (s_0-\dd, s_0+\dd)$ with the initial data $\psi_0$ also exists for $t\in [0, T]$ and satisfies
 \beq
\|\varphi_s(x, t)-\varphi_{s_0}(x, t)\|_{C^{4, 1, \ga}(M\times [0, T], g)}\leq C(n, g, T, \varphi_{s_0})|s-s_0|^{\al}.\no
\eeq

\end{theo}

We follow the argument of Theorem \ref{theo:001a}. First, we show the $L^2$ estimate the linearized equation of the twisted Calabi flow
\beqn
\pd {w}{\tau}+\Delta^2_{\td \varphi}w&=&s^{-1}f(x, \tau)+s^{-1}a_{l\bar k}(x, \tau)w_{k\bar l},\quad (x, \tau)\in M\times [0, T],\label{eq:B003x}\\
w(x, 0)&=&0, \label{eq:B004x}
\eeqn where $a_{l\bar k}$ is defined by (\ref{eq:A0007b}). Recall that  Lemma \ref{lem:A004} needs the condition that $s$ is small, but here this condition is missing. We modify  Lemma \ref{lem:A004} as the following result.

\begin{lem}\label{lem:A004x}Let $\td \varphi:=  \varphi_{s_0}(x, t)(t\in [0, T])$ be the solution of twisted Calabi flow for $s=s_0\in (0, 1)$, which implies that there exists two constants $\te, \Te>0$  on $M\times [0, T]$ such that
\beq
\oo_{\td \varphi(x, t)}\geq \te \oo_g, \quad \|\td \varphi\|_{C^{N'+\ga}(M\times [0, T], g)}\leq \Te.\no
\eeq
 If $w(x, \tau)$ is a solution of (\ref{eq:B003x})-(\ref{eq:B004x}),
there exists a constant $C( g,  \Te  )$ such that for any $\tau\in [0, T]$
\beq
\|w(\cdot, \tau)\|_{L^2(M, \oo_{\td \varphi(\tau)})}\leq C( \Te)e^{C(\Te)s^{-2}\tau}\sup_{t\in [0, \tau]}\|f(\cdot, t)\|_{L^2(M, \oo_{\td \varphi(t)})}. \no
\eeq

\end{lem}
\begin{proof}By direct calculation, we have
\beqn
\pd {}{\tau}\int_M\;w^2\oo_{\td \varphi}^n&=&\int_M\;2w\pd w{\tau}\,\oo_{ \td \varphi}^n+w^2
\pd {}{\tau}\oo_{ \td \varphi}^n\no\\
&\leq&\int_M\;2w\Big(-\Delta^2_{ \td \varphi}w+\frac 1sf
+\frac 1s a_{l\bar k}w_{k\bar l}\Big)\oo_{ \td \varphi}^n+C(\Te)\int_M\;w^2\oo_{ \td \varphi}^n\no\\
&\leq&-2\int_M\;(\Delta_{ \td \varphi}w)^2\oo_{ \td \varphi}^n+2s^{-1}\int_M\;a_{l\bar k}w_{k\bar l}w\,\oo_{\td \varphi}^n\no\\
&&+2s^{-1}\int_M\;fw+C(\Te)\int_M\;w^2\oo_{ \td \varphi}^n.\label{eq:A016ax}
\eeqn
Note that   we have
\beq
|a_{l\bar k}w_{k\bar l}w|\leq \frac 12 s |\Na\bar \Na w|^2+\frac {C(\Te)}{ s}w^2,\no
\eeq which implies that
\beq
\int_M\;a_{l\bar k}w_{k\bar l}w\,\oo_{\td \varphi}^n
\leq\frac 12 s\int_M\;(\Delta_{\td \varphi} w)^2+C(\Te)s^{-1}\int_M\; w^2\,\oo_{\td \varphi}^n. \label{eq:A017ax}
\eeq
 Combining (\ref{eq:A016ax}) with (\ref{eq:A017ax}), we have
\beqn &&
\pd {}{\tau}\int_M\;w^2\oo_{ \td \varphi}^n\leq
-\int_M\;(\Delta_{\td \varphi}w)^2\oo_{\td \varphi}^n+
 {C( \Te) }s^{-2}\int_M\;w^2\,\oo_{ \td \varphi}^n\no\\&&+2s^{-1}\Big(\int_M\,f^2\,\oo_{ \td \varphi}^n\Big)^{\frac 12}\Big(\int_M\;w^2\,\oo_{ \td \varphi}^n\Big)^{\frac 12}\no\\
&\leq&C(\Te)s^{-2}\int_M\;w^2\,\oo_{ \td \varphi}^n+2s^{-1}\Big(\int_M\,f^2\,\oo_{ \td \varphi}^n\Big)^{\frac 12}\Big(\int_M\;w^2\,\oo_{ \td \varphi}^n\Big)^{\frac 12}. \label{eq:A018ax}
\eeqn
Let
$ E(\tau)= \Big(\int_M\;w^2\oo_{\td \varphi}\Big)^{\frac 12}.
$
The inequality (\ref{eq:A018ax}) implies that
\beq
E'(\tau)\leq C( \Te)s^{-2} E(\tau)+s^{-1}\|f(\cdot, \tau)\|_{L^2(M, \oo_{\td \varphi(\tau)})}.\no
\eeq
Thus, we have
\beq
E(\tau)\leq C( \Te)s e^{C(\Te)s^{-2}\tau}\sup_{t\in [0, \tau]}\|f(\cdot, t)\|_{L^2(M, \oo_{\td \varphi(t)})}.  \no
\eeq
The lemma is proved.

\end{proof}

\begin{lem}\label{lem004x}  Let $\td \varphi$ be the function  in Lemma \ref{lem:A004x}. Then the operator
\beq
DL_s|_{\td \varphi}:  C^{4, 1, \ga}(M\times [0, T], g, 0)\ri C^{0, 0, \ga}(M\times [0, T], g).\no
\eeq
is injective and surjective and satisfies
\beq
\|(DL_s|_{\td \varphi})^{-1}\|\leq C(\te, \Te, n, g)e^{C(\Te)s^{-1} T} s^{-\ka(\ga, n)},\no
\eeq where $\ka(\ga, n)>1.$

\end{lem}
\begin{proof} We follow the argument of Lemma \ref{lem004}.
By (\ref{eq:A200}) in the proof of Lemma \ref{lem004}, we have the estimate
\beqn &&
\|w\|_{{C^{4, 1, \ga}(M\times [0, sT], g)}}\leq C(\te, \Te)s^{-1} \|\td f\|_{{C^{0, 0, \ga}(M\times [0, sT], g)}}\no\\&&+C(\te, \Te, n, g)\ee  \|w\|_{{C^{4, 1, \ga}(M\times [0, sT], g)}}+C(\te, \Te, n, g) s^{-\ka'-1} \max_{t\in [0, sT]}\|w\|_{L^2(M, g)}.\no\\\label{eq:A200x}\eeqn
Note that by Lemma \ref{lem:A004x} we have
\beq
\max_{t\in [0, sT]}\|w\|_{L^2(M, g)}\leq C(\Te)e^{C(\Te)s^{-1} T}\sup_{t\in [0, sT]}\|\td f(\cdot, t)\|_{L^2(M, \oo_{\td \varphi(t)})}.\label{eq:A201x}
\eeq
Thus, (\ref{eq:A200x}) and (\ref{eq:A201x}) imply that
\beqn &&
\|w\|_{{C^{4, 1, \ga}(M\times [0, sT], g)}}\no\\&\leq& C(\te, \Te, n, g)s^{-1} \|\td f\|_{{C^{0, 0, \ga}(M\times [0, sT], g)}}+C(\te, \Te, n, g)s^{-\ka'-1} \max_{t\in [0, sT]}\|w\|_{L^{2}(M, g)}\no\\
&\leq&C(\te, \Te, n, g)s^{-1} \|\td f\|_{{C^{0, 0, \ga}(M\times [0, sT], g)}}+C(\te, \Te, n, g) s^{-\ka'-1}e^{C(\Te)s^{-1} T}\max_{t\in [0, sT]}\|\td f\|_{L^{2}(M, g)}\no\\
&\leq&C(\te, \Te, n, g)s^{-\ka'-1}e^{C(\Te)s^{-1} T} \|\td f\|_{{C^{0, 0, \ga}(M\times [0, sT], g)}}. \no
\eeqn
 As in the proof of Lemma \ref{lem004}, (\ref{eq:A014})-(\ref{eq:A015}) imply that
\beq
\|u\|_{{C^{4, 1, \ga}(M\times [0, T], g)}}\leq C(\te, \Te, n, g)s^{-\ka'-1-\frac {\ga}4}e^{C(\Te)s^{-1} T} \| f\|_{{C^{0, 0, \ga}(M\times [0, T], g)}}.\no
\eeq
The lemma is proved.

\end{proof}

\begin{lem}\label{lem013x}Let $\td \varphi$ be the function  in Lemma \ref{lem:A004x},  and $f\in C^{0, 0, \ga}(M\times [0, T], g)$. Define \beqn\Psi_f :  C^{4, 1, \ga}(M\times [0, T], g, \psi_0)&\ri& C^{4, 1, \ga}(M\times [0, T], g, \psi_0)\no\\
  \varphi&\ri& \varphi+(DL_s|_{\td \varphi})^{-1}(f-L_s(\varphi)).\no
\eeqn There exists  $\dd_1=\dd_1(g, n, \te, \Te, s_0,   T)>0$ satisfying the following properties.
If $h_1(x, t) $ and $ h_2(x, t)$ satisfy
\beq
\|h_1(x, t)-\td \varphi\|_{C^{4, 1, \ga}(M\times [0, T], g)}\leq \dd_1|s-s_0|^{\al}, \quad \|h_2(x, t)-\td \varphi\|_{C^{4, 1, \ga}(M\times [0, T], g)}\leq \dd_1|s-s_0|^{\al},\no
\eeq where $\al>0,$ then for any $s\in (\frac 12s_0, 1]$ we have
\beq
\|\Psi_f(h_1)-\Psi_f(h_2)\|_{C^{4, 1, \ga}(M\times [0, T], g)}\leq \frac 12 \|h_1-h_2\|_{C^{4, 1, \ga}(M\times [0, T], g)}.\no
\eeq
\end{lem}

\begin{proof} As in the proof of Lemma \ref{lem013}, using (\ref{eq:A020}) and Lemma \ref{lem004x} we have
\beqn &&
\Big\|(DL_s|_{\td \varphi})^{-1}\Big(DL_s|_{h_s}-DL_s|_{\td \varphi}\Big)(h_1-h_2)\Big\|_{C^{4, 1, \ga}(M\times [0, T], g)}\no\\
&\leq& C(\te, \Te, n, g) e^{C(\Te)s^{-1} T} s^{-\ka(\ga, n)} \Big\|\Big(DL_s|_{h_s}-DL_s|_{\td \varphi}\Big)(h_1-h_2)\Big\|_{C^{0, 0, \ga}(M\times [0, T], g)}\no\\
&\leq& C(\te, \Te, n, g)  e^{C(\Te)s^{-1} T} s^{-\ka(\ga, n)} \|h_s-\td \varphi\|_{C^{4, 1, \ga}(M\times [0, T], g)}\cdot  \|h_1-h_2\|_{C^{4, 1, \ga}(M\times [0, T], g)}\no\\
&\leq&C(\te, \Te, n, g)\dd_1|s-s_0|^{\al} e^{C(\Te)s^{-1} T} s^{-\ka(\ga, n)} \|h_1-h_2\|_{C^{4, 1, \ga}(M\times [0, T], g)}. \label{eq:A022}
\eeqn
We can choose $\dd_1>0$ small such that for any $s\in (\frac 12s_0, 1]$
\beqn
&& C(\te, \Te, n, g)\dd_1 e^{C(\Te)s^{-1} T} s^{-\ka(\ga, n)}\no \\
&\leq& C(\te, \Te, n, g)\dd_1  e^{C(\Te)s_0^{-1} T}s_0^{-\ka(\ga, n)}\leq \frac 12. \label{eq:A049b}
\eeqn
Thus, by (\ref{eq:A022}) and (\ref{eq:A049b}) we have
\beq
\|\Psi_f(h_1)-\Psi_f(h_2)\|_{C^{4, 1, \ga}(M\times [0, T], g)}\leq \frac 12 \|h_1-h_2\|_{C^{4, 1, \ga}(M\times [0, T], g)}.\no
\eeq
The lemma is proved.

\end{proof}

\begin{proof}[Proof of Theorem \ref{theo:main3a}]
(1). Fix $\psi_0\in \cH(\oo_g).$
  We denote by $\td \varphi:=\varphi_{s_0}(x, t)(t\in [0, \infty))$ the solution of twisted Calabi flow at $s=s_0$ with the initial data $\psi_0,$ i.e. $L_{s_0}(\varphi_{s_0}(x, t))=0$. Then we have
  \beq
  |L_s(\varphi_{s_0}(x, t))| =|L_s(\varphi_{s_0}(x, t))-L_{s_0}(\varphi_{s_0}(x, t)) |\leq C(\td \varphi)|s-s_0|.\label{eq:A500x}
  \eeq
 By Lemma \ref{lem004x}, for any $s\in (0, 1)$ the operator
\beq
DL_s|_{\td \varphi}:  C^{4, 1, \ga}(M\times [0, T], g, 0)\ri C^{0, 0, \ga}(M\times [0, T], g) \no
\eeq
is injective and surjective, and satisfies the inequality
\beq
\|(DL_s|_{\td \varphi})^{-1}\|\leq C(\te, \Te, n, g)e^{C(\Te)s^{-1} T} s^{-\ka(\ga, n)}.\no
\eeq
By Lemma \ref{lem013x}, there exists $\dd_1(g, n, \te, \Te, s_0,   T)>0$ such that for any $s\in (\frac 12s_0, 1]$  and any $f\in C^{0, 0, \ga}(M\times [0, T], g)$ the map \beqn\Psi_f :  C^{4, 1, \ga}(M\times [0, T], g, \psi_0)&\ri& C^{4, 1, \ga}(M\times [0, T], g, \psi_0)\no\\
  \varphi&\ri& \varphi+(DL_s|_{\td \varphi})^{-1}(f-L_s(\varphi))\no
\eeqn is a contraction and satisfies
\beq
\|\Psi_f(h_1)-\Psi_f(h_2)\|_{C^{4, 1, \ga}(M\times [0, T], g)}\leq \frac 12 \|h_1-h_2\|_{C^{4, 1, \ga}(M\times [0, T], g)}, \no
\eeq where $h_1(x, t) $ and $ h_2(x, t)$ satisfy
\beqn
\|h_1(x, t)-\td \varphi\|_{C^{4, 1, \ga}(M\times [0, T], g)}&\leq& \dd_1|s-s_0|^{\al}, \\ \|h_2(x, t)-\td \varphi\|_{C^{4, 1, \ga}(M\times [0, T], g)}&\leq& \dd_1|s-s_0|^{\al}. \label{eq:A029}
\eeqn Here we choose $\al\in (0, 1)$.

 (2). We next show that there exists $\dd_2=\dd_2(\te, \Te, n, g, T, s_0, \ga)>0$ such that  if $v\in C^{0, 0, \ga}(M\times [0, T], g)$ satisfies \beq \|v-L_s(\td \varphi)\|_{C^{0, 0, \ga}(M\times [0, T], g)}\leq \dd_2|s-s_0|^{\al}, \label{eq:A030z}\eeq  we can find $\varphi\in  C^{4, 1, \ga}(M\times [0, T], g, \psi_0)$ with $L_s(\varphi)=v. $
For any $k\in \NN$ we define
\beq
\varphi_k=\Psi_v^{k-1}(\td \varphi).\no
\eeq
Then $\varphi_1=\td \varphi$ and  we have
\beqs &&
\|\varphi_2-\varphi_1\|_{C^{4, 1, \ga}(M\times [0, T], g)}=
\|\varphi_2-\td \varphi\|_{C^{4, 1, \ga}(M\times [0, T], g)}\\&=&\|(DL_s|_{\td \varphi})^{-1}(v-L_s(\td \varphi))\|_{C^{4, 1, \ga}(M\times [0, T], g)}\\
&\leq&C(\te, \Te, n, g) e^{C(\Te)s^{-1} T} s^{-\ka(\ga, n)}\|v-L_s(\td \varphi)\|_{C^{0, 0, \ga}(M\times [0, T], g)}\\
&\leq &C(\te, \Te, n, g)e^{C(\Te)s_0^{-1} T} s_0^{-\ka(\ga, n)}\dd_2|s-s_0|^{\al}\leq\frac 12 \dd_1|s-s_0|^{\al},
\eeqs  where $\dd_2=\dd_2(\te, \Te, n, g, T, s_0, \ga)>0$ is small.
We claim that for any $k\geq 1$ we have
\beq
\|\varphi_{k+1}-\varphi_{k}\|_{C^{4, 1, \ga}(M\times [0, T], g)}\leq \frac {1}{2^{k}}\dd_1|s-s_0|^{\al}. \label{eq:A028}
\eeq
By induction, we assume that for $k=1, 2,\cdots, j-1$ the inequality (\ref{eq:A028}) holds. Then for $k=j$ we have
\beqn
\|\varphi_{j}-\varphi_1\|_{C^{4, 1, \ga}(M\times [0, T], g)}&\leq&\sum_{i=1}^{j-1}\,
\|\varphi_{i+1}-\varphi_i\|_{C^{4, 1, \ga}(M\times [0, T], g)}\no\\&\leq&\sum_{i=1}^{j-1}\,
\frac 1{2^{i}} \dd_1|s-s_0|^{\al}\leq \dd_1|s-s_0|^{\al}.\label{eq:A090}
\eeqn
Thus, $\varphi_j$ satisfies the condition (\ref{eq:A029}) and we have
\beqn &&
\|\varphi_{j+1}-\varphi_j\|_{C^{4, 1, \ga}(M\times [0, T], g)}=\|\Psi(\varphi_j)-\Psi(\varphi_{j-1})\|_{C^{4, 1, \ga}(M\times [0, T], g)}\no\\&\leq &
\frac 12 \|\varphi_j-\varphi_{j-1}\|_{C^{4, 1, \ga}(M\times [0, T], g)}\leq \frac 1{2^j} \dd_1|s-s_0|^{\al}.
\eeqn Thus, (\ref{eq:A028}) is proved. Therefore, $\varphi_k$ converges to a limit
function $\varphi_{\infty}$ in $C^{4, 1, \ga}(M\times [0, T], g, \psi_0)$ and $L_{s}(\varphi_{\infty})=v$.\\

 (3). By (\ref{eq:A500x}),  the function $v=0$ satisfies the condition  (\ref{eq:A030z}) when $\al\in (0, 1)$ and  $|s-s_0|$ is sufficiently small. Therefore, there exists small $\dd_0$ such that for any $s\in (s_0-\dd_0, s_0+\dd_0)$ we can find $\varphi_{\infty}(x, t)\in C^{4, 1, \ga}(M\times [0, T], g, \psi_0)$ such that
$$L_s(\varphi_{\infty})=0.$$
Moreover,  (\ref{eq:A090}) implies that
\beqs
\|\varphi_{\infty}-\varphi_{s_0}\|_{C^{4, 1, \ga}(M\times [0, T], g)}&=&
\|\varphi_{\infty}-\td \varphi\|_{C^{4, 1, \ga}(M\times [0, T], g)}\\
&\leq& \dd_1(g, n, \te, \Te, T) |s-s_0|^{\al}.\no
\eeqs
The theorem is proved.

\end{proof}

\subsection{Proof of Theorem \ref{theo:main3}}

To prove Theorem \ref{theo:main3}, we need to use the stability result of twisted Calabi flow.  Chen-He in \cite{[ChenHe1]} showed the stability  of   usual Calabi flow near cscK metrics, and using similar arguments as in \cite{[ChenHe1]} He-Li showed the stability of twisted Calabi flow (\ref{eq:000}) in \cite{[HZ]}. Here we recall the result in \cite{[HZ]}:

 \begin{theo}\label{theo:main3b} (cf. He-Li \cite{[HZ]})
Let $(M, g)$ be a compact K\"ahler manifold of complex dimension $n$ with a twisted cscK metric
$\oo_g$. For any $s\in (0, 1)$, there exist two constants $\ee_0=C(g)s^{p}(1-s)>0$ with $p\geq 6n+2$ such that for  any $\psi_0\in \cH(\oo_g)$ satisfies
\beq
\|\psi_0\|_{C^{3, \ga}(M, g)}\leq \ee_0, \label{eq:z1}
\eeq the twisted Calabi flow (\ref{eq:000}) with the initial K\"ahler potential $\psi_0$ exists for all time and converges exponentially to $\oo_g$.

\end{theo}

 Theorem \ref{theo:main3b} implies that for any $s_0\in (0, 1)$, there exist $\dd=\min\{\frac {s_0}2, \frac {1-{s_0}}2\}$ and $\ee_0=\ee_0(g, n, s_0)>0$ such that for any $\psi_0\in \cH(\oo_g)$ satisfies (\ref{eq:z1}) the twisted Calabi flow (\ref{eq:000}) for any $s\in (s_0-\dd, s_0+\dd)$ with the initial K\"ahler potential $\psi_0$ exists for all time and converges exponentially to $\oo_g$. This result together with Theorem \ref{theo:main3a} implies Theorem \ref{theo:main3}. See the proof of Theorem \ref{theo:main1} in Section \ref{sec5} for details.

\begin{appendices}

\section{The  estimates for parabolic  equations}
\subsection{The Schauder estimates for solutions to fourth-order  equations}\label{subsec1}
In this subsection, we collect some notations and results in Metsch \cite{[Me]}.

\begin{defi}\label{defi:A001}(cf. Sec.1.1 of \cite{[Me]})Let $\Om\subset \RR^n$ be an open and bounded set.
\begin{enumerate}
  \item[(1).]
For any function $u(x)$ on $\Om\subset \RR^n$, we define
\beqs
\|u\|_{C^m(\Om)}&:=& \sum_{k=0}^m\sup_{x\in \Om}|\Na^k u|,\\
\left[u\right]_{\al, \Om}&:=&\sup_{x, y\in \Om, x\neq y}\frac {|u(x)-u(y)|}{|x-y|^{\al}},\\
\left[\Na^m u\right]_{\al, \Om}&:=&\sum_{|p|=m}\,\left[\Na^p u\right]_{\al},\\
\|u\|_{C^{m, \ga}(\Om)}&:=&\|u\|_{C^m(\Om)}+\left[\Na^m u\right]_{\ga, \Om}.
\eeqs Here $m$ is a nonnegative integer and $\ga \in (0, 1)$.

\item[(2).] We define the parabolic H\"older norms of a function $u(x, t)$ on $\Om\times (0, T)$
\beqs
\|u\|_{C^{m, [m/4]}(\Om\times (0, T))}&:=& \sum_{4j+|p|\leq m}\sup_{(x, t)\in \Om\times (0, T)}|\partial_t^j\Na^p u(x, t)|,\\
\left[ u\right]^{\mathrm{space}}_{\al, \Om\times (0, T)}&:=&
\sup_{x, y\in \Om, x\neq y,  t\in (0, T)}\frac {|u(x, t)-u(y, t)|}{|x-y|^{\al}},\\
\left[ u\right]^{\mathrm{time}}_{\bb, \Om\times (0, T)}&:=&
\sup_{x\in \Om,   t, s\in (0, T), t\neq s}\frac {|u(x, t)-u(x, s)|}{|t-s|^{\bb}},\\
\left[ u\right]^{(m)}_{\ga, \Om\times (0, T)}&:=&\sum_{4j+k=m}
\left[\partial_t^j\Na^k u\right]^{\mathrm{space}}_{\ga, \Om\times (0, T)}+
\sum_{0<\frac {m+\ga-k}{4}-j<1}\, \left[ \partial_t^j\Na^k u\right]^{\mathrm{time}}_{\frac {m+\ga-k}{4}-j, \Om\times (0, T)},\\
\left[D^{4, 1}u\right]^{(0)}_{\ga, \Om\times (0, T)}&:=&\left[\Na^4u\right]^{(0)}_{\ga, \Om\times (0, T)}+\left[\partial_tu\right]^{(0)}_{\ga, \Om\times (0, T)},\\
\|u\|_{C^{m, [m/4], \ga}(\Om\times (0, T))}&:=&\|u\|_{C^{m, [m/4]}(\Om\times (0, T))}+\left[ u\right]^{(m)}_{\ga, \Om\times (0, T)}.
\eeqs

\end{enumerate}

\end{defi}

By the definition  we have
\beq
\|u\|_{C^{0, 0, \ga}(\Om\times (0, T))}=\sup_{(x, t)\in \Om\times (0, T)}| u(x, t)|+\left[ u\right]^{\mathrm{space}}_{\ga, \Om\times (0, T)}+ \left[  u\right]^{\mathrm{time}}_{\ga/4, \Om\times (0, T)}, \label{eq:A042}
\eeq
 and $\|u\|_{C^{4, 1, \ga}(\Om\times (0, T))}$ is equivalent to the norm
\beqn &&
\|u\|'_{C^{4, 1, \ga}(\Om\times (0, T))}:=\|u\|_{C^{4, 1}(\Om\times (0, T))}
+\left[D^{4, 1}u\right]^{(0)}_{\ga, \Om\times (0, T)}\no\\&=&
\sum_{|p|\leq 4}\sup_{(x, t)\in \Om\times (0, T)}|\Na^p u(x, t)|+\sup_{(x, t)\in \Om\times (0, T)}|\partial_t u(x, t)|+\left[\Na^4 u\right]^{\mathrm{space}}_{\ga, \Om\times (0, T)}\no\\
&&+\left[\Na^4 u\right]^{\mathrm{time}}_{\ga/4, \Om\times (0, T), g}+
\left[\partial_t u\right]^{\mathrm{space}}_{\ga, \Om\times (0, T)}+
\left[\partial_t u\right]^{\mathrm{time}}_{\ga/4, \Om\times (0, T)}. \label{eq:A012}
\eeqn

The H\"older norms can be defined on Riemannian manifolds.

\begin{defi}\label{defi:A002}(cf. Sec.1.3 of \cite{[Me]}) Let $(M, g)$ be a compact Riemannian manifold without boundary.
\begin{enumerate}
  \item[(1).] A finite atlas $\cA=\{(\varphi_i, U_i)\;|\; 1\leq i\leq m\}$ of $M$ consisting of boundary $\varphi_i: U_i\subset M\ri V_i\subset \RR^n$ is called a good atlas of $M$, if
      \begin{enumerate}
        \item For all $1\leq i\leq m$ the sets $V_i$ are convex,
        \item For all $1\leq i\leq m$, there exists an open set $\td U_i\supset\supset U_i$, an open and convex set $\td V_i\supset\supset V_i$ and a smooth extension $\td \varphi_i: \td U_i\ri \td V_i$ of $\varphi_i.$
      \end{enumerate}
  \item[(2).] Let  $\cA=\{(\varphi_i, U_i)\;|\; 1\leq i\leq m\}$ be a good atlas of $M_i$, $V_i=\varphi_i(U_i)$ and $\ga\in (0, 1)$. The function $u\in C^k(M, g)$ is in $ C_{\cA}^{k, \ga}(M, g)$, if
      \beq
      \|u\|_{C_{\cA}^{k, \ga}(M, g)}:=\max_{1\leq i\leq m}\|u\circ \varphi_i^{-1}\|_{C^{k, \ga}(V_i)}<\infty. \no
      \eeq
      Then the space $C_{\cA}^{k, \ga}(M, g) $ is independent of good atlas by Lemma 1.3.3 of \cite{[Me]}, and we denote it by $C^{k, \ga}(M, g)$.
  \item[(3). ] Let  $\cA=\{(\varphi_i, U_i)\;|\; 1\leq i\leq m\}$ be a good atlas of $M_i$, $V_i=\varphi_i(U_i)$ and $\ga\in (0, 1)$. The function $u\in C^{k, [k/4]}(M\times [0, T], g)$ is in $ C_{\cA}^{k, [k/4], \ga}(M\times [0, T], g)$, if
      \beq
      \|u\|_{C_{\cA}^{k, [k/4], \ga}(M\times [0, T], g)}:=\max_{1\leq i\leq m}\|u\circ \varphi_i^{-1}\|_{C^{k, [k/4], \ga}(V_i\times [0, T], g)}<\infty.\no
      \eeq
      Then the space $C_{\cA}^{k, [k/4], \ga}(M\times [0, T], g) $ is independent of good atlas by Lemma 1.3.3 of \cite{[Me]}, and we denote it by $C^{k, [k/4], \ga}(M\times [0, T], g)$.
      \item[(4).] Similarly, we can define other H\"older spaces in Definition \ref{defi:A001}.
\end{enumerate}

\end{defi}

\begin{lem}\label{lem:Me}(cf. Theorem 1.2.3 of \cite{[Me]}) Let $x_0\in \RR^n\times \RR, \rho>0, U_{\rho}:=U_{\rho}(x_0)$ and $u\in C^{4, 1}(U_{\rho})$. There exists $\ee_0(n)>0$ such that for all $\ee\in (0, \ee_0)$, $0\leq k\leq 4$ and $0\leq l\leq 3$ we have
\beqn
\|\Na^kw\|_{L^{\infty}(U_{\rho})}&\leq& \ee \left[D^{4, 1}w\right]^{(0)}_{\ga, U_{\rho}}+C(m)\ee^{-\frac {k}{4-k+\ga}}\|w\|_{L^{\infty}(U_{\rho})}, \no\\
\left[\Na^l w\right]^{(0)}_{\ga, U_{\rho}}&\leq& \ee [D^{4, 1}w]^{(0)}_{\ga, U_{\rho}}+C(m) \ee^{-\frac {l+\ga}{4-l}}\|w\|_{L^{\infty}(U_{\rho})}.\label{eq:A041}
\eeqn

\end{lem}

\begin{theo}\label{theo:Me}(cf. Theorem 4.1.6 of \cite{[Me]}) Let $L: C^4(M)\ri C^0(M)$ be an operator defined by
\beq
Lu:=a^{ij}a^{kl}\partial_{ijkl}(u)+\sum_{|\al|\leq 3}A_{\al}\Na^{\al}u,\no
\eeq where the coefficients satisfy the conditions
\beqn
\|L\|_{C^{0, \ga}(M)}&:=&\max_{|\al|\leq 3}\|A_{\al}\|_{C^{0, \ga}(M)}+\max_{1\leq i, j\leq n}\|a^{ij}\|_{C^{0, \ga}(M)}<+\infty,\no\\
\Te(L)&:=&\sup\Big\{\te>0\;\Big|\;a^{ij}(x)\xi_i\xi_j\geq \te|\xi|^2,\quad \forall\; \xi\in \RR^m\Big\}<+\infty.\no
\eeqn For any $u\in C^{4, 1, \ga}(M\times [0, T])$ with $T\geq T_0$,  we have
\beq
\|u\|_{C^{4, 1, \ga}(M\times [0, T], g)}\leq C\Big(
\|Lu\|_{C^{0, 0, \ga}(M\times [0, T], g)}+\|u(\cdot, 0)\|_{C^{4,  \ga}(M, g)}
+\sup_{0\leq t\leq T}\|u\|_{L^2(M, g)}\Big),\no
\eeq where $C$ depends on $M, T_0, \ga, \|L\|_{C^{0, \ga}(M)}, \Te(L)$ but not on $T$.

\end{theo}

We have the following interpolation inequalities.
\begin{lem}\label{lem:A5} Let $(M, g)$ be a compact Riemannian manifold of real dimension $m$. There exist $C(m, g), \ka(\ga, m)>0$ satisfying the following properties. For any function $w(x, t)\in C^{4, 1, \ga}(M\times [0, T], g)$, we have
\beqn &&
\|w\|_{C^{0, 0, \ga}(M\times [0, T], g)}+\|\Na^2 w\|_{{C^{0, 0, \ga}(M\times [0, T], g)}}\no\\&\leq& \ee  \|w\|_{{C^{4, 1, \ga}(M\times [0, T], g)}}+C(m, g) \ee^{-\ka(\ga, m)}\max_{\tau\in [0, T]}\|w(\cdot, \tau)\|_{L^2(M, g)}. \label{eq:A039x}
\eeqn

\end{lem}
\begin{proof}It suffices to show \label{eq:A039} for parabolic balls in $\RR^n\times \RR.$ Let $x_0\in\RR^n$ and $U_{\rho}(x_0)=B_{\rho}(x_0)\times (t_0-\rho^4, t_0+\rho^4).$

(1).   For any $x\in U_{\si}(x_0)$, we have
\beqn
|w(x, t)|&\leq&\frac 1{\vol_g(B_{\si}(x_0))}\Big|\int_{B_{\si}(x_0)}\,
w(y, t)\,dy\Big|\no\\&&+\frac 1{\vol_g(B_{\si}(x_0))}\int_{B_{\si}(x_0)}\,
|w(y, t)-w(x, t)|\,dy\no\\
&\leq&(\vol_g(B_{\si}(x_0)))^{-\frac 12}\Big(\int_{B_{\si}(x_0)}\,
|w(y, t)|^2\,dy\Big)^{\frac 12}\no\\
&&+\frac 1{\vol_g(B_{\si}(x_0))}\int_{B_{\si}(x_0)}\,\left[ w\right]_{\ga, B_{\si}(x_0)}
|x-y|^{\ga}\,dy\no\\
&\leq&(\vol_g(B_{\si}(x_0)))^{-\frac 12}\|w\|_{L^2(B_{\si}(x_0))}
+\si^{\ga}\,\left[ w\right]_{\ga, B_{\si}(x_0)}.\label{eq:A024}
\eeqn
Let $\ee=\si^{\ga}$, we have
\beq
(\vol_g(B_{\si}(x_0)))^{-\frac 12}=C(m)\si^{-\frac m2}=C( m)\ee^{-\frac m{2\ga}}. \label{eq:A024a}
\eeq
Therefore, (\ref{eq:A024}) and (\ref{eq:A024a}) imply that
\beq
|w|_{L^{\infty}(U_{\si}(x_0) )}\leq \ee \left[ w\right]_{\ga, U_{\si}(x_0)}+C(m)\ee^{-\frac {m}{2\ga}}\sup_{\tau\in I_{\rho}} \|w(\cdot, \tau)\|_{L^2(B_{\si}(x_0))} \label{eq:A025}
\eeq where $I_{\rho}=(t_0-\rho^4, t_0+\rho^4).$
By (\ref{eq:A041}), we have
\beq
\left[ w\right]^{(0)}_{\ga, U_{\si}}\leq \ee' [D^{4, 1}w]^{(0)}_{\ga, U_{\si}}+C(m) \ee'^{-\frac {\ga}{4}}\|w\|_{L^{\infty}( U_{\si})}. \label{eq:A045}
\eeq
Taking $\ee'=1$ in (\ref{eq:A045}) and using (\ref{eq:A025}) , we have
\beqs
\|w\|_{L^{\infty}( U_{\si})}&\leq& \ee \left[ w\right]_{\ga,  U_{\si}}+C(m, g)\ee^{-\frac {m}{2\ga}}\|w\|_{L^2( U_{\si})}\\
&\leq&\ee   [D^{4, 1}w]^{(0)}_{\ga, U_{\si}}+C(m)\ee  \|w\|_{L^{\infty}(U_{\si})}+C(m, g)\ee^{-\frac {m}{2\ga}}\sup_{\tau\in I_{\rho}} \|w(\cdot, \tau)\|_{L^2(B_{\si}(x_0))}.
\eeqs
Choosing $C(m)\ee<\frac 12$, we have
\beqn
\|w\|_{L^{\infty}(U_{\si})}
&\leq&\ee   [D^{4, 1}w]^{(0)}_{\ga, U_{\si}}+C(m)\ee^{-\frac {m}{2\ga}}\sup_{\tau\in I_{\rho}} \|w(\cdot, \tau)\|_{L^2(B_{\si}(x_0))}\no\\
&\leq&\ee   \|w\|_{C^{4, 1, \ga}(U_{\si})}+C(m)\ee^{-\frac {m}{2\ga}}\sup_{\tau\in I_{\rho}} \|w(\cdot, \tau)\|_{L^2(B_{\si}(x_0))}.\label{eq:A046}
\eeqn

(2).
By Lemma \ref{lem:Me} and (\ref{eq:A042}), we have
\beqn &&
\|\Na^2w\|_{C^{0, 0, \ga}(U_{\si}, g)}=\|\Na^2 w\|_{L^{\infty}(U_{\si})}+\left[ \Na^2w\right]^{(0)}_{\ga, U_{\si}}\no\\
&\leq& 2\ee \left[D^{4, 1}w\right]^{(0)}_{\ga, U_{\si}}+C(m)\ee^{-\frac {2}{2+\ga}}\|w\|_{L^{\infty}(U_{\si})}+C(m) \ee^{-\frac {2+\ga}{2}}\|w\|_{L^{\infty}(U_{\si})}\no\\
&\leq&2\ee  \|w\|_{{C^{4, 1, \ga}(U_{\si})}}+C(m)\ee^{-\frac {2+\ga}{2}}\|w\|_{L^{\infty}(U_{\si})}.\no
\eeqn
Combining this with (\ref{eq:A046}), we have
\beqs &&
\|\Na^2w\|_{C^{0, 0, \ga}(U_{\si}, g)}\leq 2\ee  \|w\|_{{C^{4, 1, \ga}(U_{\si})}}\\&&+C(m)\ee^{-\frac {2+\ga}{2}}\Big(\ee'   \|w\|_{C^{4, 1, \ga}(U_{\si}, g}+C(m, g)\ee'^{-\frac {m}{2\ga}}\sup_{\tau\in I_{\rho}} \|w(\cdot, \tau)\|_{L^2(B_{\si}(x_0))}\Big)\\
&\leq&2\ee  \|w\|_{{C^{4, 1, \ga}(U_{\si}, g)}}+C\ee\|w\|_{C^{4, 1, \ga}(U_{\si}), g}+C \ee^{-\frac {2\ga^2+4\ga+m\ga+4m}{4\ga}}\sup_{\tau\in I_{\rho}} \|w(\cdot, \tau)\|_{L^2(B_{\si}(x_0))},
\eeqs where we choose $\ee'=\ee^{\frac {4+\ga}{2}}.$
Therefore, we have
\beq
\|\Na^2w\|_{C^{0, 0, \ga}(U_{\si}, g)}\leq \ee  \|w\|_{{C^{4, 1, \ga}(U_{\si}, g)}}+C \ee^{-\ka}\sup_{\tau\in I_{\rho}} \|w(\cdot, \tau)\|_{L^2(B_{\si}(x_0))}, \label{eq:A042a}
\eeq where $\ka=\frac {2\ga^2+4\ga+m\ga+4m}{4\ga}$.

(3). By (\ref{eq:A042}) and (\ref{eq:A041}), we have
\beqn &&
\|w\|_{C^{0, 0, \ga}(U_{\si})}=\sup_{(x, t)\in U_{\si}}| w(x, t)|+\left[ w\right]^{(0)}_{\ga, U_{\si}}\no\\
&\leq&\sup_{(x, t)\in U_{\si}}| w|+\ee [D^{4, 1}w]^{(0)}_{\ga, U_{\rho}}+C(m) \ee^{-\frac {\ga}{4}}\|w\|_{L^{\infty}(U_{\rho})}\no\\
&\leq&\ee [D^{4, 1}w]^{(0)}_{\ga, U_{\rho}}+C(m) \ee^{-\frac {\ga}{4}}\|w\|_{L^{\infty}(U_{\rho})}\no\\
&\leq&\ee [D^{4, 1}w]^{(0)}_{\ga, U_{\rho}}+C(m) \ee^{-\frac {\ga}{4}}\Big(\ee   \|w\|_{C^{4, 1, \ga}(U_{\si})}+C(m)\ee^{-\frac {m}{2\ga}}\sup_{\tau\in I_{\rho}} \|w(\cdot, \tau)\|_{L^2(B_{\si}(x_0))} \Big)\no\\
&=&\ee [D^{4, 1}w]^{(0)}_{\ga, U_{\rho}}+C(m)\ee^{1-\frac {\ga}4}   \|w\|_{C^{4, 1, \ga}(U_{\si})}+C(m)\ee^{-\frac {m}{2\ga}-\frac {\ga}4}\sup_{\tau\in I_{\rho}} \|w(\cdot, \tau)\|_{L^2(B_{\si}(x_0))}\no\\
&\leq&(\ee+C(m)\ee^{1-\frac {\ga}4}) \|w\|_{C^{4, 1, \ga}(U_{\si})}+C(m)\ee^{-\frac {m}{2\ga}-\frac {\ga}4}\sup_{\tau\in I_{\rho}} \|w(\cdot, \tau)\|_{L^2(B_{\si}(x_0))}.\no
\eeqn
Therefore, we have
\beq
\|w\|_{C^{0, 0, \ga}(U_{\si})}\leq \ee_1 \|w\|_{C^{4, 1, \ga}(U_{\si})}+C(m)\ee_1^{-\ka'(\ga, m)}\sup_{\tau\in I_{\rho}} \|w(\cdot, \tau)\|_{L^2(B_{\si}(x_0))}, \label{eq:A042b}
\eeq where $\ee_1>0$ and $\ka'(\ga, m)>0.$  Combining (\ref{eq:A042a}) with (\ref{eq:A042b}), the inequality (\ref{eq:A039x}) holds on $U_{\si}.$

(3). By choosing a fixed good atlas and following the same   argument of Section 4 of \cite{[Me]}, we can show that (\ref{eq:A042a}) holds on a compact Riemannian manifold $(M, g)$. Here we omit the details.

\end{proof}

\begin{lem}\label{lem:B008z} Let
\beq
\tau=st,\quad w(x, \tau)=u(x, t).\no
\eeq Then we have
\beqn
\|u\|_{C^{4, 1, \ga}(\Om\times (0, T), g)}&\leq& \|w\|_{C^{4, 1, \ga}(\Om\times (0, sT), g)}\leq s^{-1-\frac {\ga}4}\|u\|_{C^{4, 1, \ga}(\Om\times (0, T), g)},\no\\
\|u\|_{C^{0, 0, \ga}(\Om\times (0, T), g)}&\leq& \|w\|_{C^{0, 0, \ga}(\Om\times (0, sT), g)}\leq  s^{-\frac {\ga}4} \|u\|_{C^{0, 0, \ga}(\Om\times (0, T), g)}.\no
\eeqn

\end{lem}
\begin{proof}
By(\ref{eq:A012}), we have
\beqn&&
\|w\|'_{C^{4, 1, \ga}(\Om\times (0, sT), g)}\no\\&=&
\sum_{|p|\leq 4}\sup_{(x, \tau)\in \Om\times (0, sT)}|\Na^p w(x, \tau)|+\sup_{(x, \tau)\in \Om\times (0, sT)}|\partial_{\tau} w(x, \tau)|\no\\
&&+\left[\Na^4 w\right]^{\mathrm{space}}_{\ga, \Om\times (0, sT), g}+\left[\Na^4 w\right]^{\mathrm{time}}_{\ga/4, \Om\times (0, sT), g}\no\\&&+
\left[\partial_{\tau} w\right]^{\mathrm{space}}_{\ga, \Om\times (0, sT), g}+
\left[\partial_{\tau} w\right]^{\mathrm{time}}_{\ga/4, \Om\times (0, sT), g}\no\\
&=&\sum_{|p|\leq 4}\sup_{(x, t)\in \Om\times (0, T)}|\Na^p u(x, t)|+s^{-1}\sup_{(x, t)\in \Om\times (0, T)}|\partial_{t} u(x, t)|\no\\
&&+\left[\Na^4 u\right]^{\mathrm{space}}_{\ga, \Om\times (0, T), g}+s^{-\frac {\ga}4}\left[\Na^4 u\right]^{\mathrm{time}}_{\ga/4, \Om\times (0, T), g}\no\\&&+s^{-1}
\left[\partial_{t} u\right]^{\mathrm{space}}_{\ga, \Om\times (0, T), g}+s^{-1-\frac {\ga}4}
\left[\partial_{t} u\right]^{\mathrm{time}}_{\ga/4, \Om\times (0, T), g}.\no
\eeqn
Thus, we have
\beq
\|u\|_{C^{4, 1, \ga}(\Om\times (0, T), g)}\leq \|w\|_{C^{4, 1, \ga}(\Om\times (0, sT), g)}\leq s^{-1-\frac {\ga}4}\|u\|_{C^{4, 1, \ga}(\Om\times (0, T), g)}.\no
\eeq
Moreover, by (\ref{eq:A042}) we have
\beqs
\|w\|_{C^{0, 0, \ga}(\Om\times (0, sT), g)}&=&\sup_{(x, \tau)\in \Om\times (0, sT)}| w(x, \tau)|+\left[ w\right]^{\mathrm{space}}_{\ga, \Om\times (0, sT), g}+ \left[  w\right]^{\mathrm{time}}_{\ga/4, \Om\times (0, sT), g}\\
&=&\sup_{(x, t)\in \Om\times (0, T)}| u(x, t)|+\left[ u\right]^{\mathrm{space}}_{\ga, \Om\times (0, T), g}+ s^{-\frac {\ga}4}\left[  u\right]^{\mathrm{time}}_{\ga/4, \Om\times (0, T), g}.\no
\eeqs Thus, we have
\beq
\|u\|_{C^{0, 0, \ga}(\Om\times (0, T), g)}\leq \|w\|_{C^{0, 0, \ga}(\Om\times (0, sT), g)}\leq  s^{-\frac {\ga}4} \|u\|_{C^{0, 0, \ga}(\Om\times (0, T), g)}.\no
\eeq

\end{proof}

\subsection{The existence of  solutions to fourth-order  equations}
In this subsection, we show the existence of  solutions to a general fourth-order linear parabolic equations. The existence of such solutions in weighted H\"older spaces is proved by He-Zeng \cite{[HZ]} by the biharmonic heat kernel method, and
the existence  in usual H\"older spaces of Subsection \ref{subsec1}  is proved by Metsch \cite{[Me]} by using the Galerkin approximation method. Here we provide an alternative proof for the existence in usual H\"older spaces by using the biharmonic heat kernel method developed by He-Zeng \cite{[HZ]}.

Let $(M, g)$ be a closed Riemannian manifold of dimension $m$. We define the biharmonic heat kernel $b_g(x, y; t)$ with respect to the metric $g$ by
\beq
\Big(\pd {}{t}+\Delta_g^2\Big)b_g(x, y; t)=0, \no
\eeq and for any continuous function $u$ on $M$,
\beq
\lim_{t\ri 0^+}\int_M\; b_g(x, y; t)u(y)\,dV_g(y)=u(x).\no
\eeq
The biharmonic heat kernel satisfies the following properties.
\begin{lem}\label{lem:HZ}(cf. Theorem 3.1 and Theorem 3.2 of He-Zeng \cite{[HZ]})
\begin{enumerate}
  \item[(1).] For any $p, q, k\geq 0$ and any $(x, y, t)\in M\times M\times (0, T)$, we have
\beq
\Big|\partial_t^k\Na_x^p\Na_y^q b_g(x, y; t)\Big|_g\leq C t^{-\frac {m+4k+p+q}{4}}\exp\Big(-\dd (t^{-\frac 14}\rho(x, y))^{\frac 43}\Big). \no
\eeq
  \item[(2).] For any point $p\in M$, there exists a local coordinate chart $U\subset\RR^n\ri M$ near $p$ such that for any $x, y\in U$ and any multi-indices $\al, \bb$, we have for any $(x, y, t)\in U\times U\times (0, T)$
\beq \no
\Big| D_x^{\bb}(D_x+D_y)^{\al} b_g(x, y; t)\Big|_g\leq C t^{-\frac {m+|\bb|}{4}}\exp\Big(-\dd (t^{-\frac 14}\rho(x, y))^{\frac 43}\Big),
\eeq where $D_x$ and $D_y$ are ordinary derivatives in $\RR^n$, and $C, \dd>0 $
depend on $T, g, |\al|+|\bb|.$
\end{enumerate}

\end{lem}

Following the argument of Theorem 3.10 of He-Zeng \cite{[HZ]}, we show the result:

\begin{theo}\label{theo:A1x} Let $(M, g)$ be a closed Riemannian manifold of dimension $m$. Consider the problem
 \beqn
\Big(\pd {}{t}+\Delta_g^2\Big)u(x, t)&=&f(x, t), \label{eq:H001}\\
u(x, 0)=0.\label{eq:H002}
\eeqn
For any $f\in C^{0, 0, \ga}(M, g)$, there exists a unique solution $u\in C^{4, 1, \ga}(M, g)$ of (\ref{eq:H001})-(\ref{eq:H002}). Moreover, $u$ satisfies
\beq
\|u\|_{C^{4, 1, \ga}(M, g)}\leq C\|f\|_{C^{0, 0, \ga}(M, g) }. \label{eq:H003}
\eeq

\end{theo}

\begin{proof} Let
\beq
V[f](x, t)=\int_0^t\,\int_M\;b_g(x, y; t-s)f(y, s)dV_g(y)ds.\no
\eeq
We show that $V[f]$ satisfies (\ref{eq:H001})-(\ref{eq:H002}) with the estimate (\ref{eq:H003}). The proof is divided into several steps.

(1). We show that $V[f](x, t)$ is a well-defined function satisfying (\ref{eq:H002}). In fact, by Lemma \ref{lem:HZ} we have
\beq
|V[f](x, t)|\leq C\|f\|_{C^0(M\times [0, T])}\int_0^t\,\int_M\,(t-s)^{-\frac {m}{4}}
\exp\Big(-\dd (t-s)^{-\frac 13}\rho(x, y)^{\frac 43}\Big)\,dV_g(y)ds. \label{eq:H0013}
\eeq
We choose a local coordinate chart  around $x$ such that $\frac 12 |x-y|\leq \rho(x, y)\leq 2|x-y|$ for  $y\in B_{r_0}(x)\subset M.$ Thus, we have
\beqn
&&\int_0^t\,\int_M\,(t-s)^{-\frac {m}{4}}
\exp\Big(-\dd (t-s)^{-\frac 13}\rho(x, y)^{\frac 43}\Big)\,dV_g(y)ds\no\\
&\leq& \int_0^t\,\int_{B_{r_0}(0)}\,(t-s)^{-\frac {m}{4}}
\exp\Big(- 2^{-\frac 43} \dd(t-s)^{-\frac 13}|w|^{\frac 43}\Big)\,dwds\no\\
&&+\int_0^t\,\int_{M\backslash B_{r_0}(0)}\,(t-s)^{-\frac {m}{4}}
\exp\Big(-\dd  (t-s)^{-\frac 13}r_0^{\frac 43}\Big)\,dV_g(y)ds\no\\
&:=&I_1+I_2.\label{eq:H0014}
\eeqn We estimate $I_1$ and $I_2$.
\beq
I_1\leq C\int_0^t\,\int_0^{r_0}\,(t-s)^{-\frac {m}{4}}
\exp\Big(- 2^{-\frac 43} \dd(t-s)^{-\frac 13}\tau^{\frac 43}\Big)\tau^{m-1}d\tau ds
\leq Ct.\label{eq:H0015}
\eeq Moreover, we have
\beqn
I_2&\leq &C \int_0^t\,\,(t-s)^{-\frac {m}{4}}
\exp\Big(-\dd  (t-s)^{-\frac 13}r_0^{\frac 43}\Big)\,ds\no\\
&\leq &C \int_{t^{-\frac 13}}^{\infty}\,\tau^{\frac 34m-4}\,e^{-\dd r_0^{\frac 43} \tau} \,d\tau\ri 0,\quad t\ri 0.\label{eq:H0016}
\eeqn Thus, by (\ref{eq:H0013})-(\ref{eq:H0016}) the function $V[f](x, t)$ is well-defined and $V[f](x, 0)=0.$ Moreover, $V[f](x, t)$ satisfies
\beq
|V[f](x, t)|\leq C \|f\|_{C^0(M\times [0, T])}.\no
\eeq

(2). We estimate $|\Na^k V[f]|$ for integer $1\leq k\leq 3$. In fact, we have
\beq
\Na^kV[f](x, t)=\int_0^t\int_M\;\Na^k_xb_g(x, y; t-s)f(y, s)dV_g(y)ds.\no
\eeq
By Lemma \ref{lem:HZ}, we have
\beqn
| \Na^kV[f]|&\leq &C\|f\|_{C^0}\int_0^t\,(t-s)^{-\frac {m+k}{4}}\int_M\;\exp\Big(-\dd (t-s)^{-\frac 13}\rho(x, y)^{\frac 43}\Big)dV_g(y)ds\no\\
  &\leq&C\|f\|_{C^0}\int_0^t\,(t-s)^{-\frac {m+k}{4}}\int_{ B_{r_0}(0)}\;\exp\Big(-\dd (t-s)^{-\frac 13}(2^{-1}|w|)^{\frac 43}\Big)dwds\no\\
  &&+C\|f\|_{C^0}\int_0^t\,(t-s)^{-\frac {m+k}{4}}\int_{M\backslash B_{r_0}(0)}\;\exp\Big(-\dd (t-s)^{-\frac 13}\rho(x, y)^{\frac 43}\Big)dV_g(y)ds\no\\
  &\leq &C\|f\|_{C^0}\int_0^t\,(t-s)^{-\frac {k}{4}}  \,ds+C\|f\|_{C^0}\int_0^t\,(t-s)^{-\frac {m+k}{4}}
  \exp\Big(-\dd (t-s)^{-\frac 13} r_0^{\frac 43}\Big)\,ds\no\\&<&\infty.\label{eq:H004}
   \eeqn
   Therefore, (\ref{eq:H004})   implies that for $1\leq k\leq 3$
   \beq
   | \Na^kV[f]|\leq C\|f\|_{C^0(M\times [0, T])}.\no
   \eeq
       Note that the last inequality of (\ref{eq:H004}) doesn't hold for $k=4. $

 (3). We next show that $\Na^3V[f](x, t)$ is differentiable at $x$. By the equality (90) of He-Zeng \cite{[HZ]}, in local coordinates for small $h$ we have
 \beqn&&
 \Na^3V[f](x+he_l, t)-\Na^3V[f](x, t)\no\\
 &=&\int_0^t\,\int_M\, \int_0^h\,\pd {}{x_l}\Na^3b_g(x+\tau e_l, y; t-s)(f(y, s)-f(x+\tau e_l, s))\,d\tau dV_g(y)ds\no\\
 &&+\int_0^t\,\int_M\, \int_0^h\,\pd {}{x_l}\Na^3b_g(x+\tau e_l, y; t-s) f(x+\tau e_l, s) \,d\tau dV_g(y)ds.  \label{eq:H006}
 \eeqn
 Note that
  \beqn &&
\int_0^t\,\int_M\,\Big|\pd {}{x_l}\Na^3b_g(x+\tau e_l, y; t-s)(f(y, s)-f(x+\tau e_l, s))\Big|\, dV_g(y)ds\no\\
&\leq&C[f]_{\al, M\times [0, T]}^{\mathrm{space}}\int_0^t\,\int_M\, (t-s)^{-\frac {m+4}{4}} \exp\Big(-\dd (t-s)^{-\frac 13}\rho(x, y)^{\frac 43}\Big)\rho(x+\tau e_l, y)^{\al} \, dV_g(y)ds\no\\
&\leq&C[f]_{\al, M\times [0, T]}^{\mathrm{space}}\int_0^t\,(t-s)^{-\frac {4-\al}{4}}\,ds\leq C[f]_{\al, M\times [0, T]}^{\mathrm{space}}. \label{eq:H010}
 \eeqn
Moreover, by the inequalities (92) and (93) of He-Zeng \cite{[HZ]} we have
\beqn &&
\Big|\int_0^t\,\int_M\,\pd {}{x_l}\Na^3b_g(x+\tau e_l, y; t-s) f(x+\tau e_l, s)\, dV_g(y)ds \Big|\no\\
&\leq&C\|f\|_{C^0(M\times [0, T])}\int_0^t\,(t-s)^{-\frac 34}\,ds\leq C\|f\|_{C^0(M\times [0, T])}. \label{eq:H011}
\eeqn
Therefore, the right-hand side of (\ref{eq:H006}) is integrable and we can change the orders of integrations by Fubini's theorem. Thus,  we have
\beqn&&
 \Na^3V[f](x+he_l, t)-\Na^3V[f](x, t)\no\\
 &=&\int_0^h\,\Big(\int_0^t\int_M\,\pd {}{x_l}\Na^3b_g(x+\tau e_l, y; t-s)(f(y, s)-f(x+\tau e_l, s))\, dV_g(y)ds\no\\
 &&+\int_0^t\,\int_M\,\pd {}{x_l}\Na^3b_g(x+\tau e_l, y; t-s) f(x+\tau e_l, s) \,  dV_g(y)ds\Big)\,d\tau.  \label{eq:H007}
 \eeqn
We denote by
\beqn
P(x, t)&=&\int_0^t\int_M\,\pd {}{x_l}\Na^3b_g(x, y; t-s)(f(y, s)-f(x, s))\, dV_g(y)ds\no\\
 &&+\int_0^t\,\int_M\,\pd {}{x_l}\Na^3b_g(x, y; t-s) f(x, s) \,  dV_g(y)ds.\no
\eeqn

\begin{claim} \label{claim1}
For any $x, x'\in M$,
\beqn
|P(x, t)|&\leq &C\|f\|_{C^{0, 0, \al}(M\times [0, T], g)}, \label{eq:H008}\\
|P(x, t)-P(x', t)|&\leq &C\|f\|_{C^{0, 0, \al}(M\times [0, T], g)}\rho(x, x')^{\al}. \label{eq:H009}
\eeqn
\end{claim}
\begin{proof}[Proof of Claim \ref{claim1}]
In fact, (\ref{eq:H008}) follows from (\ref{eq:H010}) and (\ref{eq:H011}). To prove (\ref{eq:H009}), we write
\beq
P(x, t)-P(x', t)=P_1+P_2+P_3+P_4, \label{eq:H016}
\eeq where $P_i$ are defined by
\beqn
P_1&:=&\int_0^t\int_{B_r(\xi)}\,\pd {}{x_l}\Na^3b_g(x, y; t-s)(f(y, s)-f(x, s))\, dV_g(y)ds,\no\\
P_2&:=&-\int_0^t\int_{B_r(\xi)}\,\pd {}{x_l}\Na^3b_g(x', y; t-s)(f(y, s)-f(x', s))\, dV_g(y)ds,\no\\
P_3&:=&\int_0^t\int_{M\backslash B_r(\xi)}\,\Big(\pd {}{x_l}\Na^3b_g(x, y; t-s)-\pd {}{x_l}\Na^3b_g(x', y; t-s)\Big)(f(y, s)-f(x, s))\, dV_g(y)ds\no\\
&&+\int_0^t\int_{M\backslash B_r(\xi)}\,\pd {}{x_l}\Na^3b_g(x', y; t-s)(f(x', s)-f(x, s))\, dV_g(y)ds,\no\\
P_4&:=&\int_0^t\Big(f(x, s)\int_M\, \pd {}{x_l}\Na^3b_g(x, y; t-s)\,dV_g(y)\no\\
&&-f(x', s)\int_M\,\pd {}{x_l}\Na^3b_g(x', y; t-s)\,dV_g(y)\Big)ds.\no
\eeqn
We next estimate each $P_i$. By Lemma \ref{lem:HZ}, we have
\beqn &&
|P_1|\no\\&\leq&C\|f\|_{C^{0, 0, \al}} \int_0^t\int_{B_{2r}(x)}\,\Big|\pd {}{x_l}\Na^3b_g(x, y; t-s)\Big|\rho(x, y)^{\al}\, dV_g(y)ds\no\\
&\leq &  C\|f\|_{C^{0, 0, \al}} \int_0^t\int_{B_{2r}(x)}\,(t-s)^{-\frac {m+4}{4}}\exp\Big(-\dd (t-s)^{-\frac 13}\rho(x, y)^{\frac 43}\Big)\rho(x, y)^{\al}\, dV_g(y)ds\no\\
&=&C\|f\|_{C^{0, 0, \al}} \int_{B_{2r}(x)}\,|x-y|^{\al}\Big(\int_0^t\,(t-s)^{-\frac {m+4}{4}}\exp\Big(-\dd (t-s)^{-\frac 13}|x-y|^{\frac 43}\Big) \, ds\Big)\,dV_g(y)\no\\
&\leq &C\|f\|_{C^{0, 0, \al}} \int_{B_{2r}(x)}\,|x-y|^{\al-m}\, \,dV_g(y)\no\\
&\leq&C\|f\|_{C^{0, 0, \al}}r^{\al},\label{eq:H017}
\eeqn where $r=|x-y|$ and we write $\|f\|_{C^{0, 0, \al}}:=\|f\|_{C^{0, 0, \al}(M\times [0, T], g)}$ for short.
 Similarly, we have
\beq
|P_2|=\Big|\int_0^t\int_{B_r(\xi)}\,\pd {}{x_l}\Na^3b_g(x', y; t-s)(f(y, s)-f(x', s))\, dV_g(y)ds\Big|\leq C\|f\|_{C^{0, 0, \al}}r^{\al}.  \label{eq:H018}
\eeq
Next, following the argument of the inequality (101) of He-Zeng \cite{[HZ]} we estimate $P_3:=P_{3a}+P_{3b}$ where
\beqn &&
|P_{3a}|\no\\&:=&\Big|\int_0^t\int_{M\backslash B_r(\xi)}\,\Big(\pd {}{x_l}\Na^3b_g(x, y; t-s)-\pd {}{x_l}\Na^3b_g(x', y; t-s)\Big)(f(y, s)-f(x, s))\, dV_g(y)ds\Big|\no\\
&\leq&C\|f\|_{C^{0, 0, \al}}|x-x'|\int_0^t\int_{|y-\xi|\geq r}\,(t-s)^{-\frac {m+5}{4}}\exp\Big(-\dd (t-s)^{-\frac 13}|\xi-y|^{\frac 43}\Big)|y-\xi|^{\al}\, dV_g(y)ds\no\\
&=&C\|f\|_{C^{0, 0, \al}}|x-x'|\int_{|y-\xi|\geq r}\,|y-\xi|^{\al}\Big(\int_0^t(t-s)^{-\frac {m+5}{4}}\exp\Big(-\dd (t-s)^{-\frac 13}|\xi-y|^{\frac 43}\Big)ds\Big)\, dV_g(y)\no\\
&\leq&C\|f\|_{C^{0, 0, \al}}|x-x'|\int_{|y-\xi|\geq r}\,|y-\xi|^{\al-m-1}\,dV_g(y)\no\\
&\leq&C\|f\|_{C^{0, 0, \al}}|x-x'| r^{\al-1}\leq C\|f\|_{C^{0, 0, \al}} r^{\al},\label{eq:H019}
\eeqn and
\beqn
|P_{3b}|&:=&\Big|\int_0^t\int_{M\backslash B_r(\xi)}\,\pd {}{x_l}\Na^3b_g(x', y; t-s)(f(x', s)-f(x, s))\, dV_g(y)ds\Big|\no\\
&\leq&\Big|\int_0^t\,(f(x', s)-f(x, s))\Big(\int_{M\backslash B_r(\xi)}\,\pd {}{x_l}\Na^3b_g(x', y; t-s)\, dV_g(y)\Big)\,ds\Big|\no\\
&\leq&C\|f\|_{C^{0, 0, \al}}|x-x'|^{\al}\int_0^t\,(t-s)^{-\frac 34}\,ds\leq C\|f\|_{C^{0, 0, \al}} r^{\al}. \label{eq:H020}
\eeqn
Next, we estimate $P_4.$ By the inequalities (106) of He-Zeng \cite{[HZ]}, we have
\beqn
&&\Big|\int_M\, \pd {}{x_l}\Na^3b_g(x, y; t-s)\,dV_g(y)\Big|\leq C(t-s)^{-\frac 34},\label{eq:H012}\\
&&\Big|\int_M\, \pd {}{x_l}\Na^3b_g(x, y; t-s)\,dV_g(y)-\int_M\, \pd {}{x_l}\Na^3b_g(x', y; t-s)\,dV_g(y)\Big|\no\\
&\leq&C (t-s)^{-\frac{3+\al}{4} }|x-x'|^{\al}.
\eeqn Therefore, we have
\beqn
|P_4|
&\leq&\Big|\int_0^t\,(f(x, s)-f(x', s))\int_M\, \pd {}{x_l}\Na^3b_g(x, y; t-s)\,dV_g(y)\Big|\no\\
&&+\Big|\int_0^t\,f(x', s)\Big(\int_M\,\pd {}{x_l}\Na^3b_g(x, y; t-s)\,dV_g(y)\no\\&&-\int_M\,\pd {}{x_l}\Na^3b_g(x', y; t-s)\,dV_g(y)\Big)ds\Big|\no\\
&\leq&C\|f\|_{C^{0, 0, \ga}}|x-x'|^{\al}\int_0^t\,(t-s)^{-\frac 34}\,ds\no\\&&
+C\|f\|_{C^0}|x-x'|^{\al}\int_0^t\,(t-s)^{-\frac {3+\al}4}\,ds\leq C\|f\|_{C^{0, 0, \ga}}r^{\al}.  \label{eq:H021}
\eeqn
 Thus, (\ref{eq:H009}) follows (\ref{eq:H016})-(\ref{eq:H020}) and (\ref{eq:H021}).
 \end{proof}

 By (\ref{eq:H007}), we have
 \beq
 \Na^3V[f](x+he_l, t)-\Na^3V[f](x, t)=\int_0^h\,P( x+\tau e_l)\,d\tau.\no
 \eeq
Therefore, by Claim \ref{claim1} we have
\beq
\frac {\Na^3V[f](x+he_l, t)-\Na^3V[f](x, t)}{h}\ri P(x),\quad h\ri 0.\no
\eeq Since $P(x)$ satisfies (\ref{eq:H009}), $\Na^4V[f](x, t)$ is H\"older continuous with respect to $x$ and we have
\beq [\Na^4V[f](x, t)]_{\al}^{\mathrm{space}}\leq C\|f\|_{C^{0, 0, \ga}}.\label{eq:H023}\eeq

(3). We show that $\Na^4V[f](x, t)$ is H\"older continuous with respect to $t$. By
(110) of He-Zeng \cite{[HZ]}, we have
\beq
\Na^4V[f](x, t+h)-\Na^4V[f](x, t):=J_1+J_2+J_3, \label{eq:H024}
\eeq where $J_i$ are defined by
\beqs
J_1&:=&\int_t^{t+h}\int_M\,\Na^4b_g(x, y; t+h-s)(f(y, s)-f(x, s))\,dV_g(y)ds, \\
J_2&:=&\int_t^{t+h}\, f(x, s) \Big(\int_M\,\Na^4b_g(x, y; t+h-s)\,dV_g(y)\Big)\,ds,\\
J_3&:=&\int_0^{h}\,\Big(\int_0^{t}\int_M\,\pd {}t\Na^4b_g(x, y; t+\tau-s)(f(y, s)-f(x, s))\,dV_g(y)ds\\
&&+\int_0^{t}\,f(x, s)\Big(\int_M\,\pd {}t\Na^4b_g(x, y; t+\tau-s)\,dV_g(y)\Big)ds \Big)\,d\tau.
\eeqs
We now estimate each $J_i$. By Lemma \ref{lem:HZ}, we have
\beqn
|J_1|&=&\Big|\int_0^{h}\int_M\,\Na^4b_g(x, y; \tau)(f(y, t+h-\tau)-f(x, t+h-\tau))\,dV_g(y)ds \Big|\no\\
&\leq&C\|f\|_{C^{0, 0, \al}}\int_0^{h}\int_M\,\tau^{-\frac {m+4}{4}}\exp\Big(-\dd \tau^{-\frac 13}\rho(x, y)^{\frac 43}\Big)\rho(x, y)^{\al}\,dV_g(y)d\tau\no\\
&\leq&C\|f\|_{C^{0, 0, \al}}\int_0^{h}\,\tau^{-\frac {4-\al}{4}}\,d\tau\leq C\|f\|_{C^{0, 0, \al}}h^{\frac {\al}4}. \label{eq:H025}
\eeqn Moreover, by (\ref{eq:H012}) we have
\beqn
|J_2|&=&\Big|\int_t^{t+h}\, f(x, s) \Big(\int_M\,\Na^4b_g(x, y; t+h-s)\,dV_g(y)\Big)\,ds \Big|\no\\
&\leq &C\|f\|_{C^0}\int_t^{t+h}\, (t+h-s)^{-\frac 34}\,ds\leq C\|f\|_{C^0}h^{\frac {1}4}\leq \|f\|_{C^0}h^{\frac {\al}4}, \label{eq:H026}
\eeqn where we used Lemma \ref{lem:HZ2} below in the first inequality.
We write $J_3:=J_{3a}+J_{3b}$, where $J_{3a}$ and $J_{3b}$ are  estimated by
\beqn
J_{3a}&\leq&\int_0^{h}\,\Big(\int_0^{t}\int_M\,\pd {}t\Na^4b_g(x, y; t+\tau-s)(f(y, s)-f(x, s))\,dV_g(y)ds\Big)d\tau\no\\&\leq&
C\|f\|_{C^{0, 0, \al}}\int_0^{h}\,\Big(\int_0^{t}\int_M\,(t+\tau-s)^{-\frac {m+8}{4}}\exp\Big(-\dd (t+\tau-s)^{-\frac 13}\no\\&&\cdot\rho(x, y)^{\frac 43} \Big)\rho(x, y)^{\al}\,dV_g(y)ds \Big)d\tau\no\\
&\leq& C\|f\|_{C^{0, 0, \al}}\int_0^{h}\,\Big(\int_0^{t}\,(t+\tau-s)^{-\frac {8-\al}{4}}ds \Big)d\tau\no\\
&\leq&C\|f\|_{C^{0, 0, \al}} h^{\frac {\al}4},\label{eq:H027}
\eeqn and
\beqn
J_{3b}&\leq&\Big|\int_0^{h}\,\int_0^{t}\,f(x, s)\Big(\int_M\,\pd {}t\Na^4b_g(x, y; t+\tau-s)\,dV_g(y)\Big)ds \Big)\,d\tau \Big|\no\\
&\leq&C\|f\|_{C^0}\int_0^{h}\,\Big(\int_0^{t}\,(t+\tau-s)^{-\frac 74}\,ds\Big)d\tau\no\\
&\leq &C\|f\|_{C^0} \int_0^h\,\tau^{-\frac 34}\,d\tau
\leq C\|f\|_{C^0}h^{\frac 14}\leq C\|f\|_{C^0} h^{\frac {\al}4}, \label{eq:H028}
\eeqn where we used Lemma \ref{lem:HZ2} below in the second inequality.
Combining the inequalities (\ref{eq:H024})-(\ref{eq:H028}), we have
\beqn
\frac {|\Na^4V[f](x, t+h)-\Na^4V[f](x, t)|}{h^{\frac {\al}4}}\leq C\|f\|_{C^{0, 0, \ga}}. \no
\eeqn
Therefore, $\Na^4V[f](x, t)$ is H\"older continuous with respect to $t$ and we have \beq
[\Na^4V[f](x, t)]_{\al}^{\mathrm{time}}\leq C\|f\|_{C^{0, 0, \ga}}. \label{eq:H029}
\eeq

(4). Finally, following the argument of (116) in He-Zeng \cite{[HZ]} we have
\beqn
\pd {}tV[f](x, t)&=&f(x, t)+\int_0^t\,\int_M\; \pd {}tb_g(x, y; t-s)(f(y, s)-f(x, s))\,dV_g(y)ds\no\\
&&+\int_0^t\,\Big(\int_M\; \pd {}tb_g(x, y; t-s)\,dV_g(y)\Big)f(x, s)\,ds,\no
\eeqn  which implies that $V[f]$ satisfies
\beq
\pd {}tV[f](x, t)+\Delta_g^2V[f](x, t)=f(x, t).\no
\eeq in the classical sense, and by (\ref{eq:H023}) and (\ref{eq:H029}) $\pd {}tV[f]$ is H\"older continuous with respect to $x$ and $t$. Combining the above estimates, we have
\beq
\|V[f]\|_{C^{4, 1, \ga}}\leq C\|f\|_{C^{0, 0, \ga}}.\no
\eeq The uniqueness of solutions follows directly from the Schauder estimate Theorem \ref{theo:Me}. Thus, the theorem is proved.

\end{proof}

\begin{lem}\label{lem:HZ2}For any integer $k\geq 1$, we have
\beqn
\Big|\int_M\,\Na^kb_g(x, y; t-s)\,dV_g(y)\Big|&\leq& C(t-s)^{-\frac {m+k-1}{4}}. \label{eq:J004}
\eeqn

\end{lem}
\begin{proof} For any given point $x\in M$, there exists a local coordinate chart $V$ with $x\in V$. Then
\beqn &&
\Big|\int_M\,\Na^{k+1}b_g(x, y; t-s)\,dV_g(y)\Big|\no\\
&\leq&\Big|\int_U\, \Na^{k+1}b_g(x, y; t-s)\,dV_g(y)\Big|+
\Big|\int_{M\backslash U}\, \Na^{k+1}b_g(x, y; t-s)\,dV_g(y)\Big|. \label{eq:J001}
\eeqn
Note that
\beqn &&
\Big|\int_U\, \pd {}{x_l}\Na^{k}b_g(x, y; t-s)\,dV_g(y)\Big|
=\Big|\int_U\, \Big(\pd {}{x_l}+\pd {}{y_l}\Big)\Na^{k}b_g(x, y; t-s)\,dV_g(y)\Big|\no\\
&&+\Big|\int_U\,  \Na^{k}b_g(x, y; t-s)\,\pd {}{y_l}(dV_g(y))\Big|+\Big|\int_{\partial U}\, \Na^{k}b_g(x, y; t-s)\,dS_g(y)\Big|\no\\
&\leq &C(t-s)^{-\frac {m+k}{4}},\label{eq:J002}
\eeqn where we used Lemma \ref{lem:HZ} in the last inequality. Moreover, we have
\beq
\Big|\int_{M\backslash U}\, \Na^{k+1}b_g(x, y; t-s)\,dV_g(y)\Big|
\leq C(t-s)^{-\frac {m+4(k+1)}{4}}\exp\Big(-\dd (t-s)^{-\frac 13} r_0^{\frac 43}\Big)\leq C. \label{eq:J003}
\eeq
Combining (\ref{eq:J001})-(\ref{eq:J003}), we have the inequality (\ref{eq:J004}). The lemma is proved.

\end{proof}

Theorem \ref{theo:A1x} implies the following general existence result by the method of continuity. The readers are referred to Theorem 5.2 of Gilbarg-Trudinger \cite{[GT]} or Lemma 4.3.10 of Metsch \cite{[Me]}.

\begin{theo}\label{theo:A2x} Let $(M, g)$ be a compact Riemannian manifold and  $L$ be an operator defined by
\beq
Lu:=a^{ij}a^{kl}\partial_{ijkl}(u)+\sum_{|\al|\leq 3}A_{\al}\Na^{\al}u,
\eeq where the coefficients satisfy the conditions
\beqn
\La&:=&\max_{|\al|\leq 3}\|A_{\al}\|_{C^{0, \ga}(M)}+\max_{1\leq i, j\leq n}\|a^{ij}\|_{C^{0, \ga}(M)}<+\infty,\label{eq:H031x}\\
\Te&:=&\sup\Big\{\te>0\;\Big|\;a^{ij}(x)\xi_i\xi_j\geq \te|\xi|^2,\quad \forall\; \xi\in \RR^m\Big\}<+\infty.\label{eq:H032x}
\eeqn
For any $f\in C^{0, 0, \ga}(M\times [0, T], g)$, there exists a unique solution $u\in C^{4, 1, \ga}(M\times [0, T], g)$ of
\beqn
     \pd ut+Lu&=&f, \quad  \hbox{on}\quad M\times [0, T], \label{eq:H030}\\
     u(\cdot, 0)&=&0, \quad  \hbox{on}\quad  M. \label{eq:H031}
 \eeqn
\end{theo}
\begin{proof}[Sketch of the proof] For any $s\in [0, 1]$, we define the operator
\beq
L_s=(1-s)L_0+sL_1,
\eeq where $L_0=\Delta_g^2$ and $L_1=L.$ Then the coefficients of $L_s$ satisfies
(\ref{eq:H031x})-(\ref{eq:H032x}) with possibly different constants $\La, \Te>0$ independent of $s$.  Since
$L_s$ is an operator from $C^{4, 1, \al}(M\times [0, T], g)$ to $C^{0, 0, \al}(M\times [0, T], g)$ and   Theorem \ref{theo:Me} (the Schauder estimate) holds, we have
\beq
\|u\|_{C^{4, 1, \al}(M\times [0, T], g)}\leq C\|L_s u\|_{C^{0, 0, \al}(M\times [0, T], g)}.
\eeq
Since $L_0$ is invertible by Theorem \ref{theo:A1x},  $L_s$ is also invertible for any $s\in (0, 1]$ by the method of continuity. The theorem is proved.

\end{proof}

 \subsection{The estimates for solutions to second-order  equations}
In this subsection, we collect some notations and results from Lieberman \cite{[Lieb]}.
\begin{defi}\label{defi:B001}(cf. Page 46-47 of \cite{[Lieb]})
 \begin{enumerate}
   \item[(1).] We denote by $X=(x, t)$ a point in $\RR^{n+1}$. For any two points
   $X_1=(x_1, t_1)$ and $X_2=(x_2, t_2)$ we define
   \beq
   |X_1-X_2|=\max\{|x_1-x_2|, |t_1-t_2|^{\frac 12}\}. \no
   \eeq
   We denote by $Q(X_0, R)$ the cylinder
   \beq
   Q(X_0, R)=Q(R)=\{X\in \RR^{n+1}\;|\;|X-X_0|<R, t<t_0\}.\no
   \eeq We also use the ball $B(x_0, R)=\{x\in \RR^n\,|\,|x-x_0|<R\}.$  Let $\al\in (0, 1), \Om\subset \RR^{n+1}$. we denote by $\cP\Om$ the parabolic boundary, i.e. the set of all points $X_0\in \partial \Om$ such that for any $\ee>0$ the cylinder $Q(X_0, \ee)$ contains points not in $\Om. $
   \item[(2).]
We define
\beqs
d(X_0)&=&\inf\{|X-X_0|\;|\;X\in \cP \Om, t<t_0\},\\
d(X, Y)&=&\min\{d(X), d(Y)\},\\
\left[f\right]_{k+\al}^{(b)}&=&\sup_{X\neq Y}\sum_{ |\bb|+2j=k}d(X, Y)^{a+b}\frac {|D_x^{\bb}D_t^jf(X)-D_x^{\bb}D_t^jf(Y)|}{|X-Y|^{\al}},\\
\left<f\right>_{k+\al}^{(b)}&=&\sup_{X\neq Y, x=y}\sum_{ |\bb|+2j=k-1}d(X, Y)^{a+b}\frac {|D_x^{\bb}D_t^jf(X)-D_x^{\bb}D_t^jf(Y)|}{|X-Y|^{1+\al}},\\
|f|_{k+\al}^{(b)}&=&\sum_{|\bb|+2j\leq k}|D_x^{\bb}D_t^j f|_0^{(|\bb|+2j+b)}+[f]_{k+\al}^{(b)}+\left<f\right>_{k+\al}^{(b)},\\
|f|_{k+\al}^*&=&|f|_{k+\al}^{(0)},\\
H_{k+\al}^{*}(\Om)&=&\{f\;|\;|f|^*_{k+\al}<\infty\},\\
H_{k+\al}^{(b)}(\Om)&=&\{f\;|\;|f|^{(b)}_{k+\al}<\infty\}.
\eeqs
 \end{enumerate}

\end{defi}

\begin{theo}\label{theo:A1}(cf. Theorem 4.9 of \cite{[Lieb]}) Let $\Om$ be a bounded domain in  $  \RR^{n+1}$, and let $
a^{ij}\in H_{\al}^{(0)}, b^i\in H_{\al}^{(1)}, c\in H_{\al}^{(2)}, f\in H_{\al}^{(2)}$ satisfy the following conditions for positive constants $A, B, \la, \La, c_1$
\beqn
\la |\xi|^2&\leq& a^{ij}\xi_i\xi_j\leq \La |\xi|^2, \quad [a^{ij}]_{\al}\leq A,\label{eq:A072}\\
|b^{i}|_{\al}&\leq& B, \quad |c|_{\al}\leq c_1.\label{eq:A074}
\eeqn
If $u\in H_{2+ \al}^*$ is a solution of
\beq
\pd ut=a^{ij}u_{ij}+b^iu_i+cu+f,\quad (x, t)\in \Om,\no
\eeq then we have
\beq
|u|_{{2+ \al}}^*\leq C(A, B, c_1, n, \la, \La)(|u|_0+|f|^{(2)}_{\al}).\no
\eeq

\end{theo}
By choosing a fixed good atlas and using Theorem \ref{theo:A1}, we have
\begin{theo}\label{theo:A2}Let $(M, g)$ be a compact Riemannian manifold without boundary, and $
a^{ij}\in C^{\al}(M\times [0, T], g), b^i\in C^{\al}(M\times [0, T], g), c\in C^{\al}(M\times [0, T], g), f\in C^{\al}(M\times [0, T], g)$ satisfy the conditions (\ref{eq:A072})-(\ref{eq:A074}).
If $u\in C^{2+\al}(M\times [0, T], g)$ is a solution of
\beq
\pd ut=a^{ij}u_{ij}+b^iu_i+cu+f,\quad (x, t)\in \Om,\no
\eeq then for $\tau\in (0, T)$ we have
\beq
|u|_{C^{2+ \al}(M\times [\tau, T], g)}\leq C(A, B, c_1, n, \la, \La, \tau)(|u|_0+|f|_{C^{\al}(M\times [0, T], g)}).\no
\eeq

\end{theo}

Theorem \ref{theo:A2} implies the following higher order estimates.
\begin{theo}\label{theo:A3}(cf. Page 141 of \cite{[GT]}) Let $(M, g)$ be a compact Riemannian manifold without boundary, and $
a^{ij}\in C^{k+\al}(M\times [0, T]), b^i\in C^{k+\al}(M\times [0, T]), c\in C^{k+\al}(M\times [0, T]), f\in C^{k+\al}(M\times [0, T])$ satisfy the following conditions for positive constants $A, B, \la, \La, c_1$
\beqn
\la |\xi|^2&\leq& a^{ij}\xi_i\xi_j\leq \La |\xi|^2, \quad |a^{ij}|_{k+\al}\leq A,\no\\
|b^{i}|_{k+\al}&\leq& B, \quad |c|_{k+\al}\leq c_1.\no
\eeqn
If $u\in C^{k+\al}(M\times [0, T])$ is a solution of
\beq
\pd ut=a^{ij}u_{ij}+b^iu_i+cu+f,\quad (x, t)\in \Om,\no
\eeq then for $\tau\in (0, T)$ we have
\beq
|u|_{C^{k+2+ \al}(M\times [\tau, T], g)}\leq C(A, B, c_1, n, \la, \La, \tau)(|u|_0+|f|_{C^{k+\al}(M\times [0, T], g)}).\no
\eeq

\end{theo}

\begin{theo}\label{theo:A4}(cf. Lemma 19.1 of P. Li \cite{[LiPeter]}, The parabolic Moser iteration) Let $(M, g)$ be a compact Riemannian manifold without boundary. Consider the equation
\beq
\pd ut=a^{ij}u_{ij}+b^iu_i+cu+d,\quad (x, t)\in \Om,\no
\eeq where the coefficients $a^{ij}, b^i, c$ and $d$ satisfy the conditions (\ref{eq:A072}) and (\ref{eq:A074}). Then for $\tau\in (0, T)$ we have
\beq
|u|_{L^{\infty}(M\times [\tau, T], g)}\leq C(A, B, c_1, n, \la, \La, \tau)\|u\|_{L^2(M\times [0, T], g)}.\no
\eeq

\end{theo}

\end{appendices}

\vskip10pt

\noindent Jie He, Beijing University of Chemical Technology; hejie@amss.ac.cn.\\

\noindent Haozhao Li, Institute of Geometry and Physics, and  School of Mathematical Sciences, University of Science and Technology of China, No. 96 Jinzhai Road, Hefei, Anhui Province, 230026, China;  hzli@ustc.edu.cn.\\

\end{document}